\documentclass[a4paper,oneside,11pt]{article}
\usepackage[margin = 3cm]{geometry}
\usepackage{amsmath}
\usepackage{amssymb}
\usepackage{amsthm}
\usepackage{enumerate}
\usepackage{bm}
\usepackage{enumitem}
\usepackage{url}
\usepackage{hyperref}
\newcommand{\N}{\mathbb{N}}
\newcommand{\Z}{\mathbb{Z}}
\newcommand{\R}{\mathbb{R}}
\newcommand{\dH}{\, \mathrm{d} \mathcal{H}^{d-1} \,}
\newcommand{\dd}{\, \mathrm{d} \,}

\newcommand{\dx}{\, \mathrm{dx}\,}

\newcommand{\dt}{\, \mathrm{dt}\,}
\newcommand{\dr}{\, \mathrm{dr}\,}
\newcommand{\pd}{\partial}
\newcommand{\eps}{\varepsilon}
\newcommand{\der}{ \mathrm{D} }
\newcommand{\pdnu}{ \pd_{\bm{n}}}
\newcommand{\Laplace}{\Delta}
\newcommand{\abs}[1]{\left| #1 \right|}
\newcommand{\norm}[1]{\| #1 \|}
\newcommand{\bignorm}[1]{\left\| #1 \right\|}
\newcommand{\inner}[2]{\langle #1 , #2 \rangle}

\newcommand{\mean}[1]{\overline{#1}}
\renewcommand{\div}{\, \mathrm{div}\,}
\newcommand{\dist}{\text{dist}}

\newtheorem{thm}{Theorem}[section]
\newtheorem{lemma}[thm]{Lemma}

\newtheorem{remark}{Remark}[section]

\newtheorem{defn}{Definition}[section]
\newtheorem{assump}{Assumption}[section]

\numberwithin{equation}{section}
\begin{document}

\title{Global weak solutions and asymptotic limits of a Cahn--Hilliard--Darcy system modelling tumour growth}

\author{Harald Garcke \footnotemark[1] \and Kei Fong Lam \footnotemark[1]}

\date{\today}

\maketitle

\renewcommand{\thefootnote}{\fnsymbol{footnote}}
\footnotetext[1]{Fakult\"at f\"ur Mathematik, Universit\"at Regensburg, 93040 Regensburg, Germany
({\tt \{Harald.Garcke, Kei-Fong.Lam\}@mathematik.uni-regensburg.de}).}

\begin{abstract}
We study the existence of weak solutions to a Cahn--Hilliard--Darcy system coupled with a convection-reaction-diffusion equation through the fluxes, through the source terms and in Darcy's law.  The system of equations arises from a mixture model for tumour growth accounting for transport mechanisms such as chemotaxis and active transport.  We prove, via a Galerkin approximation, the existence of global weak solutions in two and three dimensions, along with new regularity results for the velocity field and for the pressure.  Due to the coupling with the Darcy system, the time derivatives have lower regularity compared to systems without Darcy flow, but in the two dimensional case we employ a new regularity result for the velocity to obtain better integrability and temporal regularity for the time derivatives.  Then, we deduce the global existence of weak solutions for two variants of the model; one where the velocity is zero and another where the chemotaxis and active transport mechanisms are absent.
\end{abstract}



\maketitle

\section{Introduction}
In recent years there has been an increased focus on the mathematical modelling and analysis of tumour growth.  Many new models have been proposed and numerical simulations have been carried out to provide new and important insights on cancer research, see for instance \cite{book:Cristini} and \cite[Chap. 3]{book:Fasano}.  In this work we analyse a diffuse interface model proposed in \cite{article:GarckeLamSitkaStyles}, which models a mixture of tumour cells and healthy cells in the presence of an unspecified chemical species acting as a nutrient.  More precisely, for a bounded domain $\Omega \subset \R^{d}$ where the cells reside and $T > 0$, we consider the following set of equations,
\begin{subequations}\label{CHDN}
\begin{alignat}{3}
\div \bm{v} & = \Gamma_{\bm{v}} && \text{ in } \Omega \times (0,T) =: Q, \label{CHDN:div} \\
\bm{v} & = - K (\nabla p  - (\mu + \chi \sigma)\nabla \varphi) &&  \text{ in } Q, \label{CHDN:Darcy} \\
\pd_{t}\varphi  + \div (\varphi \bm{v}) & = \div (m(\varphi) \nabla \mu) + \Gamma_{\varphi} &&  \text{ in } Q, \\
\mu & = A \Psi'(\varphi) - B \Laplace \varphi - \chi \sigma &&  \text{ in } Q, \\
\pd_{t}\sigma + \div (\sigma \bm{v}) & = \div (n(\varphi) (D \nabla \sigma - \chi \nabla \varphi)) - \mathcal{S} &&  \text{ in } Q.
\end{alignat}
\end{subequations}
Here, $\bm{v}$ denotes the volume-averaged velocity of the mixture, $p$ denotes the pressure, $\sigma$ denotes the concentration of the nutrient, $\varphi \in [-1,1]$ denotes the difference in volume fractions, with $\{\varphi = 1\}$ representing the unmixed tumour tissue, and $\{\varphi = -1\}$ representing the surrounding healthy tissue, and $\mu$ denotes the chemical potential for $\varphi$.

The model treats the tumour and healthy cells as inertia-less fluids, leading to the appearance of a Darcy-type subsystem with a source term $\Gamma_{\bm{v}}$.  The order parameter $\varphi$ satisfies a convective Cahn--Hilliard type equation with additional source term $\Gamma_{\varphi}$, and similarly, the nutrient concentration $\sigma$ satisfies a convection-reaction-diffusion equation with a non-standard flux and a source term $\mathcal{S}$.  We refer the reader to \cite[\S 2]{article:GarckeLamSitkaStyles} for the derivation from thermodynamic principles, and to \cite[\S 2.5]{article:GarckeLamSitkaStyles} for a discussion regarding the choices for the source terms $\Gamma_{\varphi}, \Gamma_{\bm{v}}$ and $\mathcal{S}$.

The positive constants $K$ and $D$ denote the permeability of the mixture and the diffusivity of the nutrient, $m(\varphi)$ and $n(\varphi)$ are positive mobilities for $\varphi$ and $\sigma$, respectively.  The parameter $\chi \geq 0$ regulates the chemotaxis effect (see \cite{article:GarckeLamSitkaStyles} for more details), $\Psi(\cdot)$ is a potential with two equal minima at $\pm 1$, $A$ and $B$ denote two positive constants related to the thickness of the diffuse interface and the surface tension.

We supplement the above with the following boundary and initial conditions
\begin{subequations}\label{CHDNbdy}
\begin{alignat}{3}
\pdnu \varphi = \pdnu \mu & = 0 && \text{ on } \pd \Omega \times (0,T) =: \Sigma, \\
\bm{v} \cdot \bm{n} = \pdnu p & = 0 && \text{ on } \Sigma, \label{CHDNbdy:velopress} \\
n(\varphi) D \pdnu \sigma & = b ( \sigma_{\infty} - \sigma) && \text{ on } \Sigma, \\
\varphi(0) = \varphi_{0}, \quad \sigma(0) & = \sigma_{0} && \text{ on } \Omega.
\end{alignat}
\end{subequations}
Here $\varphi_{0}$, $\sigma_{0}$ and $\sigma_{\infty}$ are given functions and $b > 0$ is a constant.  We denote $\pdnu f := \nabla f \cdot \bm{n}$ as the normal derivative of $f$ at the boundary $\pd \Omega$, where $\bm{n}$ is the outer unit normal.  Associated to \eqref{CHDN} is the free energy density $N(\varphi, \sigma)$ for the nutrient, which is defined as
\begin{align}\label{choice:Nvarphisigma}
N(\varphi, \sigma) := \frac{D}{2} \abs{\sigma}^{2} + \chi \sigma (1- \varphi).
\end{align}
Note that
\begin{align*}
N_{,\sigma} := \frac{\pd N}{\pd \sigma} = D \sigma + \chi (1-\varphi), \quad N_{,\varphi} := \frac{\pd N}{\pd \varphi} = -\chi \sigma,
\end{align*}
so that \eqref{CHDN} may also be written as
\begin{subequations}\label{CHDN:version:generalN}
\begin{alignat}{3}
\div \bm{v} & = \Gamma_{\bm{v}}, \\
\bm{v} & = - K (\nabla p  - \mu \nabla \varphi + N_{,\varphi} \nabla \varphi), \\
\pd_{t}\varphi  + \div ( \varphi \bm{v}) & = \div (m(\varphi) \nabla \mu) + \Gamma_{\varphi}, \\
\mu & = A \Psi'(\varphi) - B \Laplace \varphi + N_{,\varphi}, \\
\pd_{t}\sigma + \div( \sigma \bm{v}) & = \div (n(\varphi) \nabla N_{,\sigma}) - \mathcal{S},
\end{alignat}
\end{subequations}
which is the general phase field model proposed in \cite{article:GarckeLamSitkaStyles}.  In this work we do not aim to analyse such a model with a general free energy density $N(\varphi, \sigma)$, but we will focus solely on the choice \eqref{choice:Nvarphisigma} and the corresponding model \eqref{CHDN}-\eqref{CHDNbdy}.

Our goal in this work is to prove the existence of weak solutions (see Definition \ref{defn:weaksoln} below) of \eqref{CHDN}-\eqref{CHDNbdy} in two and three dimensions.  Moreover, one might expect that by setting $\Gamma_{\bm{v}} = 0$ and then sending $b \to 0$ and $K \to 0$, the weak solutions to \eqref{CHDN}-\eqref{CHDNbdy} will converge (in some appropriate sense) to the weak solutions of
\begin{subequations}\label{KzeroModel}
\begin{alignat}{3}
\pd_{t}\varphi  &= \div (m(\varphi) \nabla \mu) + \Gamma_{\varphi} && \text{ in } Q, \\
\mu & = A\Psi'(\varphi) - B \Laplace \varphi - \chi \sigma && \text{ in } Q, \\
\pd_{t}\sigma  & = \div (n(\varphi) (D \nabla \sigma - \chi \nabla \varphi)) - \mathcal{S} && \text{ in } Q, \\
0 & = \pdnu \varphi = \pdnu \mu = \pdnu \sigma  && \text{ on } \Sigma.
\end{alignat}
\end{subequations}
We denote \eqref{KzeroModel} as the limit system of vanishing permeability, where the effects of the volume-averaged velocity are neglected.  By substituting
\begin{align}\label{Hawkins:sourceterm}
\Gamma_{\varphi} = \mathcal{S} = f(\varphi)(D \sigma + \chi(1-\varphi) - \mu)
\end{align}
for some non-negative function $f(\varphi)$ leads to the model derived in \cite{article:HawkinsZeeOden12}.  The specific choices for $\Gamma_{\varphi}$ and $\mathcal{S}$ in \eqref{Hawkins:sourceterm} are motivated by linear phenomenological laws for chemical reactions.  The analysis of \eqref{KzeroModel} with the parameters
\begin{align*}
D = 1, \quad \chi = 0, \quad n(\varphi) = m(\varphi) = 1
\end{align*}
has been the subject of study in \cite{article:ColliGilardiHilhorst15,article:ColliGilardiRoccaSprekelsVV,article:ColliGilardiRoccaSprekelsAA,article:FrigeriGrasselliRocca15}, where well-posedness and long-time behaviour have been established for a large class of functions $\Psi(\varphi)$ and $f(\varphi)$.  Alternatively, one may consider the following choice of source terms
\begin{align}\label{realisticSource}
\Gamma_{\varphi} = h(\varphi)(\lambda_{p} \sigma - \lambda_{a}), \quad \mathcal{S} = \lambda_{c} h(\varphi) \sigma,
\end{align}
where $\lambda_{p}$, $\lambda_{a}$, $\lambda_{c}$ are non-negative constants representing the tumour proliferation rate, the apoptosis rate, and the nutrient consumption rate, respectively, and $h(\varphi)$ is a non-negative interpolation function such that $h(-1) = 0$ and $h(1) = 1$.  The above choices for $\Gamma_{\varphi}$ and $\mathcal{S}$ are motivated from the modelling of processes experienced by a young tumour.

The well-posedness of model \eqref{KzeroModel} with the choice \eqref{realisticSource} has been studied by the authors in \cite{article:GarckeLamNeumann} and \cite{article:GarckeLamDirichlet} with the boundary conditions \eqref{CHDNbdy} (neglecting \eqref{CHDNbdy:velopress}) in the former and for non-zero Dirichlet boundary conditions in the latter.  It has been noted in \cite{article:GarckeLamNeumann} that the well-posedness result with the boundary conditions \eqref{CHDNbdy} requires $\Psi$ to have at most quadratic growth, which is attributed to the presence of the source term $\Gamma_{\varphi} \mu = h(\varphi) \mu (\lambda_{p} \sigma - \lambda_{a})$ when deriving useful a priori estimates.  Meanwhile in \cite{article:GarckeLamDirichlet} the Dirichlet boundary conditions and the application of the Poincar\'{e} inequality allows us to overcome this restriction and allow for $\Psi$ to be a regular potential with polynomial growth of order less than 6, and by a Yosida approximation, the case where $\Psi$ is a singular potential is also covered.

We also mention the work of \cite{OCTumour} that utilises a Schauder's fixed point argument to show existence of weak solutions for $\Psi$ with quartic growth and $\Gamma_{\varphi}, \mathcal{S}$ as in \eqref{realisticSource}.  This is based on first deducing that $\sigma$ is bounded by a comparison principle, leading to $\Gamma_{\varphi} \in L^{\infty}(\Omega)$.  Then, the standard a priori estimates are derived for a Cahn--Hilliard equation with bounded source terms.  The difference between \cite{OCTumour} and \cite{article:GarckeLamNeumann,article:GarckeLamDirichlet} is the absence of the chemotaxis and active transport mechanisms, i.e., $\chi = 0$, so that the comparison principle can be applied to the nutrient equation.  We refer to \cite{Dai} for the application of a similar procedure to a multi-species tumour model with logarithmic potentials.

On the other hand, by sending $b \to 0$ and $\chi \to 0$ in \eqref{CHDN}, we should obtain weak solutions of
\begin{subequations}
\label{ChizeroModel}
\begin{alignat}{3}
\div \bm{v} & = \Gamma_{\bm{v}} && \text{ in } Q, \\
\bm{v} &= -K(\nabla p - \mu \nabla \varphi) && \text{ in } Q, \label{velo:ChizeroModel}\\
\pd_{t} \varphi + \div (\varphi \bm{v} ) & = \div (m(\varphi) \nabla \mu) + \Gamma_{\varphi} && \text{ in } Q, \\
\mu & = A\Psi'(\varphi) - B \Laplace \varphi && \text{ in } Q, \\
\pd_{t}\sigma + \div (\sigma \bm{v}) & = \div (n(\varphi) D \nabla \sigma) - \mathcal{S} && \text{ in } Q, \\
0 & = \pdnu \varphi = \pdnu \mu = \pdnu \sigma = \bm{v} \cdot \bm{n} && \text{ on } \Sigma.
\end{alignat}
\end{subequations}
We denote \eqref{ChizeroModel} as the limit system of vanishing chemotaxis.  If the source terms $\Gamma_{\bm{v}}$ and $\Gamma_{\varphi}$ are independent of $\sigma$, then \eqref{ChizeroModel} consists of an independent Cahn--Hilliard--Darcy system and an equation for $\sigma$ which is advected by the volume-averaged velocity field $\bm{v}$.  In the case where there is no nutrient and source terms, i.e., $\sigma = \Gamma_{\bm{v}} = \Gamma_{\varphi} = 0$, global existence of weak solutions in two and three dimensions has been established in \cite{article:FengWise12} via the convergence of a fully discrete and energy stable implicit finite element scheme.  For the well-posedness and long-time behaviour of strong solutions, we refer to \cite{article:LowengrubTitiZhao13}.  Meanwhile, in the case where $\Gamma_{\bm{v}} = \Gamma_{\varphi}$ is prescribed, global weak existence and local strong well-posedness for \eqref{ChizeroModel} without nutrient is shown in \cite{preprint:JiangWuZheng14}.

We also mention the work of \cite{article:BosiaContiGrasselli14} on the well-posedness and long-time behaviour of a related system also used in tumour growth, known as the Cahn--Hilliard--Brinkman system, where in \eqref{ChizeroModel} without nutrient an additional viscosity term is added to the left-hand side of the velocity equation \eqref{velo:ChizeroModel} and the mass exchange terms $\Gamma_{\bm{v}}$ and $\Gamma_{\varphi}$ are set to zero.  The well-posedness of a nonlocal variant of the Cahn--Hilliard--Brinkman system has been investigated in \cite{article:DellaportaGrasselli16}.  Furthermore, when $K$ is a function depending on $\varphi$, the model \eqref{ChizeroModel} with $\sigma = \Gamma_{\bm{v}} = \Gamma_{\varphi} = 0$ is also referred to as the Hele--Shaw--Cahn--Hilliard model (see \cite{article:LeeLowengrubGoodman01,article:LeeLowengrubGoodman01:Part2}).  In this setting, $K(\varphi)$ represents the reciprocal of the viscosity of the fluid mixture.  We refer to \cite{article:WangZhang} concerning the strong well-posedness globally in time for two dimensions and locally in time for three dimensions when $\Omega$ is the $d$-dimensional torus.  Global well-posedness in three dimensions under additional assumptions and long-time behaviour of solutions to the Hele--Shaw--Cahn--Hilliard model are investigated in \cite{article:WangWu}.

We point out that from the derivation of \eqref{CHDN} in \cite{article:GarckeLamSitkaStyles}, the source terms $\Gamma_{\bm{v}}$ and $\Gamma_{\varphi}$ are connected in the sense that $\Gamma_{\bm{v}}$ is related to sum of the mass exchange terms for the tumour and healthy cells, and $\Gamma_{\varphi}$ is related to the difference between the mass exchange terms.  Thus, if $\Gamma_{\varphi}$ would depend on the primary variables $\varphi$, $\sigma$ or $\mu$, then one expects that $\Gamma_{\bm{v}}$ will also depend on the primary variables.  Here, we are able to prove existence of weak solutions for $\Gamma_{\varphi}$ of the form \eqref{assump:Sourcetermspecificform}, which generalises the choices \eqref{Hawkins:sourceterm} and \eqref{realisticSource}, but in exchange $\Gamma_{\bm{v}}$ has to be considered as a prescribed function.  This is attributed to the presence of the source term $\Gamma_{\bm{v}} \left (\varphi \mu + \frac{D}{2} \abs{\sigma}^{2} \right )$ when deriving useful a priori estimates.  We see that if $\Gamma_{\bm{v}}$ depends on the primary variables, we obtain triplet products which cannot be controlled by the usual regularity of $\varphi$, $\mu$ and $\sigma$ in the absence of a priori estimates.

In this work we attempt to generalise the weak existence results for the models studied in  \cite{article:ColliGilardiHilhorst15, article:FrigeriGrasselliRocca15, article:GarckeLamNeumann, article:GarckeLamDirichlet, preprint:JiangWuZheng14,article:LowengrubTitiZhao13} by proving that the weak solutions of \eqref{CHDN} with $\Gamma_{\bm{v}} = 0$ converge (in some appropriate sense) to the weak solutions of \eqref{KzeroModel} as $b \to 0$ and $K \to 0$, and the weak solutions of \eqref{CHDN} converge to the weak solutions of \eqref{ChizeroModel} as $b \to 0$ and $\chi \to 0$.

This paper is organised as follows.  In Section \ref{sec:mainresults} we state the main assumptions and the main results.  In Section \ref{sec:Galerkin} we introduce a Galerkin procedure and derive some a priori estimates for the Galerkin ansatz in Section \ref{sec:apriori} for the case of three dimensions.  We then pass to the limit in Section \ref{sec:passlimit} to deduce the existence result for three dimensions, while in Section \ref{sec:asymptoticlimit} we investigate the asymptotic behaviour of solutions to \eqref{CHDN} as $K \to 0$ and $\chi \to 0$.  In Section \ref{sec:2D}, we outline the a priori estimates for two dimensions and show that the weak solutions for two dimensions yields better temporal regularity than the weak solutions for three dimensions.  In Section \ref{sec:Discussion} we discuss some of the issues present in the analysis of \eqref{CHDN} using different formulations of Darcy's law and the pressure, and with different boundary conditions for the velocity and the pressure.

\paragraph{Notation.}
For convenience, we will often use the notation $L^{p} := L^{p}(\Omega)$ and $W^{k,p} := W^{k,p}(\Omega)$ for any $p \in [1,\infty]$, $k > 0$ to denote the standard Lebesgue spaces and Sobolev spaces equipped with the norms $\norm{\cdot}_{L^{p}}$ and $\norm{\cdot}_{W^{k,p}}$.  In the case $p = 2$ we use $H^{k} := W^{k,2}$ and the norm $\norm{\cdot}_{H^{k}}$.  For the norms of Bochner spaces, we will use the notation $L^{p}(X) := L^{p}(0,T;X)$ for Banach space $X$ and $p \in [1,\infty]$.  Moreover, the dual space of a Banach space $X$ will be denoted by $X^{*}$, and the duality pairing between $X$ and $X^{*}$ is denoted by $\inner{\cdot}{\cdot}_{X, X^{*}}$.  For $d = 2$ or $3$, let $\mathcal{H}^{d-1}$ denote the $(d-1)$ dimensional Hausdorff measure on $\pd \Omega$, and we denote $\R^{d}$-valued functions and any function spaces consisting of vector-valued/tensor-valued functions in boldface.  We will use the notation $\der \bm{f}$ to denote the weak derivative of the vector function $\bm{f}$.

\paragraph{Useful preliminaries.}
For convenience, we recall the Poincar\'{e} inequality: There exists a positive constant $C_{p}$  depending only on $\Omega$ such that, for all $f \in H^{1}$,
\begin{align}
\bignorm{f - \overline{f}}_{L^{2}} & \leq C_{p} \norm{\nabla f}_{L^{2}}, \label{regular:Poincare}
\end{align}
where $\mean{f} := \frac{1}{\abs{\Omega}} \int_{\Omega} f \dx$ denotes the mean of $f$.  The Gagliardo--Nirenberg interpolation inequality in dimension $d$ is also useful (see \cite[Thm. 10.1, p. 27]{book:Friedman}, \cite[Thm. 2.1]{book:DiBenedetto} and \cite[Thm. 5.8]{book:AdamsFournier}):  Let $\Omega$ be a bounded domain with Lipschitz boundary, and $f \in W^{m,r} \cap L^{q}$, $1 \leq q,r \leq \infty$.  For any integer $j$, $0 \leq j < m$, suppose there is $\alpha \in \R$ such that
\begin{align*}
\frac{1}{p} = \frac{j}{d} + \left ( \frac{1}{r} - \frac{m}{d} \right ) \alpha + \frac{1-\alpha}{q}, \quad \frac{j}{m} \leq \alpha \leq 1.
\end{align*}
Then, there exists a positive constant $C$ depending only on $\Omega$, $m$, $j$, $q$, $r$, and $\alpha$ such that
\begin{align}
\label{GagNirenIneq}
\norm{D^{j} f}_{L^{p}} \leq C \norm{f}_{W^{m,r}}^{\alpha} \norm{f}_{L^{q}}^{1-\alpha} .
\end{align}
We will also use the following Gronwall inequality in integral form (see \cite[Lem. 3.1]{article:GarckeLamNeumann} for a proof): Let $\alpha, \beta, u$ and $v$ be real-valued functions defined on $[0,T]$.  Assume that $\alpha$ is integrable, $\beta$ is non-negative and continuous, $u$ is continuous, $v$ is non-negative and integrable.  If $u$ and $v$ satisfy the integral inequality
\begin{align*}
u(s) + \int_{0}^{s} v(t) \dt  \leq \alpha(s) + \int_{0}^{s} \beta(t) u(t) \dt \quad \text{ for } s \in (0,T],
\end{align*}
then it holds that
\begin{align}
\label{Gronwall}
u(s) + \int_{0}^{s} v(t) \dt \leq \alpha(s) + \int_{0}^{s} \beta(t) \alpha(t) \exp \left ( \int_{0}^{t} \beta(r) \dr \right ) \dt.
\end{align}
To analyse the Darcy system, we introduce the spaces
\begin{align*}
L^{2}_{0} & := \{ f \in L^{2} : \mean{f} = 0 \}, \quad H^{2}_{N} := \{ f \in H^{2} : \pdnu f = 0 \text{ on } \pd \Omega \}, \\
(H^{1})^{*}_{0} & := \{ f \in (H^{1})^{*} : \inner{f}{1}_{H^{1}} = 0 \}.
\end{align*}
Then, the Neumann-Laplacian operator $-\Laplace_{N} : H^{1} \cap L^{2}_{0} \to (H^{1})^{*}_{0}$ is positively defined and self-adjoint.  In particular, by the Lax--Milgram theorem and the Poincar\'{e} inequality \eqref{regular:Poincare} with zero mean, the inverse operator $(-\Laplace_{N})^{-1} : (H^{1})^{*}_{0} \to H^{1} \cap L^{2}_{0}$ is well-defined, and we set $u := (-\Laplace_{N})^{-1}f$ for $f \in (H^{1})^{*}_{0}$ if $\mean{u} = 0$ and
\begin{align*}
-\Laplace u = f \text{ in } \Omega, \quad \pdnu u = 0 \text{ on } \pd \Omega.
\end{align*}

\section{Main results}\label{sec:mainresults}
We make the following assumptions.
\begin{assump}\label{assump:main}
\
\begin{enumerate}[label=$(\mathrm{A \arabic*})$, ref = $(\mathrm{A \arabic*})$]
\item The constants $A$, $B$, $K$, $D$, $\chi$ and $b$ are positive and fixed.
\item \label{assump:mn} The mobilities $m, n$ are continuous on $\R$ and satisfy
\begin{align*}
m_{0} \leq m(t) \leq m_{1}, &\quad n_{0} \leq n(t) \leq n_{1} \quad \forall t \in \R,
\end{align*}
for positive constants $m_{0}$, $m_{1}$, $n_{0}$ and $n_{1}$.
\item $\Gamma_{\varphi}$ and $\mathcal{S}$ are of the form
\begin{equation}\label{assump:Sourcetermspecificform}
\begin{aligned}
\Gamma_{\varphi}(\varphi, \mu, \sigma) & =  \Lambda_{\varphi}(\varphi, \sigma) - \Theta_{\varphi}(\varphi, \sigma) \mu, \\
\mathcal{S}(\varphi, \mu, \sigma) & =  \Lambda_{S}(\varphi ,\sigma) - \Theta_{S}(\varphi, \sigma) \mu,
\end{aligned}
\end{equation}
where $\Theta_{\varphi}, \Theta_{S} : \R^{2} \to \R$ are continuous bounded functions with $\Theta_{\varphi}$ non-negative, and $\Lambda_{\varphi}, \Lambda_{S} : \R^{2} \to \R$ are continuous with linear growth
\begin{equation}\label{assump:ThetaLambda}
\begin{aligned}
\abs{\Theta_{i}(\varphi, \sigma)} \leq R_{0}, &\quad \abs{\Lambda_{i}(\varphi, \sigma)} \leq R_{0}( 1 + \abs{\varphi} + \abs{\sigma}) \quad \text{ for } i \in \{ \varphi, S \},
\end{aligned}
\end{equation}
so that
\begin{align}\label{assump:Sourcetermgrowth}
 \abs{\Gamma_{\varphi}} + \abs{\mathcal{S}} \leq R_{0}(1 + \abs{\varphi} + \abs{\mu} + \abs{\sigma}),
\end{align}
for some positive constant $R_{0}$.
\item \label{assump:Gammabmv:prescribed} $\Gamma_{\bm{v}}$ is a prescribed function belonging to $L^{4}(0,T;L^{2}_{0})$.
\item $\Psi \in C^{2}(\R)$ is a non-negative function satisfying
\begin{align}
\Psi(t) & \geq R_{1} \abs{t}^{2} - R_{2} \quad \forall t \in \R \label{assump:lowerbddPsi}
\end{align}
and either one of the following,
\begin{itemize}
\item[$\mathrm{1.}$] if $\Theta_{\varphi}$ is non-negative and bounded, then
\begin{align}\label{assump:Psiquadratic}
\Psi(t) \leq R_{3}(1 + \abs{t}^{2}), & \quad \abs{\Psi'(t)} \leq R_{4} (1 + \abs{t}), \quad \abs{\Psi''(t)} \leq R_{4};
\end{align}
\item[$\mathrm{2.}$] if $\Theta_{\varphi}$ is positive and bounded, that is,
\begin{align}\label{assump:TheatvarphiPositive}
 R_{0} \geq \Theta_{\varphi}(t,s) \geq R_{5} > 0 \quad \forall t,s \in \R,
\end{align}
then
\begin{align}\label{assump:PsiSuperquadratic}
\abs{\Psi''(t)} \leq R_{6} (1 + \abs{t}^{q}), \; q \in [0,4),
\end{align}
\end{itemize}
for some positive constants $R_{1}, R_{2}, R_{3}, R_{4}, R_{5}, R_{6}$. Furthermore we assume that
\begin{align}\label{assump:An0quadratic}
 A > \frac{2 \chi^{2}}{D R_{1}} .
 \end{align}
\item The initial and boundary data satisfy
\begin{align*}
\sigma_{\infty} \in L^{2}(0,T;L^{2}(\pd \Omega)), \quad \sigma_{0} \in L^{2}, \quad  \varphi_{0} \in H^{1}.
 \end{align*}
\end{enumerate}
\end{assump}

We point out that some of the above assumptions are based on previous works on the well-posedness of Cahn--Hilliard systems for tumour growth.  For instance, \eqref{assump:Psiquadratic} and \eqref{assump:An0quadratic} reflect the situation encountered in \cite{article:GarckeLamNeumann}, where if $\Theta_{\varphi} = 0$, i.e., $\Gamma_{\varphi}$ is independent of $\mu$, then the derivation of the a priori estimate requires a quadratic potential.  But in the case where \eqref{assump:TheatvarphiPositive} is satisfied, we can allow $\Psi$ to be a regular potential with polynomial growth of order less than 6, and by a Yosida approximation, we can extend our existence results to the situation where $\Psi$ is a singular potential, see for instance \cite{article:GarckeLamDirichlet}.  Moreover, the condition \eqref{assump:An0quadratic} is a technical assumption based on the fact that the second term of the nutrient free energy $\chi \sigma (1-\varphi)$ does not have a positive sign.

Meanwhile, the linearity of the source terms $\Gamma_{\varphi}$ and $\mathcal{S}$ with respect to the chemical potential $\mu$ assumed in \eqref{assump:Sourcetermspecificform} is a technical assumption based on the expectation that, at best, we have weak convergence for Galerkin approximation to $\mu$, which is in contrast with $\varphi$ and $\sigma$ where we might expect a.e convergence and strong convergence for the Galerkin approximations.  Moreover, if we consider
\begin{align*}
\Theta_{\varphi}(\varphi, \sigma) = \Theta_{S}(\varphi, \sigma) = f(\varphi), \quad \Lambda_{\varphi}(\varphi, \sigma) = \Lambda_{S}(\varphi, \sigma) = f(\varphi)(D \sigma + \chi(1-\varphi)),
\end{align*}
for a non-negative function $f(\varphi)$, then we obtain the source terms in \cite{article:ColliGilardiHilhorst15, article:FrigeriGrasselliRocca15, article:HawkinsZeeOden12}.

Compared to the set-up in \cite{preprint:JiangWuZheng14}, in \ref{assump:Gammabmv:prescribed} we prescribe a higher temporal regularity for the prescribed source term $\Gamma_{\bm{v}}$.  This is needed when we estimate the source term $\Gamma_{\bm{v}} \frac{D}{2} \abs{\sigma}^{2}$ in the absence of a priori estimates, see Section \ref{sec:SourcetermDarcy} for more details.  The mean zero condition is a consequence of the no-flux boundary condition $\bm{v} \cdot \bm{n} = 0$ on $\pd \Omega$ and the divergence equation \eqref{CHDN:div}.  In particular, we can express the Darcy subsystem \eqref{CHDN:div}-\eqref{CHDN:Darcy} as an elliptic equation for the pressure $p$:
\begin{subequations}\label{PressurePoisson}
\begin{alignat}{3}
-\Laplace p &= \frac{1}{K} \Gamma_{\bm{v}} -\div ((\mu + \chi \sigma) \nabla \varphi) && \text{ in } \Omega, \\
\pdnu p &= 0 && \text{ on } \pd \Omega.
\end{alignat}
\end{subequations}
Solutions to \eqref{PressurePoisson} are uniquely determined up to an arbitrary additive function that may only depend on time, and thus without loss of generality, we impose the condition $\mean{p} = \frac{1}{\abs{\Omega}} \int_{\Omega} p \dx = 0$ to \eqref{PressurePoisson}.  We may then define $p$ as
\begin{align}\label{pressure:nonlocal:defn}
p = (-\Laplace_{N})^{-1} \left ( \frac{1}{K} \Gamma_{\bm{v}} - \div ((\mu + \chi \sigma) \nabla \varphi) \right ),
\end{align}
if $\frac{1}{K} \Gamma_{\bm{v}} - \div ((\mu + \chi \sigma) \nabla \varphi) \in (H^{1})^{*}_{0}$.

\begin{remark}
In the case $\Gamma_{\bm{v}} = 0$, one can also consider the assumption
\begin{align}\label{crosssourcetermnonnegative}
\mathcal{S} N_{,\sigma} - \Gamma_{\varphi} \mu = \mathcal{S} (D \sigma + \chi(1-\varphi)) - \Gamma_{\varphi} \mu \geq 0
\end{align}
instead of \eqref{assump:TheatvarphiPositive}, which holds automatically if $\Gamma_{\varphi}$ and $\mathcal{S}$ are chosen to be of the form \eqref{Hawkins:sourceterm}.  In fact this property is used in \cite{article:ColliGilardiHilhorst15, article:FrigeriGrasselliRocca15}.
\end{remark}

We make the following definition.
\begin{defn}[Weak solutions for 3D]\label{defn:weaksoln}
We call a quintuple $(\varphi, \mu, \sigma, \bm{v}, p)$ a weak solution to \eqref{CHDN}-\eqref{CHDNbdy} if
\begin{align*}
\varphi & \in L^{\infty}(0,T;H^{1}) \cap L^{2}(0,T;H^{3}) \cap W^{1,\frac{8}{5}}(0,T;(H^{1})^{*}), \\
\sigma & \in L^{\infty}(0,T;L^{2}) \cap L^{2}(0,T;H^{1}) \cap W^{1,\frac{5}{4}}(0,T;(W^{1,5})^{*}), \\
\mu  & \in L^{2}(0,T;H^{1}), \quad p \in L^{\frac{8}{5}}(0,T;H^{1} \cap L^{2}_{0}), \quad \bm{v} \in L^{2}(0,T;\bm{L}^{2}),
\end{align*}
such that $\varphi(0) = \varphi_{0}$,
\begin{align*}
\inner{\sigma_{0}}{\zeta}_{H^{1}, (H^{1})^{*}} = \inner{\sigma(0)}{\zeta}_{H^{1},(H^{1})^{*}} \quad \forall \zeta \in H^{1},
\end{align*}
and
\begin{subequations}\label{CHDN:weak}
\begin{align}
\inner{\pd_{t}\varphi}{\zeta}_{H^{1}, (H^{1})^{*}} & = \int_{\Omega} -m(\varphi) \nabla \mu \cdot \nabla \zeta + \Gamma_{\varphi} \zeta + \varphi \bm{v} \cdot \nabla \zeta \dx, \label{CHDN:varphi} \\
\int_{\Omega} \mu \zeta \dx & = \int_{\Omega} A \Psi'(\varphi) \zeta + B \nabla \varphi \cdot \nabla \zeta - \chi \sigma \zeta \dx, \label{CHDN:mu} \\
\inner{\pd_{t}\sigma}{\phi}_{W^{1,5}, (W^{1,5})^{*}} & = \int_{\Omega} -n(\varphi) (D \nabla \sigma - \chi \nabla \varphi) \cdot \nabla \phi - \mathcal{S} \phi + \sigma \bm{v} \cdot \nabla \phi \dx  \label{CHDN:sigma} \\
\notag & + \int_{\pd \Omega} b (\sigma_{\infty} - \sigma) \phi \dH, \\
\int_{\Omega}  \nabla p \cdot \nabla \zeta \dx &= \int_{\Omega} \frac{1}{K} \Gamma_{\bm{v}} \zeta + (\mu + \chi \sigma) \nabla \varphi \cdot \nabla \zeta \dx, \label{CHDN:pressure} \\
\int_{\Omega} \bm{v} \cdot \bm{\zeta} \dx & = \int_{\Omega} - K (\nabla p - (\mu + \chi \sigma) \nabla \varphi) \cdot \bm{\zeta} \dx,\label{CHDN:velo}
\end{align}
\end{subequations}
for a.e. $t \in (0,T)$ and for all $\zeta \in H^{1}$, $\phi \in W^{1,5}$, and $\bm{\zeta} \in \bm{L}^{2}$.
\end{defn}

Neglecting the nutrient $\sigma$, we observe that our choice of function spaces for $(\varphi, \mu, p, \bm{v})$ coincide with those in \cite[Defn. 2.1(i)]{preprint:JiangWuZheng14}.  In contrast to the usual $L^{2}(0,T;(H^{1})^{*})$-regularity (see \cite{article:ColliGilardiHilhorst15, article:FrigeriGrasselliRocca15}) we obtain a less regular time derivative $\pd_{t}\varphi$.  The drop in the time regularity from $2$ to $\frac{8}{5}$ is attributed to the convection term $\div (\varphi \bm{v})$ belonging to $L^{\frac{8}{5}}(0,T;(H^{1})^{*})$.  The same is true for the regularity for the time derivative $\pd_{t}\sigma$ in $L^{\frac{5}{4}}(0,T;(W^{1,5})^{*})$ as the convection term $\div(\sigma \bm{v})$ lies in the same space.  We refer the reader to the end of Section \ref{sec:convectionandtimederivative} for a calculation motivating the choice of function spaces for $\div (\sigma \bm{v})$ and $\pd_{t} \sigma$.  Furthermore, the embedding of $L^{\infty}(0,T;H^{1}) \cap W^{1,\frac{8}{5}}(0,T;(H^{1})^{*})$ into $C^{0}([0,T];L^{2})$ from \cite[\S 8, Cor. 4]{article:Simon86} guarantees that the initial condition for $\varphi$ is meaningful. However, for $\sigma$ we have the embedding $L^{\infty}(0,T;L^{2}) \cap W^{1,\frac{5}{4}}(0,T;(W^{1,5})^{*}) \subset \subset C^{0}([0,T];(H^{1})^{*})$, and so $\sigma(0)$ makes sense as a function in $(H^{1})^{*}$.  Thus, the initial condition $\sigma_{0}$ is attained as an equality in $(H^{1})^{*}$.  We now state the existence result for \eqref{CHDN}-\eqref{CHDNbdy}.
\begin{thm}[Existence of weak solutions in 3D and energy inequality]\label{thm:exist}
Let $\Omega \subset \R^{3}$ be a bounded domain with $C^{3}$-boundary $\pd \Omega$.  Suppose Assumption \ref{assump:main} is satisfied.  Then, there exists a weak solution quintuple $(\varphi, \mu, \sigma, \bm{v}, p)$ to \eqref{CHDN}-\eqref{CHDNbdy} in the sense of Definition \ref{defn:weaksoln} with
\begin{align}\label{pressure:velo:regularity}
p \in L^{\frac{8}{7}}(0,T;H^{2}) , \quad \bm{v} \in L^{\frac{8}{7}}(0,T;\bm{H}^{1}),
\end{align}
and in addition satisfies
\begin{equation}
\label{weaksoln:energy:ineq}
\begin{aligned}
& \norm{\varphi}_{L^{\infty}(H^{1}) \cap L^{2}(H^{3})\cap W^{1,\frac{8}{5}}((H^{1})^{*})}  + \norm{\sigma}_{L^{\infty}(L^{2}) \cap W^{1,\frac{5}{4}}((W^{1,5})^{*})) \cap L^{2}(H^{1})} \\
& \quad + \norm{\mu}_{L^{2}(H^{1})} + b^{\frac{1}{2}} \norm{\sigma}_{L^{2}(L^{2}(\pd \Omega))} + \norm{p}_{L^{\frac{8}{5}}(H^{1}) \cap L^{\frac{8}{7}}(H^{2})}  \\
& \quad + K^{-\frac{1}{2}} \left ( \norm{\bm{v}}_{L^{2}(\bm{L}^{2}) \cap L^{\frac{8}{7}}(\bm{H}^{1})} +  \norm{\div (\varphi \bm{v})}_{L^{\frac{8}{5}}((H^{1})^{*})} + \norm{\div (\sigma \bm{v})}_{L^{\frac{5}{4}}((W^{1,5})^{*})} \right ) \\
& \quad \leq C,
\end{aligned}
\end{equation}
where the constant $C$  does not depend on $(\varphi, \mu, \sigma, \bm{v}, p)$ and is uniformly bounded for $b, \chi \in (0,1]$ and is also uniformly bounded for $K \in (0,1]$ when $\Gamma_{\bm{v}} = 0$.
\end{thm}
The regularity result \eqref{pressure:velo:regularity} is new compared to estimates for weak solutions in \cite{preprint:JiangWuZheng14}, which arises from a deeper study of the Darcy subsystem, and can be obtained even in the absence of the nutrient.  We mention that higher regularity estimates for the pressure $p$ in $L^{2}(0,T;H^{2})$ and the velocity $\bm{v}$ in $L^{2}(0,T;\bm{H}^{1})$ are also established in \cite{preprint:JiangWuZheng14}, but these are for strong solutions local in time in three dimensions and global in time for two dimensions.

We now investigate the situation in two dimensions, where the Sobolev embeddings in two dimensions yields better integrability exponents.

\begin{thm}[Existence of weak solutions in 2D]\label{thm:exist2D}
Let $\Omega \subset \R^{2}$ be a bounded domain with $C^{3}$-boundary $\pd \Omega$.  Suppose Assumption \ref{assump:main} is satisfied.  Then, there exists a quintuple $(\varphi, \mu, \sigma, \bm{v}, p)$ to \eqref{CHDN}-\eqref{CHDNbdy} with the following regularity
\begin{align*}
\varphi & \in L^{\infty}(0,T;H^{1}) \cap L^{2}(0,T;H^{3}) \cap W^{1,w}(0,T;(H^{1})^{*}), \quad \mu \in L^{2}(0,T;H^{1}), \\
\sigma & \in L^{2}(0,T;H^{1}) \cap L^{\infty}(0,T;L^{2}) \cap W^{1,r}(0,T;(H^{1})^{*}), \\
p & \in L^{k}(0,T;H^{1} \cap L^{2}_{0}) \cap L^{q}(0,T;H^{2}), \quad \bm{v} \in L^{2}(0,T;\bm{L}^{2}) \cap L^{q}(0,T;\bm{H}^{1}),
\end{align*}
for
\begin{align*}
1 \leq k < 2, \quad 1 \leq q < \frac{4}{3}, \quad 1 < r < \frac{8}{7}, \quad \frac{4}{3} \leq w < 2,
\end{align*}
such that \eqref{CHDN:varphi}, \eqref{CHDN:mu}, \eqref{CHDN:pressure}, \eqref{CHDN:velo} and
\begin{align*}
\inner{\pd_{t}\sigma}{\zeta}_{H^{1}, (H^{1})^{*}} = \int_{\Omega} -n(\varphi) (D \nabla \sigma - \chi \nabla \varphi) \cdot \nabla \zeta - \mathcal{S} \zeta + \sigma \bm{v} \cdot \nabla \zeta \dx + \int_{\pd \Omega} b (\sigma_{\infty} - \sigma) \zeta \dH
\end{align*}
are satisfied for a.e. $t \in (0,T)$, for all $\zeta \in H^{1}$, and all $\bm{\zeta} \in \bm{L}^{2}$.  Furthermore, the initial conditions $\varphi(0) = \varphi_{0}$ and $\sigma(0) = \sigma_{0}$ are attained as in Definition \ref{defn:weaksoln}, and an analogous inequality to \eqref{weaksoln:energy:ineq} also holds.
\end{thm}
The proof of Theorem \ref{thm:exist2D} is similar to that of Theorem \ref{thm:exist}, and hence the details are omitted.  In Section \ref{sec:2D} we will only present the derivation of a priori estimates.  It is due to the better exponents for embeddings in two dimensions and the regularity result for the velocity that we obtain better regularities for the time derivatives $\pd_{t}\varphi$ and $\pd_{t}\sigma$, namely $\pd_{t} \sigma(t)$ belongs to the dual space $(H^{1})^{*}$ for a.e. $t \in (0,T)$.  Furthermore, as mentioned in Remark \ref{rem:2D:W14star} below, if we only have $\bm{v} \in L^{2}(0,T;\bm{L}^{2})$, then the convection term $\div (\sigma \bm{v})$ and the time derivative $\pd_{t}\sigma$ would only belong to the dual space $L^{\frac{4}{3}}(0,T;(W^{1,4})^{*})$.  However, even with the improved temporal regularity, as $\pd_{t}\sigma \notin L^{2}(0,T;(H^{1})^{*})$, we do not have a continuous embedding into the space $C^{0}([0,T];L^{2})$ and so $\sigma(0)$ may not be well-defined as an element of $L^{2}$.

We now state the two asymptotic limits of \eqref{CHDN} for three dimensions, and note that analogous asymptotic limits also hold for two dimensions.
\begin{thm}[Limit of vanishing permeability]
For $b, K \in (0,1]$, we denote a weak solution to \eqref{CHDN}-\eqref{CHDNbdy} with $\Gamma_{\bm{v}} = 0$ and initial conditions $(\varphi_{0}, \sigma_{0})$ by $(\varphi^{K}, \mu^{K}, \sigma^{K}, \bm{v}^{K}, p^{K})$.  Then, as $b \to 0$ and $K \to 0$, it holds that
\begin{subequations}
\begin{alignat*}{3}
\varphi^{K} & \rightarrow \varphi && \quad \text{ weakly-}* && \quad \text{ in } L^{\infty}(0,T;H^{1}) \cap L^{2}(0,T;H^{3}) \cap W^{1,\frac{8}{5}}(0,T;(H^{1})^{*}), \\
\sigma^{K} & \rightarrow \sigma && \quad \text{ weakly-}* && \quad \text{ in } L^{2}(0,T;H^{1}) \cap L^{\infty}(0,T;L^{2}) \cap W^{1,\frac{5}{4}}(0,T;(W^{1,5})^{*}), \\
\mu^{K} & \rightarrow \mu && \quad \text{ weakly } && \quad \text{ in } L^{2}(0,T;H^{1}), \\
p^{K} & \rightarrow p && \quad \text{ weakly } && \quad \text{ in } L^{\frac{8}{5}}(0,T;H^{1}) \cap L^{\frac{8}{7}}(0,T;H^{2}), \\
\bm{v}^{K} & \rightarrow \bm{0} && \quad \text{ strongly } && \quad \text{ in } L^{2}(0,T;\bm{L}^{2}) \cap L^{\frac{8}{7}}(0,T;\bm{H}^{1}),
\end{alignat*}
\end{subequations}
where $(\varphi, \mu, \sigma, p)$ satisfies
\begin{subequations}
\begin{align}
\inner{\pd_{t}\varphi}{\zeta}_{H^{1}, (H^{1})^{*}} & = \int_{\Omega} -m(\varphi) \nabla \mu \cdot \nabla \zeta + \Gamma_{\varphi}(\varphi, \mu, \sigma) \zeta \dx,  \\
\int_{\Omega} \mu \zeta \dx & = \int_{\Omega} A \Psi'(\varphi) \zeta + B \nabla \varphi \cdot \nabla \zeta - \chi \sigma \zeta \dx,  \\
\inner{\pd_{t}\sigma}{\phi}_{W^{1,5}, (W^{1,5})^{*}} & = \int_{\Omega} -n(\varphi) (D \nabla \sigma - \chi \nabla \varphi) \cdot \nabla \phi - \mathcal{S}(\varphi, \mu, \sigma) \phi \dx,   \\
\int_{\Omega}  \nabla p \cdot \nabla \zeta \dx &= \int_{\Omega} (\mu + \chi \sigma) \nabla \varphi \cdot \nabla \zeta \dx,
\end{align}
\end{subequations}
for all $\zeta \in H^{1}$, $\phi \in W^{1,5}$ and a.e. $t \in (0,T)$.  A posteriori, it holds that
\begin{align*}
\pd_{t}\varphi, \pd_{t} \sigma \in L^{2}(0,T;(H^{1})^{*}),
\end{align*}
and thus $\varphi(0) = \varphi_{0}$ and $\sigma(0) = \sigma_{0}$.
\end{thm}

\begin{thm}[Limit of vanishing chemotaxis]
For $b,\chi \in (0,1]$, we denote a weak solution to \eqref{CHDN}-\eqref{CHDNbdy} with corresponding initial conditions $(\varphi_{0}, \sigma_{0})$ by $(\varphi^{\chi}, \mu^{\chi}, \sigma^{\chi}, \bm{v}^{\chi}, p^{\chi})$.  Then, as $b \to 0$ and $\chi \to 0$, it holds that
\begin{subequations}
\begin{alignat*}{3}
\varphi^{\chi} & \rightarrow \varphi && \quad \text{ weakly-}* && \quad \text{ in } L^{\infty}(0,T;H^{1}) \cap L^{2}(0,T;H^{3}) \cap W^{1,\frac{8}{5}}(0,T;(H^{1})^{*}), \\
\sigma^{\chi} & \rightarrow \sigma && \quad \text{ weakly-}* && \quad \text{ in } L^{2}(0,T;H^{1}) \cap L^{\infty}(0,T;L^{2}) \cap W^{1,\frac{5}{4}}(0,T;(W^{1,5})^{*}), \\
\mu^{\chi} & \rightarrow \mu && \quad \text{ weakly } && \quad \text{ in } L^{2}(0,T;H^{1}), \\
p^{\chi} & \rightarrow p && \quad \text{ weakly } && \quad \text{ in } L^{\frac{8}{5}}(0,T;H^{1}) \cap L^{\frac{8}{7}}(0,T;H^{2}), \\
\bm{v}^{\chi} & \rightarrow \bm{v} && \quad \text{ weakly } && \quad \text{ in } L^{2}(0,T;\bm{L}^{2}) \cap L^{\frac{8}{7}}(0,T;\bm{H}^{1}),
\end{alignat*}
\end{subequations}
and
\begin{subequations}
\begin{alignat*}{3}
\div(\varphi^{\chi} \bm{v}^{\chi}) & \rightarrow \div (\varphi \bm{v}) && \quad \text{ weakly } && \quad \text{ in } L^{\frac{8}{5}}(0,T;(H^{1})^{*}), \\
\div(\sigma^{\chi} \bm{v}^{\chi})& \rightarrow \div (\sigma \bm{v}) && \quad \text{ weakly } && \quad \text{ in } L^{\frac{5}{4}}(0,T;(W^{1,5})^{*}),
\end{alignat*}
\end{subequations}
where ($\varphi, \mu, \sigma, \bm{v}, p)$ satisfies
\begin{subequations}
\begin{align}
\inner{\pd_{t}\varphi}{\zeta}_{H^{1}, (H^{1})^{*}} & = \int_{\Omega} -m(\varphi) \nabla \mu \cdot \nabla \zeta + \Gamma_{\varphi}(\varphi, \mu, \sigma) \zeta + \varphi \bm{v} \cdot \nabla \zeta \dx,  \\
\int_{\Omega} \mu \zeta \dx & = \int_{\Omega} A \Psi'(\varphi) \zeta + B \nabla \varphi \cdot \nabla \zeta \dx,  \\
\inner{\pd_{t}\sigma}{\phi}_{W^{1,5}, (W^{1,5})^{*}} & = \int_{\Omega} -n(\varphi) D \nabla \sigma \cdot \nabla \phi - \mathcal{S}(\varphi, \mu, \sigma) \phi + \sigma \bm{v} \cdot \nabla \phi \dx , \\
\int_{\Omega}  \nabla p \cdot \nabla \zeta \dx &= \int_{\Omega} \frac{1}{K} \Gamma_{\bm{v}} \zeta + \mu \nabla \varphi \cdot \nabla \zeta \dx,  \\
\int_{\Omega} \bm{v} \cdot \bm{\zeta} \dx & = \int_{\Omega} - K (\nabla p - \mu \nabla \varphi) \cdot \bm{\zeta} \dx,
\end{align}
\end{subequations}
for all $\zeta \in H^{1}$, $\phi \in W^{1,5}$, $\bm{\zeta} \in \bm{L}^{2}$ and a.e. $t \in (0,T)$, with the attainment of initial conditions as in Definition \ref{defn:weaksoln}.
\end{thm}

\section{Galerkin approximation}\label{sec:Galerkin}
We will employ a Galerkin approximation similar to the one used in \cite{preprint:JiangWuZheng14}.  For the approximation, we use the eigenfunctions of the Neumann--Laplacian operator $\{w_{i}\}_{i \in \N}$.  Recall that the inverse Neumann--Laplacian operator  $\mathcal{L} := (-\Laplace_{N})^{-1} \vert_{L^{2}_{0}} : L^{2}_{0} \to L^{2}_{0}$ is compact, positive and symmetric.  Indeed, let $f, g \in L^{2}_{0}$ with $z = \mathcal{L}f$, $y = \mathcal{L}g$.  Then,
\begin{align*}
(\mathcal{L}f,f)_{L^{2}} = \int_{\Omega} z f \dx = \int_{\Omega} \abs{\nabla z}^{2} \dx \geq 0, \quad (\mathcal{L}f,g)_{L^{2}} = \int_{\Omega} \nabla z \cdot \nabla y \dx = (f, \mathcal{L}g)_{L^{2}}.
\end{align*}
Furthermore, let $\{f_{n}\}_{n \in \N} \subset L^{2}_{0}$ denote a sequence with corresponding solution sequence $\{z_{n} = \mathcal{L} f_{n} \}_{n \in \N} \subset H^{1} \cap L^{2}_{0}$.  By elliptic regularity theory, we have that $z_{n} \in H^{2}_{N}$ for all $n \in \N$.  Then, by reflexive weak compactness theorem and Rellich--Kondrachov theorem, there exists a subsequence such that $z_{n_{j}} \to z \in H^{1} \cap L^{2}_{0}$ as $j \to \infty$.

Thus, by the spectral theorem, the operator $\mathcal{L}$ admits a countable set of eigenfunctions $\{v_{n}\}_{n \in \N}$ that forms a complete orthonormal system in $L^{2}_{0}$.  The eigenfunctions of the Neumann--Laplacian operator is then given by $w_{1} = 1$, $w_{i} = v_{i-1}$ for $i \geq 2$, and $\{w_{i}\}_{i \in \N}$ is a basis of $L^{2}$.

Elliptic regularity theory gives that $w_{i} \in H^{2}_{N}$ and for every $g \in H^{2}_{N}$, we obtain for $g_{k} := \sum_{i=1}^{k} (g, w_{i})_{L^{2}} w_{i}$ that
\begin{align*}
\Laplace g_{k} = \sum_{i=1}^{k} (g, w_{i})_{L^{2}} \Laplace w_{i} = \sum_{i=1}^{k} (g, \lambda_{i} w_{i})_{L^{2}} w_{i} = \sum_{i=1}^{k} (g, \Laplace w_{i})_{L^{2}} w_{i} = \sum_{i=1}^{k} (\Laplace g, w_{i})_{L^{2}} w_{i},
\end{align*}
where $\lambda_{i}$ is the corresponding eigenvalue to $w_{i}$.  This shows that $\Laplace g_{k}$ converges strongly to $\Laplace g$ in $L^{2}$.  Making use of elliptic regularity theory again gives that $g_{k}$ converges strongly to $g$ in $H^{2}_{N}$.  Thus the eigenfunction $\{w_{i}\}_{i \in \N}$ of the Neumann--Laplace operator forms an orthonormal basis of $L^{2}$ and is also a basis of $H^{2}_{N}$.

Later in Section \ref{sec:passlimit}, we will need to use the property that $H^{2}_{N}$ is dense in $H^{1}$ and $W^{1,5}$.  We now sketch the argument for the denseness of $H^{2}_{N}$ in $W^{1,5}$ and the argument for $H^{1}$ follows in a similar fashion.

\begin{lemma}
$H^{2}_{N}$ is dense in $W^{1,5}$.
\end{lemma}

\begin{proof}
Take $g \in W^{1,5}$, as $\Omega$ has a $C^{3}$-boundary, by standard results \cite[Thm. 3, \S 5.3.3]{book:Evans} there exists a sequence $g_{n} \in C^{\infty}(\overline{\Omega})$ such that $g_{n} \to g$ strongly in $W^{1,5}$.  Let $\eps > 0$ be fixed, and define $D_{\eps} := \{ x \in \Omega : \text{ dist}(x, \pd \Omega) \leq \eps \}$.  Let $\zeta_{\eps} \in C^{\infty}_{c}(\Omega)$ be a smooth cut-off function such that $\zeta_{\eps} = 1$ in $\Omega \setminus \overline{D_{\eps}}$ and $\zeta_{\eps} = 0$ in $\overline{D_{\frac{\eps}{2}}}$.

As $g_{n} \in C^{\infty}(\overline{\Omega})$, its trace on $\pd \Omega$ is well-defined.  Choosing $\eps$ sufficiently small allows us to use a classical result from differential geometry about tubular neighbourhoods, i.e., for any $z \in \mathrm{Tub}_{\eps}(\pd \Omega) := \{x \in \R^{d} : \abs{\dist(z, \pd \Omega)} \leq \eps \}$ there exists a unique $y \in \pd \Omega$ such that
\begin{align*}
z = y + \dist(z, \pd \Omega) \bm{n}(y),
\end{align*}
where $\bm{n}$ is the outer unit normal of $\pd \Omega$.  We consider a bounded smooth function $f_{n,\eps}: \R^{d} \to \R$ such that
\begin{align*}
f_{n,\eps}(z) = g_{n}(y) \text{ for all } z \in \mathrm{Tub}_{\eps}(\pd \Omega) \text{ satisfying } z = y + \dist(z, \pd \Omega) \bm{n}(y).
\end{align*}
We now define the smooth function $G_{n,\eps}$ as
\begin{align*}
G_{n,\eps}(x) := \zeta_{\eps}(x) g_{n}(x) + (1-\zeta_{\eps}(x)) f_{n,\eps}(x).
\end{align*}
By construction, the values of the function $f_{n,\eps}$ in $D_{\eps} \subset \mathrm{Tub}_{\eps}(\pd \Omega)$ are constant in the normal direction, so $\nabla G_{n,\eps} \cdot \bm{n} = 0$ on $\pd \Omega$ and thus $G_{n,\eps} \in H^{2}_{N}$.  Furthermore, we compute that
\begin{align*}
\norm{G_{n,\eps} - g_{n}}_{L^{5}} & = \norm{(1-\zeta_{\eps}) (f_{n,\eps} - g_{n})}_{L^{5}(D_{\eps})}, \\
\norm{\nabla (G_{n,\eps} - g_{n})}_{L^{5}} & = \norm{(g_{n} - f_{n,\eps}) \nabla \zeta_{\eps}  + (1-\zeta_{\eps}) \nabla g_{n} + (1-\zeta_{\eps}) \nabla f_{n,\eps}}_{L^{5}}.
\end{align*}
Using that $g_{n}, f_{n,\eps}$ are smooth functions on $\overline{\Omega}$ and that the Lebesgue measure of $D_{\eps}$ tends to zero as $\eps \to 0$ we have the strong convergence of $G_{n,\eps}$ to $g_{n}$ in $L^{5}$.  For the difference in the gradients, we use that $\zeta_{\eps} \to 1$ a.e. in $\Omega$, Lebesgue's dominated convergence theorem and the boundedness of $\nabla g_{n}$ and $\nabla f_{n,\eps}$ to deduce that
\begin{align*}
\norm{(1-\zeta_{\eps}) \nabla g_{n}}_{L^{5}} + \norm{(1-\zeta_{\eps}) \nabla f_{n,\eps}}_{L^{5}} \to 0 \text{ as } \eps \to 0.
\end{align*}
For the remaining term $\norm{ (g_{n} - f_{n,\eps}) \nabla \zeta_{\eps} }_{L^{5}}$ we use
that the support of $\nabla \zeta_{\eps}$ lies in $D_{\eps} \setminus \overline{D_{\frac{\eps}{2}}}$ and for any $z \in D_{\eps} \setminus \overline{D_{\frac{\eps}{2}}}$,
\begin{align*}
& \abs{f_{n,\eps}(z) - g_{n}(z)} = \abs{g_{n}(y) - g_{n}(y + \text{dist}(z, \pd \Omega) \bm{n}(y))} \\
& \quad \leq \int_{0}^{\text{dist}(z, \pd \Omega)} \abs{\nabla g_{n}(y + \xi \bm{n}(y))} \dd \xi \leq \norm{\nabla g_{n}}_{L^{\infty}} \text{\dist}(z, \pd \Omega) \leq C \eps.
\end{align*}
That is, $f_{n,\eps}$ converges uniformly to $g_{n}$ in $D_{\eps} \setminus \overline{D_{\frac{\eps}{2}}}$.  Furthermore, using $\norm{\nabla \zeta_{\eps}}_{L^{\infty}} \leq \frac{C}{\eps}$ in $D_{\eps} \setminus \overline{D_{\frac{\eps}{2}}}$ and $\abs{D_{\eps} \setminus \overline{D_{\frac{\eps}{2}}}} \leq C \eps$ we obtain $\norm{(g_{n} - f_{\eps,n}) \nabla \zeta_{\eps}}_{L^{5}} \leq C \eps^{\frac{1}{5}} \to 0$ as $\eps \to 0$.  This shows that $G_{n,\eps}$ converges strongly to $g_{n}$ in $W^{1,5}$.

\end{proof}

We denote
\begin{align*}
W_{k} := \mathrm{span} \{ w_{1}, \dots, w_{k} \}
\end{align*}
as the finite dimensional space spanned by the first $k$ basis functions and consider
\begin{subequations}\label{Galerkin:ansatz}
\begin{align}
\varphi_{k}(t,x) =\sum_{i=1}^{k} \alpha_{i}^{k}(t) w_{i}(x), \; \mu_{k}(t,x) = \sum_{i=1}^{k} \beta_{i}^{k}(t) w_{i}(x), \; \sigma_{k}(t,x) = \sum_{i=1}^{k} \gamma_{i}^{k}(t) w_{i}(x),
\end{align}
\end{subequations}
and the following Galerkin ansatz:  For $1 \leq j \leq k$,
\begin{subequations}\label{Galerkin:system}
\begin{align}\displaybreak[3]
\int_{\Omega} \pd_{t} \varphi_{k} w_{j} \dx & = \int_{\Omega} - m(\varphi_{k}) \nabla \mu_{k} \cdot \nabla w_{j} + \Gamma_{\varphi, k} w_{j} + \varphi_{k} \bm{v}_{k} \cdot \nabla w_{j} dx, \label{Galerkin:varphi} \\
\int_{\Omega} \mu_{k} w_{j} \dx & = \int_{\Omega} A \Psi'(\varphi_{k}) w_{j} + B \nabla \varphi_{k} \cdot \nabla w_{j} - \chi \sigma_{k} w_{j} \dx,  \label{Galerkin:mu} \\
\int_{\Omega} \pd_{t} \sigma_{k} w_{j} \dx & = \int_{\Omega} -n(\varphi_{k}) (D \nabla \sigma_{k} - \chi \nabla \varphi_{k}) \cdot \nabla w_{j} - \mathcal{S}_{k} w_{j} + \sigma_{k} \bm{v}_{k} \cdot \nabla w_{j} \dx  \label{Galerkin:sigma} \\
\notag & + \int_{\pd \Omega} b ( \sigma_{\infty} - \sigma_{k}) w_{j} \dH,
\end{align}
\end{subequations}
where we define the Galerkin ansatz for the pressure $p_{k}$ and the velocity field $\bm{v}_{k}$ by
\begin{align}
p_{k} & = (-\Laplace_{N})^{-1} \left ( \frac{1}{K} \Gamma_{\bm{v}} -\div ((\mu_{k} + \chi \sigma_{k}) \nabla \varphi_{k}) \right ), \label{Galerkin:pressure} \\
\bm{v}_{k} & = - K (\nabla p_{k} - (\mu_{k} + \chi \sigma_{k}) \nabla \varphi_{k}),\label{Galerkin:velocity}
\end{align}
and we set
\begin{align*}
\Gamma_{\varphi, k} := \Gamma_{\varphi}(\varphi_{k}, \mu_{k}, \sigma_{k}), \quad \mathcal{S}_{k} := \mathcal{S}(\varphi_{k}, \mu_{k}, \sigma_{k}).
\end{align*}
Note that in \eqref{Galerkin:pressure}, the properties $\Gamma_{\bm{v}} \in L^{2}_{0}$ and $\nabla \varphi_{k} \cdot \bm{n} = 0$ on $\pd \Omega$ show that the term inside the bracket belongs to $L^{2}_{0}$ and hence $p_{k}$ is well-defined.  Let $\bm{M}$ and $\bm{S}$ denote the following mass and stiffness matrices, respectively:  For $1 \leq i,j \leq k$,
\begin{align*}
\bm{M}_{ij} = \int_{\Omega} w_{i} w_{j} \dx, \quad \bm{S}_{ij} := \int_{\Omega} \nabla w_{i} \cdot \nabla w_{j} \dx.
\end{align*}
Thanks to the orthonormality of $\{w_{i}\}_{i \in \N}$ in $L^{2}$, we see that $\bm{M}$ is the identity matrix.  It is convenient to define the following matrices with components
\begin{alignat*}{3}
(\bm{C}^{k})_{ji} := \int_{\Omega} w_{i} \bm{v}_{k} \cdot \nabla w_{j}  \dx, \quad (\bm{M}_{\pd \Omega})_{ji}  := \int_{\pd \Omega}  w_{i} w_{j} \dH, \\
(\bm{S}_{m}^{k})_{ji} := \int_{\Omega} m(\varphi_{k}) \nabla w_{i} \cdot \nabla w_{j} \dx, \quad (\bm{S}_{n}^{k})_{ji} := \int_{\Omega} n(\varphi_{k}) \nabla w_{i} \cdot \nabla w_{j} \dx,
\end{alignat*}
for $1 \leq i, j \leq k$.  Furthermore, we introduce the notation
\begin{align*}
R_{\varphi,j}^{k} := \int_{\Omega} \Gamma_{\varphi,k} w_{j} \dx ,  \quad  R_{S,j}^{k} := \int_{\Omega} \mathcal{S}_{k} w_{j} \dx, \\
\psi_{j}^{k} := \int_{\Omega} \Psi'(\varphi_{k})w_{j} \dx, \quad \Sigma_{j}^{k} := \int_{\pd \Omega} \sigma_{\infty} w_{j} \dH,
\end{align*}
for $1 \leq i,j \leq k$, and denote
\begin{align*}
\bm{R}_{\varphi}^{k} := (R_{\varphi,1}^{k}, \dots, R_{\varphi,k}^{k})^{\top}, \quad \bm{R}_{S}^{k} := (R_{S,1}^{k}, \dots, R_{S,k}^{k})^{\top}, \\
\bm{\psi}^{k} := (\psi_{1}^{k}, \dots, \psi_{k}^{k})^{\top}, \quad  \bm{\Sigma}^{k} := (\Sigma_{1}^{k}, \dots, \Sigma_{k}^{k})^{\top},
\end{align*}
as the corresponding vectors, so that we obtain an initial value problem for a system of equations for $\bm{\alpha}_{k} := (\alpha_{1}^{k}, \dots \alpha_{k}^{k})^{\top}$, $\bm{\beta}_{k} := (\beta_{1}^{k}, \dots, \beta_{k}^{k})^{\top}$, and $\bm{\gamma}_{k} := (\gamma_{1}^{k}, \dots, \gamma_{k}^{k})^{\top}$ as follows,
\begin{subequations}\label{discrete:system1}
\begin{align}
\frac{\dd}{\dt} \bm{\alpha}_{k} & = - \bm{S}_{m}^{k} \bm{\beta}_{k} + \bm{R}_{\varphi}^{k} + \bm{C}^{k} \bm{\alpha}_{k}, \label{discrete:varphi} \\
 \bm{\beta}_{k} & = A \bm{\psi}^{k} + B\bm{S} \bm{\alpha}_{k} - \chi  \bm{\gamma}_{k}, \label{discrete:mu} \\
 \frac{\dd}{\dt} \bm{\gamma}_{k} & = - \bm{S}_{n}^{k}(D \bm{\gamma}_{k} - \chi \bm{\alpha}_{k}) - \bm{R}_{S}^{k} + \bm{C}^{k} \bm{\gamma}_{k} - b\bm{M}_{\pd \Omega}\bm{\gamma}_{k} + b \bm{\Sigma}^{k}, \label{discrete:sigma} \\
 p_{k} & = \left (-\Laplace_{N} \right)^{-1} \left ( \frac{1}{K} \Gamma_{\bm{v}} - \div (( \mu_{k} + \chi \sigma_{k}) \nabla \varphi_{k}) \right ) \label{discrete:pressure}, \\
 \bm{v}_{k} & =  -K (\nabla p_{k} - (\mu_{k} + \chi \sigma_{k}) \nabla \varphi_{k}) , \label{discrete:velo}
\end{align}
\end{subequations}
and we supplement \eqref{discrete:system1} with the initial conditions
\begin{align}\label{discrete:system:initial}
(\bm{\alpha}_{k})_{j}(0) = \int_{\Omega} \varphi_{0} w_{j} \dx, \quad (\bm{\gamma}_{k})_{j}(0) = \int_{\Omega} \sigma_{0} w_{j} \dx,
\end{align}
for $1 \leq j \leq k$, which satisfy
\begin{align}\label{initialcond:bdd}
\bignorm{\sum_{i=1}^{k} (\bm{\alpha}_{k})_{i}(0) w_{i}}_{H^{1}} \leq C\norm{\varphi_{0}}_{H^{1}}, \quad \bignorm{\sum_{j=1}^{k} (\bm{\gamma}_{k})_{i}(0) w_{i}}_{L^{2}} \leq \norm{\sigma_{0}}_{L^{2}} \quad \forall k \in \N,
\end{align}
for some constant $C$ not depending on $k$.

We can substitute \eqref{discrete:mu}, \eqref{discrete:pressure} and \eqref{discrete:velo} into \eqref{discrete:varphi} and \eqref{discrete:sigma}, and obtain a coupled system of ordinary differential equations for $\bm{\alpha}_{k}$ and $\bm{\gamma}_{k}$, where $\bm{S}_{m}^{k}$, $\bm{C}^{k}$ and $\bm{S}_{n}^{k}$ depend on the solutions $\bm{\alpha}_{k}$ and $\bm{\gamma}_{k}$ in a non-linear manner.  Continuity of $m(\cdot)$, $n(\cdot)$, $\Psi'(\cdot)$ and the source terms, and the stability of $(-\Laplace_{N})^{-1}$ under perturbations imply that the right-hand sides of \eqref{discrete:system1} depend continuously on $(\bm{\alpha}_{k}, \bm{\gamma}_{k})$.   Thus, we can appeal to the theory of ODEs (via the Cauchy--Peano theorem \cite[Chap. 1, Thm. 1.2]{book:Coddington}) to infer that the initial value problem \eqref{discrete:system1}-\eqref{discrete:system:initial} has at least one local solution pair $(\bm{\alpha}_{k}, \bm{\gamma}_{k})$ defined on $[0,t_{k}]$ for each $k \in \N$.

We may define $\bm{\beta}_{k}$ via the relation \eqref{discrete:mu} and hence the Galerkin ansatz $\varphi_{k}, \mu_{k}$ and $\sigma_{k}$ can be constructed from \eqref{Galerkin:ansatz}.  Then, we can define $p_{k}$ and $\bm{v}_{k}$ via \eqref{Galerkin:pressure} and \eqref{Galerkin:velocity}, respectively.  Furthermore, as the basis function $w_{j}$ belongs to $H^{2}$ for each $j \in \N$, by the Sobolev embedding $H^{2} \subset L^{\infty}$, we obtain that $\div (w_{i} \nabla w_{j}) \in L^{2}$ for $i,j \in \N$ and hence the function $\div ((\mu_{k} + \chi \sigma_{k}) \nabla \varphi_{k})$ belongs to $L^{2}$.  Then, by elliptic regularity theory, we find that $p_{k}(t) \in H^{2}_{N} \cap L^{2}_{0}$ for all $t \in [0,t_{k}]$.  This in turn implies that
\begin{align}\label{Galerkinvkspace}
\bm{v}_{k}(t) \in \{ \bm{f} \in \bm{H}^{1} : \div \bm{f} = \Gamma_{\bm{v}}, \; \bm{f} \cdot \bm{n} = 0 \text{ on } \pd \Omega \} \text{ for all } t \in [0,t_{k}].
\end{align}
Next, we show that the Galerkin ansatz can be extended to the interval $[0,T]$ using a priori estimates.

\section{A priori estimates}\label{sec:apriori}
In this section, the positive constants $C$ are independent of $k$, $\Gamma_{\bm{v}}$, $K$, $b$ and $\chi$, and may change from line to line.  We will denote positive constants that are uniformly bounded for $b, \chi \in (0,1]$ and are also uniformly bounded for $K \in (0,1]$ when $\Gamma_{\bm{v}} = 0$ by the symbol $\mathcal{E}$.

We first state the energy identity satisfied by the Galerkin ansatz.
Let $\delta_{ij}$ denote the Kronecker delta.  Multiplying \eqref{Galerkin:varphi} with $\beta_{j}^{k}$, \eqref{Galerkin:mu} with $\frac{\dd}{\dt} \alpha_{j}^{k}$, \eqref{Galerkin:sigma} with $N_{,\sigma}^{k} := D \gamma_{j}^{k} + \chi (\delta_{1j} - \alpha_{j}^{k})$, and then summing the product from $j = 1$ to $k$ lead to
\begin{subequations}
\begin{align*}
\int_{\Omega} \pd_{t}\varphi_{k} \mu_{k} \dx & = \int_{\Omega} - m(\varphi_{k}) \abs{\nabla \mu_{k}}^{2} + \Gamma_{\varphi,k} \mu_{k} +  \varphi_{k} \bm{v}_{k} \cdot \nabla \mu_{k} \dx , \\
\int_{\Omega} \mu_{k} \pd_{t}\varphi_{k} \dx & = \frac{\dd}{\dt}\int_{\Omega} A \Psi'(\varphi_{k}) + \frac{B}{2} \abs{\nabla \varphi_{k}}^{2} \dx - \int_{\Omega} \chi \sigma_{k} \pd_{t}\varphi_{k} \dx, \\
\int_{\Omega} \pd_{t}\sigma_{k} N_{,\sigma}^{k} \dx & = \int_{\Omega} -n(\varphi_{k}) \abs{\nabla N_{,\sigma}^{k}}^{2} - \mathcal{S}_{k} N_{,\sigma}^{k} +  \sigma_{k} \bm{v}_{k} \cdot \nabla N_{,\sigma}^{k} \dx \\
\notag & + \int_{\pd \Omega} b (\sigma_{\infty} - \sigma_{k}) N_{,\sigma}^{k} \dH.
\end{align*}
\end{subequations}
Here, we used that $w_{1} = 1$ and $\nabla w_{1} = \bm{0}$.  Then, summing the three equations leads to
\begin{equation}\label{apriorienergy1}
\begin{aligned}
& \frac{\dd}{\dt} \int_{\Omega}  A \Psi(\varphi_{k}) + \frac{B}{2} \abs{\nabla \varphi_{k}}^{2} + N(\varphi_{k}, \sigma_{k}) \dx \\
& \quad +  \int_{\Omega} m(\varphi_{k}) \abs{\nabla \mu_{k}}^{2} + n(\varphi_{k}) \abs{\nabla N_{,\sigma}^{k}}^{2} \dx + \int_{\pd \Omega} Db \abs{\sigma_{k}}^{2} \dH \\
& \quad =  \int_{\Omega} \Gamma_{\varphi,k} \mu_{k} - \mathcal{S}_{k} N_{,\sigma}^{k} + (\varphi_{k} \bm{v}_{k} \cdot \nabla \mu_{k} + \sigma_{k} \bm{v}_{k} \cdot \nabla N_{,\sigma}^{k}) \dx \\
& \quad +  \int_{\pd \Omega}  b (\sigma_{\infty} N_{,\sigma}^{k} - \sigma_{k} \chi (1-\varphi_{k}))\dH.
\end{aligned}
\end{equation}
Next, multiplying \eqref{Galerkin:velocity} with $\frac{1}{K} \bm{v}_{k}$, integrating over $\Omega$ and integrating by parts gives
\begin{align*}
\int_{\Omega} \frac{1}{K} \abs{\bm{v}_{k}}^{2} \dx & = \int_{\Omega} -\nabla p_{k} \cdot \bm{v}_{k} +(\mu_{k} + \chi \sigma_{k}) \nabla \varphi_{k} \cdot \bm{v}_{k} \dx \\
& = \int_{\Omega} \Gamma_{\bm{v}} p_{k} + (\mu_{k} + \chi \sigma_{k}) \nabla \varphi_{k} \cdot \bm{v}_{k} \dx,
\end{align*}
where we used that $\div \bm{v}_{k} = \Gamma_{\bm{v}}$ and $\bm{v}_{k} \cdot \bm{n} = 0$ on $\pd \Omega$.  Similarly, we see that
\begin{align*}
\int_{\Omega} (\varphi_{k} \nabla \mu_{k} + \sigma_{k} \nabla N_{,\sigma}^{k} ) \cdot \bm{v}_{k} \dx & = \int_{\Omega} \varphi_{k} \bm{v}_{k} \cdot \nabla \mu_{k} + \sigma_{k} \bm{v}_{k} \cdot \nabla (D \sigma_{k} + \chi (1-\varphi_{k})) \dx \\
& =  -\int_{\Omega} \varphi_{k} \Gamma_{\bm{v}} \mu_{k} + (\mu_{k} + \chi \sigma_{k}) \bm{v}_{k} \cdot \nabla \varphi_{k} - \tfrac{D}{2} \bm{v}_{k} \cdot \nabla \abs{\sigma_{k}}^{2} \dx \\
& = - \int_{\Omega} \Gamma_{\bm{v}} \left ( \varphi_{k} \mu_{k} + \tfrac{D}{2} \abs{\sigma_{k}}^{2} \right ) + (\mu_{k} + \chi \sigma_{k}) \nabla \varphi_{k} \cdot \bm{v}_{k} \dx.
\end{align*}
In particular, we have
\begin{equation}\label{apriorienergy2}
\begin{aligned}
\int_{\Omega} \frac{1}{K} \abs{\bm{v}_{k}}^{2} \dx = \int_{\Omega}  \Gamma_{\bm{v}} \left ( p_{k} - \mu_{k} \varphi_{k} - \frac{D}{2} \abs{\sigma_{k}}^{2} \right ) - (\varphi_{k} \nabla \mu_{k} + \sigma_{k} \nabla N_{,\sigma}^{k}) \cdot \bm{v}_{k} \dx.
\end{aligned}
\end{equation}
Adding \eqref{apriorienergy2} to \eqref{apriorienergy1} leads to
\begin{equation}\label{apriorienergy:identity}
\begin{aligned}
& \frac{\dd}{\dt} \int_{\Omega} A \Psi(\varphi_{k}) + \frac{B}{2} \abs{\nabla \varphi_{k}}^{2} + N(\varphi_{k}, \sigma_{k}) \dx \\
& \quad + \int_{\Omega} m(\varphi_{k}) \abs{\nabla \mu_{k}}^{2} + n(\varphi_{k}) \abs{\nabla N_{,\sigma}^{k}}^{2} + \frac{1}{K} \abs{\bm{v}_{k}}^{2} \dx + \int_{\pd \Omega} Db \abs{\sigma_{k}}^{2} \dH \\
& \quad =  \int_{\Omega} \Gamma_{\varphi,k} \mu_{k} - \mathcal{S}_{k} N_{,\sigma}^{k} + \Gamma_{\bm{v}} \left ( p_{k} - \mu_{k} \varphi_{k} - \frac{D}{2} \abs{\sigma_{k}}^{2} \right ) \dx \\
& \quad + \int_{\pd \Omega} b(\sigma_{\infty} (D \sigma_{k} + \chi (1-\varphi_{k})) - \sigma_{k} \chi(1-\varphi_{k})) \dH.
\end{aligned}
\end{equation}
To derive the first a priori estimate for the Galerkin ansatz, it suffices to bring \eqref{apriorienergy:identity} into a form where we can apply Gronwall's inequality.  We start with estimating the boundary term on the right-hand side of \eqref{apriorienergy:identity}.  By H\"{o}lder's inequality and Young's inequality,
\begin{align*}
& \abs{\int_{\pd \Omega} b (\sigma_{\infty} (D \sigma_{k}  + \chi (1-\varphi_{k})) - \sigma_{k} \chi(1-\varphi_{k})) \dH} \\
& \quad \leq b \left ( \norm{\sigma_{\infty}}_{L^{2}(\pd \Omega)} \norm{D \sigma_{k} + \chi (1-\varphi_{k})}_{L^{2}(\pd \Omega)} + \chi \norm{\sigma_{k}}_{L^{2}(\pd \Omega)} \left ( \abs{\pd \Omega}^{\frac{1}{2}} + \norm{\varphi_{k}}_{L^{2}(\pd \Omega)} \right ) \right ) \\
& \quad \leq \frac{Db}{2} \norm{\sigma_{k}}_{L^{2}(\pd \Omega)}^{2} + b \left (1 + \frac{\chi^{2}}{D} \right ) \norm{\varphi_{k}}_{L^{2}(\pd \Omega)}^{2} + bC \left (\chi + (1 + \chi^{2})\norm{\sigma_{\infty}}_{L^{2}(\pd \Omega)}^{2} \right ).
\end{align*}
By the trace theorem and the growth condition \eqref{assump:lowerbddPsi}, we have
\begin{equation}\label{varphiL2Gammanorm}
\begin{aligned}
\norm{\varphi}_{L^{2}(\pd \Omega)}^{2} & \leq C_{\mathrm{tr}}^{2} \left (\norm{\varphi}_{L^{2}}^{2} + \norm{\nabla \varphi}_{\bm{L}^{2}}^{2} \right )\\
& \leq C_{\mathrm{tr}}^{2} \left (\frac{1}{R_{1}} \norm{\Psi(\varphi)}_{L^{1}} + \norm{\nabla \varphi}_{\bm{L}^{2}}^{2} \right ) + C(R_{2}, \abs{\Omega}, C_{\mathrm{tr}}),
\end{aligned}
\end{equation}
where the positive constant $C_{\mathrm{tr}}$ from the trace theorem only depends on $\Omega$, and so
\begin{equation}\label{boundryintegralRHS}
\begin{aligned}
& \abs{\int_{\pd \Omega} b (\sigma_{\infty} (D \sigma_{k}  + \chi (1-\varphi_{k})) - \sigma_{k} \chi(1-\varphi_{k})) \dH} \\
& \quad \leq \frac{Db}{2} \norm{\sigma_{k}}_{L^{2}(\pd \Omega)}^{2} + C b \left (1 + \chi^{2} \right )  \left ( \norm{\Psi(\varphi_{k})}_{L^{1}} + \norm{\nabla \varphi_{k}}_{\bm{L}^{2}}^{2} \right )\\
& \quad + C b \left (1 + \chi^{2} \right ) + bC \left ( \chi + (1 + \chi^{2})\norm{\sigma_{\infty}}_{L^{2}(\pd \Omega)}^{2} \right ).
\end{aligned}
\end{equation}

\subsection{Estimation of the source terms}
For the source term
\begin{align*}
\int_{\Omega} \Gamma_{\varphi,k} \mu_{k} - \mathcal{S}_{k} N_{,\sigma}^{k} + \Gamma_{\bm{v}} \left ( p_{k} - \mu_{k} \varphi_{k} - \frac{D}{2} \abs{\sigma_{k}}^{2} \right ) \dx
\end{align*}
that appears on the right-hand side of \eqref{apriorienergy:identity} we will divide its analysis into two parts.  We first analyse the part involving $\Gamma_{\bm{v}}$, which will involve a closer look at the Darcy subsystem to deduce an estimate on $\norm{p_{k}}_{L^{2}}$.  For the remainder $\Gamma_{\varphi,k} \mu_{k} - \mathcal{S}_{k} N_{,\sigma}^{k}$ term we will estimate it differently based on the assumptions on $\Theta_{\varphi}$.

\subsubsection{Pressure estimates}\label{sec:pressureest}
Before we estimate the source terms involving $\Gamma_{\bm{v}}$, we look at the Darcy subsystem, which can be expressed as an elliptic equation for the pressure (we will drop the subscript $k$ for clarity)
\begin{subequations}\label{pressure:sys}
\begin{alignat}{3}
-\Laplace p & = \frac{1}{K} \Gamma_{\bm{v}} - \div ( (\mu + \chi \sigma) \nabla \varphi) && \text{ in } \Omega, \text{ with } \mean{p} = 0, \\
\pdnu p & = 0 && \text{ on } \pd \Omega.
\end{alignat}
\end{subequations}
The following lemma is similar to \cite[Lem. 3.1]{preprint:JiangWuZheng14}, and the hypothesis is fulfilled by the Galerkin ansatz.

\begin{lemma}\label{lem:pressure:est}
Let $\Omega \subset \R^{3}$ be a bounded domain with $C^{3}$-boundary.  Given $\varphi \in H^{2}_{N}$, $\mu, \sigma \in H^{1}$, the source term $\Gamma_{\bm{v}} \in L^{2}_{0}$, and the function $p$ satisfying the above elliptic equation \eqref{pressure:sys}.  Then, the following estimate hold
\begin{equation}\label{pressure:L2:est}
\begin{aligned}
\norm{p}_{L^{2}} \leq \frac{C}{K} \norm{\Gamma_{\bm{v}}}_{L^{2}} + C \left (\norm{\nabla \mu}_{\bm{L}^{2}} + \chi \norm{\sigma}_{L^{6}} \right ) \norm{\nabla \varphi}_{\bm{L}^{\frac{3}{2}}} + C \mean{\mu} \norm{\nabla \varphi}_{\bm{L}^{2}},
\end{aligned}
\end{equation}
for some positive constant $C$ depending only on $\Omega$.
\end{lemma}

\begin{proof}
We first recall some properties of the inverse Neumann-Laplacian operator.  Suppose for $g \in L^{2}_{0}$, $f = (-\Laplace_{N})^{-1} g \in H^{1} \cap L^{2}_{0}$ solves
\begin{align}\label{elliptic:example}
-\Laplace f = g \text{ in } \Omega, \quad \pdnu f = 0 \text{ on } \pd \Omega.
\end{align}
Then, testing with $f$ and integrating over $\Omega$, applying integration by parts and the Poincar\'{e} inequality \eqref{regular:Poincare} leads to
\begin{align}\label{NeumannEll:basic:Est}
\norm{\left ( - \Laplace_{N} \right)^{-1} g}_{H^{1}} = \norm{f}_{H^{1}} \leq c \norm{\nabla f}_{\bm{L}^{2}} \leq C \norm{g}_{L^{2}},
\end{align}
for positive constants $c$ and $C$ depending only on $C_{p}$.  Elliptic regularity theory then gives that $f \in H^{2}_{N}$ with
\begin{align}\label{NeumannEllipticH2}
\norm{f}_{H^{2}} \leq C \left (\norm{f}_{H^{1}} + \norm{g}_{L^{2}} \right ) \leq C \norm{g}_{L^{2}},
\end{align}
with a positive constant $C$ depending only on $\Omega$.  Returning to the pressure system, we observe from \eqref{pressure:nonlocal:defn} and the above that
\begin{equation}\label{pressure:sys:L2:estimate}
\begin{aligned}
 \norm{p}_{L^{2}} & \leq \frac{1}{K} \norm{\left ( -\Laplace_{N} \right)^{-1} \Gamma_{\bm{v}}}_{L^{2}} + \norm{ \left ( -\Laplace_{N} \right )^{-1} \left ( \div (( \mu + \chi \sigma) \nabla \varphi) \right ) }_{L^{2}} \\
& \leq \frac{C}{K} \norm{\Gamma_{\bm{v}}}_{L^{2}}  + \norm{ \left ( -\Laplace_{N} \right )^{-1} \left ( \div (( \mu - \mean{\mu} + \chi \sigma) \nabla \varphi) \right ) }_{L^{2}} \\
& + \norm{\left (-\Laplace_{N} \right)^{-1} (\div (\mean{\mu} \,  \nabla \varphi))}_{L^{2}},
\end{aligned}
\end{equation}
for some positive constant $C$ depending only on $C_{p}$.  Note that the third term on the right-hand side can be estimated as
\begin{align}\label{pressureterm:est:simplification}
\mean{\mu} \norm{ \left ( -\Laplace_{N} \right)^{-1} \div \nabla (\varphi - \mean{\varphi})}_{L^{2}} = \mean{\mu} \norm{\varphi - \mean{\varphi}}_{L^{2}}\leq C_{p} \mean{\mu} \norm{\nabla \varphi}_{\bm{L}^{2}}.
\end{align}
We now consider estimating the second term on the right-hand side of \eqref{pressure:sys:L2:estimate}.  By assumption $\mu, \sigma \in H^{1}$ and $\varphi \in H^{2}_{N}$, we have that
\begin{align}\label{auxiliary:pressure:RHS:L2}
\norm{(\mu - \overline{\mu} + \chi \sigma)\nabla \varphi}_{\bm{L}^{2}} \leq \norm{\mu - \overline{\mu} + \chi \sigma}_{L^{6}} \norm{\nabla \varphi}_{\bm{L}^{3}},
\end{align}
and so if we consider the function $h := (-\Laplace_{N})^{-1} (\div ((\mu - \mean{\mu} + \chi \sigma) \nabla \varphi))$, then we obtain that
\begin{align}\label{auxiliary:pressure:weak}
\int_{\Omega} \nabla h \cdot \nabla \zeta \dx = \int_{\Omega} - (\mu - \overline{\mu} + \chi \sigma) \nabla \varphi \cdot \nabla \zeta \dx \quad \forall \zeta \in H^{1}
\end{align}
must hold, and by \eqref{auxiliary:pressure:RHS:L2} and the Poincar\'{e} inequality \eqref{regular:Poincare} with zero mean  it holds that $h \in H^{1} \cap L^{2}_{0}$.  We now define $f:= (-\Laplace_{N})^{-1}(h) \in H^{2}_{N}$, and consider testing with $\zeta = f$ in \eqref{auxiliary:pressure:weak}, leading to
\begin{align*}
\int_{\Omega} \abs{h}^{2} \dx = \int_{\Omega} \nabla h \cdot \nabla f \dx =  \int_{\Omega} -(\mu - \overline{\mu} + \chi \sigma) \nabla \varphi \cdot \nabla f \dx.
\end{align*}
Since $f \in H^{2}_{N}$, elliptic regularity theory and H\"{o}lder's inequality gives
\begin{align*}
\norm{h}_{L^{2}}^{2}  & \leq \norm{(\mu  - \overline{\mu} + \chi \sigma) \nabla \varphi}_{\bm{L}^{\frac{6}{5}}} \norm{\nabla f}_{\bm{L}^{6}} \leq C \norm{(\mu  - \overline{\mu} + \chi \sigma) \nabla \varphi}_{\bm{L}^{\frac{6}{5}}} \norm{f}_{H^{2}} \\
& \leq C\norm{(\mu  - \overline{\mu} + \chi \sigma) \nabla \varphi}_{\bm{L}^{\frac{6}{5}}} \norm{h}_{L^{2}},
\end{align*}
where the constant $C$ depends on $\Omega$ and the constant in \eqref{NeumannEllipticH2}.  Thus we obtain
\begin{equation}\label{InvLaplDiv:L2:L6/5}
\begin{aligned}
\norm{ \left ( -\Laplace_{N} \right )^{-1} \left ( \div (( \mu - \mean{\mu} + \chi \sigma) \nabla \varphi) \right ) }_{L^{2}} & \leq C \norm{(\mu  - \overline{\mu} + \chi \sigma) \nabla \varphi}_{\bm{L}^{\frac{6}{5}}} \\
& \leq C \left ( \norm{\mu - \mean{\mu}}_{L^{6}} + \chi \norm{\sigma}_{L^{6}} \right ) \norm{\nabla \varphi}_{\bm{L}^{\frac{3}{2}}}
\end{aligned}
\end{equation}
for some constant $C$ depending only on $\Omega$.  By the Sobolev embedding $H^{1} \subset L^{6}$ (with constant $C_{\mathrm{Sob}}$ that depends only on $\Omega$) and the Poincar\'{e} inequality, we find that
\begin{align}\label{muL6Poin}
\norm{\mu - \mean{\mu}}_{L^{6}} \leq C_{\mathrm{Sob}} \norm{\mu - \mean{\mu}}_{H^{1}} \leq c(C_{\mathrm{Sob}}, C_{p}) \norm{\nabla \mu}_{\bm{L}^{2}}.
\end{align}
Substituting the above elements into \eqref{pressure:sys:L2:estimate} yields \eqref{pressure:L2:est}.
\end{proof}

\begin{remark}
We choose not to use the estimate
\begin{align}\label{Alt:pressure:L2:nabalvarphi:L3}
c \norm{h}_{L^{2}} \leq \norm{\nabla h}_{\bm{L}^{2}} \leq \norm{(\mu - \overline{\mu} + \chi \sigma) \nabla \varphi}_{\bm{L}^{2}}
\end{align}
obtained from substituting $\zeta = h$ in \eqref{auxiliary:pressure:weak}, where $c$ is a positive constant depending only on $C_{p}$, since by \eqref{auxiliary:pressure:RHS:L2} we require control of $\nabla \varphi$ in the $\bm{L}^{3}(\Omega)$-norm and this is not available when deriving the first a priori estimate.  Thus, we make use of the auxiliary problem $f = (-\Laplace_{N})^{-1}(h)$ to derive another estimate on $\norm{h}_{L^{2}}$ that involves controlling $\nabla \varphi$ in the weaker $\bm{L}^{\frac{3}{2}}(\Omega)$-norm.
\end{remark}

Next, we state regularity estimates for the pressure and the velocity field.  The hypothesis will be fulfilled for the Galerkin ansatz once we derived the a priori estimates in Section \ref{sec:apriori}.  Note that in Lemma \ref{lem:reg:pressurevelocity} below, we consider a source term $\Gamma_{\bm{v}} \in L^{2}(0,T;L^{2}_{0})$, so that our new regularity results for the pressure and the velocity is also applicable to the setting considered in \cite{preprint:JiangWuZheng14}.
\begin{lemma}\label{lem:reg:pressurevelocity}
Let $\varphi \in L^{\infty}(0,T;H^{1}) \cap L^{2}(0,T;H^{2}_{N} \cap H^{3})$, $\sigma \in L^{2}(0,T;H^{1})$, $\mu \in L^{2}(0,T;H^{1})$, the source term $\Gamma_{\bm{v}} \in L^{2}(0,T;L^{2}_{0})$, and the function $p$ satisfying \eqref{pressure:sys}.  Then,
\begin{align}\label{pressure:L8/5H1bdd}
\norm{p}_{L^{\frac{8}{5}}(H^{1})}^{\frac{8}{5}} \leq C_{1} \norm{\varphi}_{L^{\infty}(H^{1})}^{\frac{6}{5}} \norm{\mu + \chi\sigma}_{L^{2}(H^{1})}^{\frac{8}{5}}  \norm{\varphi}_{L^{2}(H^{3})}^{\frac{2}{5}} + \frac{C_{1}}{K} T^{\frac{1}{5}} \norm{\Gamma_{\bm{v}}}_{L^{2}(L^{2})}^{\frac{8}{5}},
\end{align}
for some positive constant $C_{1}$ depending only on $\Omega$, and
\begin{equation}\label{pressure:L8/7H2bdd}
\begin{aligned}
\norm{p}_{L^{\frac{8}{7}}(H^{2})}^{\frac{8}{7}} & \leq C_{2} T^{\frac{3}{7}} K^{-\frac{8}{7}}\norm{\Gamma_{\bm{v}}}_{L^{2}(L^{2})}^{\frac{8}{7}} + C_{2} T^{\frac{2}{7}} \norm{p}_{L^{\frac{8}{5}}(H^{1})}^{\frac{8}{7}} \\
& + C_{2}\norm{\varphi}_{L^{\infty}(H^{1})}^{\frac{2}{7}} \norm{\mu + \chi \sigma}_{L^{2}(H^{1})}^{\frac{8}{7}}  \norm{\varphi}_{L^{2}(H^{3})}^{\frac{6}{7}},
\end{aligned}
\end{equation}
for some positive constant $C_{2}$ depending only on $\Omega$.  Moreover, if we have the relation
\begin{align*}
\bm{v} = -K \left ( \nabla p - (\mu + \chi \sigma) \nabla \varphi \right ),
\end{align*}
then
\begin{equation}\label{velocityL8/7H1bdd}
\begin{aligned}
\norm{ \der \bm{v}}_{L^{\frac{8}{7}}(\bm{L}^{2})}^{\frac{8}{7}} & \leq C_{3}K \norm{p}_{L^{\frac{8}{7}}(H^{2})}^{\frac{8}{7}}  + C_{3}K \norm{\mu + \chi \sigma}_{L^{2}(H^{1})}^{\frac{8}{7}} \norm{\varphi}_{L^{2}(H^{3})}^{\frac{6}{7}} \norm{\varphi}_{L^{\infty}(H^{1})}^{\frac{2}{7}},
\end{aligned}
\end{equation}
for some positive constant $C_{3}$ depending only on $\Omega$.
\end{lemma}
\begin{proof}
From \eqref{pressure:sys} we see that $p$ satisfies $\mean{p} = 0$ and
\begin{align*}
\int_{\Omega} \nabla p \cdot \nabla \zeta \dx = \int_{\Omega} (\mu + \chi \sigma) \nabla \varphi \cdot \nabla \zeta + \frac{1}{K} \Gamma_{\bm{v}} \zeta \dx \quad \forall \zeta \in H^{1}(\Omega).
\end{align*}
Testing with $\zeta = p$ and applying the H\"{o}lder's inequality and the Poincar\'{e} inequality \eqref{regular:Poincare} gives
\begin{equation}\label{L2:gradpressure}
\begin{aligned}
\norm{\nabla p}_{\bm{L}^{2}} \leq \norm{(\mu + \chi \sigma) \nabla \varphi}_{\bm{L}^{2}} + \frac{C_{p}}{K} \norm{\Gamma_{\bm{v}}}_{L^{2}}.
\end{aligned}
\end{equation}
Applying H\"{o}lder's inequality and the Sobolev embedding $H^{1} \subset L^{6}$ yields that
\begin{equation}\label{Galerkin:pressure:L2grad}
\begin{aligned}
\norm{(\mu + \chi \sigma) \nabla \varphi}_{\bm{L}^{2}} & \leq \norm{\mu + \chi \sigma}_{L^{6}} \norm{\nabla \varphi}_{\bm{L}^{3}}  \leq C_{\mathrm{Sob}} \norm{\mu + \chi \sigma}_{H^{1}} \norm{\nabla \varphi}_{\bm{L}^{3}}.
\end{aligned}
\end{equation}
By the Gagliardo--Nirenberg inequality \eqref{GagNirenIneq} with parameters $j = 0$, $p = 3$, $r = 2$, $m = 2$, $d = 3$ and $q = 2$,
\begin{align}\label{GNineq}
\norm{\nabla \varphi}_{\bm{L}^{3}} \leq C \norm{\nabla \varphi}_{\bm{H}^{2}}^{\frac{1}{4}} \norm{\nabla \varphi}_{\bm{L}^{2}}^{\frac{3}{4}} \leq C\norm{\varphi}_{H^{3}}^{\frac{1}{4}} \norm{\varphi}_{H^{1}}^{\frac{3}{4}},
\end{align}
where $C > 0$ is a constant depending only on $\Omega$.  Then, the boundedness of $\mu, \sigma$ in $L^{2}(0,T;H^{1})$ and $\varphi$ in $L^{2}(0,T;H^{3}) \cap L^{\infty}(0,T;H^{1})$ leads to
\begin{align*}
\int_{0}^{T} \norm{(\mu + \chi \sigma) \nabla \varphi}_{\bm{L}^{2}}^{\frac{8}{5}} \dt & \leq C \int_{0}^{T} \norm{\mu + \chi \sigma}_{H^{1}}^{\frac{8}{5}}  \norm{\varphi}_{H^{3}}^{\frac{2}{5}} \norm{\varphi}_{H^{1}}^{\frac{6}{5}} \dt \\
& \leq C \norm{\varphi}_{L^{\infty}(H^{1})}^{\frac{6}{5}} \norm{\mu + \chi \sigma}_{L^{2}(H^{1})}^{\frac{8}{5}} \norm{\varphi}_{L^{2}(H^{3})}^{\frac{2}{5}}.
\end{align*}
By \eqref{L2:gradpressure} we find that
\begin{align*}
\int_{0}^{T} \norm{\nabla p}_{\bm{L}^{2}}^{\frac{8}{5}} \dt & \leq \int_{0}^{T} \norm{(\mu + \chi \sigma) \nabla \varphi}_{\bm{L}^{2}}^{\frac{8}{5}} + \frac{C}{K} \norm{\Gamma_{\bm{v}}}_{L^{2}}^{\frac{8}{5}} \dt \\
& \leq  C \norm{\varphi}_{L^{\infty}(H^{1})}^{\frac{6}{5}}  \norm{\mu + \chi \sigma}_{L^{2}(H^{1})}^{\frac{8}{5}} \norm{\varphi}_{L^{2}(H^{3})}^{\frac{2}{5}} + \frac{C}{K}  T^{\frac{1}{5}} \norm{\Gamma_{\bm{v}}}_{L^{2}(L^{2})}^{\frac{8}{5}},
\end{align*}
where the positive constant $C$ depends only on $\Omega$.  As $\mean{p} = 0$, by the Poincar\'{e} inequality \eqref{regular:Poincare}, we see that
\begin{align*}
\norm{p}_{L^{\frac{8}{5}}(H^{1})}^{\frac{8}{5}} & \leq C \norm{\varphi}_{L^{\infty}(H^{1})}^{\frac{6}{5}} \norm{\mu + \chi \sigma}_{L^{2}(H^{1})}^{\frac{8}{5}} \norm{\varphi}_{L^{2}(H^{3})}^{\frac{2}{5}} + \frac{C}{K} T^{\frac{1}{5}}\norm{\Gamma_{\bm{v}}}_{L^{2}(L^{2})}^{\frac{8}{5}},
\end{align*}
for some positive constant $C$ depending only on $\Omega$.  Next, we see that
\begin{align*}
\norm{\div ((\mu + \chi \sigma) \nabla \varphi)}_{L^{2}} & \leq \norm{(\mu + \chi \sigma) \Laplace \varphi}_{L^{2}} + \norm{\nabla (\mu + \chi \sigma) \cdot \nabla \varphi}_{L^{2}} \\
& \leq \norm{\mu + \chi \sigma}_{L^{6}} \norm{\Laplace \varphi}_{L^{3}} + \norm{\nabla (\mu + \chi \sigma)}_{\bm{L}^{2}} \norm{\nabla \varphi}_{\bm{L}^{\infty}} .
\end{align*}
By the Gagliardo--Nirenberg inequality (\ref{GagNirenIneq}), we find that
\begin{equation}\label{GN:D2fL3andnablafLinfty}
\begin{aligned}
\norm{D^{2} \varphi}_{L^{3}} & \leq C \norm{\varphi}_{H^{3}}^{\frac{3}{4}} \norm{\varphi}_{L^{6}}^{\frac{1}{4}} \leq C \norm{\varphi}_{H^{3}}^{\frac{3}{4}} \norm{\varphi}_{H^{1}}^{\frac{1}{4}}, \\
\norm{\nabla \varphi}_{\bm{L}^{\infty}} & \leq C \norm{\varphi}_{H^{3}}^{\frac{3}{4}} \norm{\varphi}_{L^{6}}^{\frac{1}{4}} \leq C \norm{\varphi}_{H^{3}}^{\frac{3}{4}} \norm{\varphi}_{H^{1}}^{\frac{1}{4}},
\end{aligned}
\end{equation}
and so, we have
\begin{equation}\label{pressure:divRHSL2}
\begin{aligned}
\norm{\div ((\mu + \chi \sigma) \nabla \varphi)}_{L^{2}} \leq C \norm{\mu + \chi \sigma}_{H^{1}} \norm{\varphi}_{H^{3}}^{\frac{3}{4}} \norm{\varphi}_{H^{1}}^{\frac{1}{4}}.
\end{aligned}
\end{equation}
That is, $\div (( \mu + \chi \sigma) \nabla \varphi) \in L^{2}$.  Since by assumption $\Gamma_{\bm{v}} \in L^{2}_{0}$, using elliptic regularity theory, we find that $p(t,\cdot) \in H^{2}$ for a.e. $t$ and there exists a constant $C$ depending only on $\Omega$, such that
\begin{align}\label{Galerkin:pressure:H2}
\norm{p}_{H^{2}} \leq C \left ( \norm{p}_{H^{1}} + \norm{\div ((\mu + \chi \sigma)\nabla \varphi)}_{L^{2}} +  K^{-1} \norm{\Gamma_{\bm{v}}}_{L^{2}} \right ).
\end{align}
Furthermore, from \eqref{pressure:divRHSL2}, we see that
\begin{align*}
\int_{0}^{T} \norm{\div ((\mu + \chi \sigma) \nabla \varphi)}_{L^{2}}^{\frac{8}{7}} \dt & \leq C\norm{\varphi}_{L^{\infty}(H^{1})}^{\frac{2}{7}} \int_{0}^{T} \norm{\mu + \chi \sigma}_{H^{1}}^{\frac{8}{7}} \norm{\varphi}_{H^{3}}^{\frac{6}{7}} \dt \\
& \leq  C \norm{\varphi}_{L^{\infty}(H^{1})}^{\frac{2}{7}} \norm{\mu + \chi \sigma}_{L^{2}(H^{1})}^{\frac{8}{7}} \norm{\varphi}_{L^{2}(H^{3})}^{\frac{6}{7}},
\end{align*}
and so for some positive constant $C$ depending only on $\Omega$, it holds that
\begin{equation}\label{Galerkin:pressure:L8/7H2}\begin{aligned}
\int_{0}^{T} \norm{p}_{H^{2}}^{\frac{8}{7}} \dt & \leq CT^{\frac{3}{7}} K^{-\frac{8}{7}} \norm{\Gamma_{\bm{v}}}_{L^{2}(L^{2})}^{\frac{8}{7}} + C T^{\frac{2}{7}} \norm{p}_{L^{\frac{8}{5}}(H^{1})}^{\frac{8}{7}} \\
& + C \norm{\varphi}_{L^{\infty}(H^{1})}^{\frac{2}{7}} \norm{\mu + \chi \sigma}_{L^{2}(H^{1})}^{\frac{8}{7}} \norm{\varphi}_{L^{2}(H^{3})}^{\frac{6}{7}}.
\end{aligned}
\end{equation}
For the velocity field $\bm{v}$ we estimate as follows.  Let $1 \leq i, j \leq 3$ be fixed, we obtain from \eqref{GN:D2fL3andnablafLinfty},
\begin{equation}\label{Galerkin:velo:H1}
\begin{aligned}
\norm{D_{i} v_{j}}_{L^{2}} & = K\norm{D_{i} D_{j} p - D_{i}(\mu + \chi \sigma) D_{j} \varphi - (\mu + \chi \sigma) D_{i} D_{j} \varphi}_{L^{2}} \\
& \leq K \left ( \norm{p}_{H^{2}} + \norm{\nabla (\mu + \chi \sigma)}_{\bm{L}^{2}} \norm{\nabla \varphi}_{\bm{L}^{\infty}} + \norm{\mu + \chi \sigma}_{L^{6}} \norm{D^{2}\varphi}_{L^{3}} \right ) \\
& \leq K \left ( \norm{p}_{H^{2}} + C \norm{\mu + \chi \sigma}_{H^{1}} \norm{\varphi}_{H^{3}}^{\frac{3}{4}} \norm{\varphi}_{H^{1}}^{\frac{1}{4}} \right ).
\end{aligned}
\end{equation}
Applying the same calculation as in \eqref{Galerkin:pressure:L8/7H2} yields
\begin{align*}
 \int_{0}^{T} \norm{ \der \bm{v}}_{\bm{L}^{2}}^{\frac{8}{7}} \dt & \leq C K \int_{0}^{T} \norm{p}_{H^{2}}^{\frac{8}{7}} + \norm{\mu + \chi \sigma}_{H^{1}}^{\frac{8}{7}} \norm{\varphi}_{H^{3}}^{\frac{6}{7}} \norm{\varphi}_{H^{1}}^{\frac{2}{7}} \dt \\
& \leq CK \left ( \norm{p}_{L^{\frac{8}{7}}(H^{2})}^{\frac{8}{7}} + \norm{\mu + \chi \sigma}_{L^{2}(H^{1})}^{\frac{8}{7}} \norm{\varphi}_{L^{2}(H^{3})}^{\frac{6}{7}} \norm{\varphi}_{L^{\infty}(H^{1})}^{\frac{2}{7}} \right ),
\end{align*}
for some positive constant $C$ depending only on $\Omega$.
\end{proof}

\subsubsection{Source term from the Darcy system}\label{sec:SourcetermDarcy}
To estimate the third source term
\begin{align*}
\int_{\Omega} \Gamma_{\bm{v}} \left ( p_{k} - \mu_{k} \varphi_{k} - \frac{D}{2} \abs{\sigma_{k}}^{2} \right ) \dx = \int_{\Omega} \Gamma_{\bm{v}} \left ( p_{k} - \mean{\mu_{k}} \varphi_{k} + (\mu_{k} - \mean{\mu_{k}}) \varphi_{k} - \frac{D}{2} \abs{\sigma_{k}}^{2} \right ) \dx
\end{align*}
of the energy equality we use H\"{o}lder's inequality to obtain
\begin{align*}
\abs{\int_{\Omega} \Gamma_{\bm{v}} \frac{D}{2} \abs{\sigma_{k}}^{2} + \Gamma_{\bm{v}} \varphi_{k} (\mu_{k} - \mean{\mu_{k}}) \dx} & \leq \frac{D}{2}\norm{\Gamma_{\bm{v}}}_{L^{2}} \norm{\sigma_{k}}_{L^{4}}^{2} \\
&  + \norm{\Gamma_{\bm{v}}}_{L^{\frac{3}{2}}} \norm{\mu_{k} - \mean{\mu_{k}}}_{L^{6}} \norm{\varphi_{k}}_{L^{6}}.
\end{align*}
By the Gagliardo--Nirenberg inequality \eqref{GagNirenIneq} with $j = 0$, $r = 2$, $m = 1$, $p = 4$, $q = 2$ and $\alpha = \frac{3}{4}$, we have
\begin{align*}
\norm{\sigma_{k}}_{L^{4}}^{2} \leq C \norm{\sigma_{k}}_{H^{1}}^{\frac{3}{2}} \norm{\sigma_{k}}_{L^{2}}^{\frac{1}{2}} = C \left ( \norm{\sigma_{k}}_{L^{2}}^{2} + \norm{\sigma_{k}}_{L^{2}}^{\frac{1}{2}} \norm{\nabla \sigma_{k}}_{\bm{L}^{2}}^{\frac{3}{2}} \right ).
\end{align*}
By Young's inequality with H\"{o}lder exponents (i.e., $ab \leq \frac{\eps}{p} a^{p} + \frac{\eps^{-q/p}}{q} b^{q}$ for $\frac{1}{p} + \frac{1}{q} = 1$ and $\eps > 0$), we find that
\begin{align*}
\frac{D}{2} \norm{\Gamma_{\bm{v}}}_{L^{2}} \norm{\sigma_{k}}_{L^{4}}^{2} \leq C \left (\norm{\Gamma_{\bm{v}}}_{L^{2}} \norm{\sigma_{k}}_{L^{2}}^{2} + \norm{\Gamma_{\bm{v}}}_{L^{2}}^{4} \norm{\sigma_{k}}_{L^{2}}^{2} \right ) + \frac{n_{0} D^{2}}{4} \norm{\nabla \sigma_{k}}_{\bm{L}^{2}}^{2},
\end{align*}
for some positive constant $C$ depending only in $n_{0}$, $D$ and $\Omega$.  Then, by \eqref{muL6Poin} we have
\begin{align*}
& \abs{\int_{\Omega} \Gamma_{\bm{v}} \frac{D}{2} \abs{\sigma_{k}}^{2} + \Gamma_{\bm{v}} \varphi_{k} (\mu_{k} - \mean{\mu_{k}}) \dx} \\
& \quad \leq \frac{n_{0} D^{2}}{4} \norm{\nabla \sigma_{k}}_{\bm{L}^{2}}^{2} + C \left ( 1 + \norm{\Gamma_{\bm{v}}}_{L^{2}}^{4} \right ) \norm{\sigma_{k}}_{L^{2}}^{2} + \norm{\Gamma_{\bm{v}}}_{L^{\frac{3}{2}}} c(C_{p}, C_{\mathrm{Sob}}) \norm{\nabla \mu_{k}}_{\bm{L}^{2}} \norm{\varphi_{k}}_{H^{1}} \\
& \quad \leq \frac{n_{0} D^{2}}{4} \norm{\nabla \sigma_{k}}_{\bm{L}^{2}}^{2} + C \left ( 1 + \norm{\Gamma_{\bm{v}}}_{L^{2}}^{4} \right ) \norm{\sigma_{k}}_{L^{2}}^{2} + \frac{m_{0}}{8} \norm{\nabla \mu_{k}}_{\bm{L}^{2}}^{2} + C\norm{\Gamma_{\bm{v}}}_{L^{2}}^{2}\norm{\varphi_{k}}_{H^{1}}^{2},
\end{align*}
where the positive constant $C$ depends only on $\Omega$, $m_{0}$, $n_{0}$ and $D$.  Here we point out that the assumption $\Gamma_{\bm{v}} \in L^{4}(0,T;L^{2}_{0})$ is needed.  For the remainder term $\Gamma_{\bm{v}} (p_{k} - \mean{\mu_{k}} \varphi_{k})$, we find that
\begin{align*}
& p_{k} - \mean{\mu_{k}} \varphi_{k} \\
& \quad = \left ( \left (-\Laplace_{N} \right )^{-1} \left ( \frac{1}{K} \Gamma_{\bm{v}} - \div ( (\mu_{k} - \mean{\mu_{k}} + \chi \sigma_{k}) \nabla \varphi_{k}) - \mean{\mu_{k}} \div \nabla( \varphi_{k} - \mean{\varphi_{k}}) \right )  \right )  - \mean{\mu_{k}} \varphi_{k} \\
& \quad = \left ( \left (-\Laplace_{N} \right )^{-1} \left ( \frac{1}{K} \Gamma_{\bm{v}} - \div ( (\mu_{k} - \mean{\mu_{k}} + \chi \sigma_{k}) \nabla \varphi_{k}) \right ) \right ) - \mean{\mu_{k}} \; \mean{\varphi_{k}} ,
\end{align*}
where we used
\begin{align*}
(-\Laplace_{N})^{-1} \left ( - \overline{\mu_{k}} \div \nabla (\varphi_{k} - \overline{\varphi_{k}} )\right ) = \overline{\mu_{k}} (\varphi_{k} - \overline{\varphi_{k}}).
\end{align*}
Then, by
\begin{align*}
\int_{\Omega} \Gamma_{\bm{v}} \mean{\varphi_{k}} \; \mean{\mu_{k}} \dx = \mean{\mu_{k}} \; \mean{\varphi_{k}} \int_{\Omega} \Gamma_{\bm{v}} \dx = 0,
\end{align*}
it holds that
\begin{align*}
\int_{\Omega} \Gamma_{\bm{v}} (p_{k} - \mean{\mu_{k}} \varphi_{k}) \dx = \int_{\Omega} \Gamma_{\bm{v}} \left ( \left (-\Laplace_{N} \right )^{-1} \left ( \frac{1}{K} \Gamma_{\bm{v}} - \div ( (\mu_{k} - \mean{\mu_{k}} + \chi \sigma_{k}) \nabla \varphi_{k}) \right ) \right ).
\end{align*}
Applying the calculations in the proof of Lemma \ref{lem:pressure:est} (specifically \eqref{NeumannEll:basic:Est}, \eqref{InvLaplDiv:L2:L6/5} and \eqref{muL6Poin}), H\"{o}lder's inequality and Young's inequality, we find that
\begin{align*}
& \abs{\int_{\Omega} \Gamma_{\bm{v}} (p_{k} - \mean{\mu_{k}} \varphi_{k}) \dx} \\
& \quad \leq \frac{C}{K} \norm{\Gamma_{\bm{v}}}_{L^{2}}^{2} + C \norm{\Gamma_{\bm{v}}}_{L^{2}} \left ( \norm{\nabla \mu_{k}}_{\bm{L}^{2}} + \chi \norm{\sigma_{k}}_{H^{1}} \right ) \norm{\nabla \varphi_{k}}_{\bm{L}^{\frac{3}{2}}} \\
& \quad \leq \frac{C}{K} \norm{\Gamma_{\bm{v}}}_{L^{2}}^{2} + \frac{m_{0}}{8} \norm{\nabla \mu_{k}}_{\bm{L}^{2}}^{2} + \frac{n_{0} D^{2}}{4} \norm{\nabla \sigma_{k}}_{\bm{L}^{2}}^{2} + \norm{\sigma_{k}}_{L^{2}}^{2}  + C (1 + \chi^{2}) \norm{\Gamma_{\bm{v}}}_{L^{2}}^{2} \norm{\nabla \varphi_{k}}_{\bm{L}^{2}}^{2},
\end{align*}
where $C$ is a positive constant depending only on $\abs{\Omega}$, $C_{p}$, $C_{\mathrm{Sob}}$, $D$, $n_{0}$ and $m_{0}$.  Here we point out that if we applied \eqref{Alt:pressure:L2:nabalvarphi:L3} instead of \eqref{pressure:L2:est} then we obtain a term containing $\norm{\nabla \varphi}_{\bm{L}^{3}}$ on the right-hand side and this cannot be controlled by the left-hand side of \eqref{apriorienergy:identity}.  Using \eqref{assump:lowerbddPsi} we have
\begin{align}\label{varphiL2norm}
\norm{\varphi_{k}}_{L^{2}}^{2} \leq \frac{1}{R_{1}} \norm{\Psi(\varphi_{k})}_{L^{1}} + \frac{R_{2}}{R_{1}} \abs{\Omega}.
\end{align}
Then, we obtain the following estimate
\begin{equation}\label{sourceterm:Gammabmv:prescribed:est}
\begin{aligned}
& \abs{\int_{\Omega} \Gamma_{\bm{v}} \left ( p_{k} - \mu_{k} \varphi_{k} - \frac{D}{2} \abs{\sigma_{k}}^{2} \right ) \dx } \\
 & \quad \leq C \left (\frac{1}{K} \norm{\Gamma_{\bm{v}}}_{L^{2}}^{2} + 1 \right ) + \frac{n_{0} D^{2}}{2} \norm{\nabla \sigma_{k}}_{\bm{L}^{2}}^{2} + \frac{m_{0}}{4} \norm{\nabla \mu_{k}}_{\bm{L}^{2}}^{2} \\
& \quad + C \left (1 + \norm{\Gamma_{\bm{v}}}_{L^{2}}^{4} \right )\norm{\sigma_{k}}_{L^{2}}^{2} + C \norm{\Psi(\varphi_{k})}_{L^{1}} + C(1 + \chi^{2}) \norm{\Gamma_{\bm{v}}}_{L^{2}}^{2} \norm{\nabla \varphi_{k}}_{\bm{L}^{2}}^{2},
\end{aligned}
\end{equation}
for some positive constant $C$ depending only on $R_{1}$, $R_{2}$, $\Omega$, $m_{0}$, $n_{0}$ and $D$.  Here we point out that it is crucial for the source term $\Gamma_{\bm{v}}$ to be prescribed and is not a function of $\varphi$, $\mu$ and $\sigma$, otherwise the product term $\norm{\Gamma_{\bm{v}}}_{L^{2}}^{4} \norm{\sigma_{k}}_{L^{2}}^{2}$ and $\norm{\Gamma_{\bm{v}}}_{L^{2}}^{2} \norm{\nabla \varphi_{k}}_{\bm{L}^{2}}^{2}$ cannot be controlled in the absence of any a priori estimates.  For the remaining source term
\begin{align*}
\int_{\Omega} \Gamma_{\varphi,k} \mu_{k} - \mathcal{S}_{k} N_{,\sigma}^{k}  \dx
\end{align*}
we split the analysis into two cases and combine with \eqref{sourceterm:Gammabmv:prescribed:est} to derive an energy inequality.

\subsubsection{Energy inequality for non-negative $\Theta_{\varphi}$}\label{sec:energyineq:nonnegativeThetavarphi}
Suppose $\Theta_{\varphi}$ is non-negative and bounded, and $\Psi$ is a potential that satisfies \eqref{assump:Psiquadratic}.  We will estimate the mean of $\mu_{k}$ by setting $j = 1$ in \eqref{Galerkin:mu}, and using the growth condition \eqref{assump:Psiquadratic} to obtain
\begin{align*}
\abs{\int_{\Omega} \mu_{k} \dx}^{2} & = \abs{\int_{\Omega} A \Psi'(\varphi_{k}) - \chi \sigma_{k} \dx}^{2} \leq 2A^{2} \norm{\Psi'(\varphi_{k})}_{L^{1}}^{2} + 2 \chi^{2} \norm{\sigma_{k}}_{L^{1}}^{2} \\
& \leq 2 A^{2} R_{4}^{2} \left ( \abs{\Omega} + \abs{\Omega}^{\frac{1}{2}} \norm{\varphi_{k}}_{L^{2}} \right )^{2} + 2 \chi^{2} \abs{\Omega} \norm{\sigma_{k}}_{L^{2}}^{2} \\
& \leq C(A, R_{4}, \abs{\Omega}) + 4 A^{2} R_{4}^{2} \abs{\Omega} \norm{\varphi_{k}}_{L^{2}}^{2} + 2 \chi^{2} \abs{\Omega} \norm{\sigma_{k}}_{L^{2}}^{2} .
\end{align*}
Then, by the Poincar\'{e} inequality \eqref{regular:Poincare} and the growth condition \eqref{assump:lowerbddPsi}, we find that
\begin{equation}\label{muL2norm}
\begin{aligned}
\norm{\mu_{k}}_{L^{2}}^{2} & \leq 2 C_{P}^{2} \norm{\nabla \mu_{k}}_{\bm{L}^{2}}^{2} + 2 \abs{\Omega} \abs{\mean{\mu_{k}}}^{2} \\
& \leq 2 C_{p}^{2} \norm{\nabla \mu_{k}}_{\bm{L}^{2}}^{2} + 8 A^{2} R_{4}^{2} \norm{\varphi_{k}}_{L^{2}}^{2} + 4 \chi^{2} \norm{\sigma_{k}}_{L^{2}}^{2} + C(A, R_{4}, \abs{\Omega}) \\
& \leq 2 C_{p}^{2} \norm{\nabla \mu_{k}}_{\bm{L}^{2}}^{2} + \frac{8 A^{2} R_{4}^{2}}{R_{1}} \norm{\Psi(\varphi_{k})}_{L^{1}} +  4  \chi^{2} \norm{\sigma_{k}}_{L^{2}}^{2} + C(A, R_{4}, R_{1}, R_{2}, \abs{\Omega}) .
\end{aligned}
\end{equation}
Note that by the specific form \eqref{assump:Sourcetermspecificform} for $\Gamma_{\varphi}$ we have that
\begin{align*}
\Gamma_{\varphi,k} \mu_{k} = \Lambda_{\varphi}(\varphi_{k}, \sigma_{k}) \mu_{k} - \Theta_{\varphi}(\varphi_{k}, \sigma_{k}) \abs{\mu_{k}}^{2} .
\end{align*}
Moving the non-negative term $\Theta_{\varphi}(\varphi_{k}, \sigma_{k}) \abs{\mu_{k}}^{2}$ to the left-hand side of \eqref{apriorienergy:identity} and subsequently neglecting it, we estimate the remainder using the growth condition (\ref{assump:Sourcetermgrowth}) and H\"{o}lder's inequality as follows (here we use the notation $\Lambda_{\varphi,k} := \Lambda_{\varphi}(\varphi_{k}, \sigma_{k})$),
\begin{equation}\label{GammavarphiSNsigma:remainder}
\begin{aligned}
&  \abs{\int_{\Omega} \Lambda_{\varphi,k} \mu_{k} - \mathcal{S}_{k}(D \sigma_{k} + \chi(1 - \varphi_{k})) \dx} \\
& \quad \leq \norm{\Lambda_{\varphi,k}}_{L^{2}} \norm{\mu_{k}}_{L^{2}} + \left (\norm{\Lambda_{S,k}}_{L^{2}} + R_{0} \norm{\mu_{k}}_{L^{2}} \right )\norm{D \sigma_{k} + \chi(1-\varphi_{k})}_{L^{2}} \\
& \quad \leq C \left (1 + \chi + (1+\chi) \norm{\varphi_{k}}_{L^{2}} + (1+D)\norm{\sigma_{k}}_{L^{2}} \right ) \norm{\mu_{k}}_{L^{2}} \\
& \quad + C \left (1 + \norm{\varphi_{k}}_{L^{2}} + \norm{\sigma}_{L^{2}} \right ) \left ( \chi \abs{\Omega}^{\frac{1}{2}} + D \norm{\sigma}_{L^{2}} + \chi \norm{\varphi_{k}}_{L^{2}} \right )
\end{aligned}
\end{equation}
where $C$ is a positive constant depending only on $R_{0}$ and $\abs{\Omega}$.  By Young's inequality, \eqref{muL2norm} and \eqref{varphiL2norm}, we have
\begin{equation}\label{nonnegaTheta:sourceterm:apriori:1}
\begin{aligned}
& \abs{\int_{\Omega} \Lambda_{\varphi,k} \mu_{k} - \mathcal{S}_{k}(D \sigma_{k} + \chi(1 - \varphi_{k})) \dx} \\
& \quad \leq \frac{m_{0}}{8 C_{p}^{2}} \norm{\mu_{k}}_{L^{2}}^{2} +  C(1+ \chi + D +\chi^{2}) \norm{\varphi_{k}}_{L^{2}}^{2} + C(1+ \chi + D)^{2} \norm{\sigma_{k}}_{L^{2}}^{2} +   C(1 + \chi + \chi^{2}) \\
& \quad \leq  \frac{m_{0}}{4} \norm{\nabla \mu_{k}}_{\bm{L}^{2}}^{2} + C(1 + \chi^{2}) \norm{\sigma_{k}}_{L^{2}}^{2} + C(1 + \chi^{2}) \norm{\Psi(\varphi_{k})}_{L^{1}} + C(1 + \chi^{2}),
\end{aligned}
\end{equation}
for some positive constant $C$ depending only on $\abs{\Omega}$, $R_{0}$, $R_{1}$, $R_{2}$, $R_{4}$, $A$, $D$, $C_{p}$ and $m_{0}$.  Using the fact that
\begin{align*}
\norm{D \nabla \sigma}_{\bm{L}^{2}} \leq \norm{\nabla (D \sigma + \chi (1-\varphi))}_{\bm{L}^{2}} + \norm{\chi \nabla \varphi}_{\bm{L}^{2}},
\end{align*}
we now estimate the right-hand side of \eqref{apriorienergy:identity} using \eqref{boundryintegralRHS}, \eqref{sourceterm:Gammabmv:prescribed:est} and \eqref{nonnegaTheta:sourceterm:apriori:1}, which leads to
\begin{equation}\label{nonnegative:Thetavarphi:apriori:est:combined}
\begin{aligned}
& \frac{\dd}{\dt} \int_{\Omega} A \Psi(\varphi_{k}) + \frac{B}{2} \abs{\nabla \varphi_{k}}^{2} + \frac{D}{2} \abs{\sigma_{k}}^{2} + \chi \sigma_{k}(1-\varphi_{k}) \dx \\
& \quad  +  \frac{m_{0}}{2} \norm{\nabla \mu_{k}}_{\bm{L}^{2}}^{2} + \frac{n_{0} D^{2}}{2} \norm{\nabla \sigma_{k}}_{\bm{L}^{2}}^{2} + \frac{1}{K} \norm{\bm{v}_{k}}_{\bm{L}^{2}}^{2}  + \frac{Db}{2} \norm{\sigma_{k}}_{L^{2}(\pd \Omega)}^{2} \\
& \quad \leq C(1 + b)(1 + \chi^{2} ) \norm{\Psi(\varphi_{k})}_{L^{1}} + C \left ( \norm{\Gamma_{\bm{v}}}_{L^{2}}^{2} + b \right )(1 + \chi^{2}) \norm{\nabla \varphi_{k}}_{\bm{L}^{2}}^{2}\\
& \quad +  C \left (1 + \chi^{2} + \norm{\Gamma_{\bm{v}}}_{L^{2}}^{4} \right ) \norm{\sigma_{k}}_{L^{2}}^{2} + C(1+b)(1 + \chi^{2}) + \frac{C}{K} \norm{\Gamma_{\bm{v}}}_{L^{2}}^{2} \\
& \quad  + bC (1+\chi^{2}) \norm{\sigma_{\infty}}_{L^{2}(\pd \Omega)}^{2},
\end{aligned}
\end{equation}
for some positive constant $C$ not depending on $\Gamma_{\bm{v}}$, $K$, $b$ and $\chi$.  Integrating \eqref{nonnegative:Thetavarphi:apriori:est:combined} with respect to $t$ from $0$ to $s \in (0,T]$ leads to
\begin{equation}\label{nonnegative:Thetavarphi:apriori:integrated}
\begin{aligned}
& A \norm{\Psi(\varphi_{k}(s))}_{L^{1}} + \frac{B}{2} \norm{\nabla \varphi_{k}(s)}_{\bm{L}^{2}}^{2} + \frac{D}{2} \norm{\sigma_{k}(s)}_{L^{2}}^{2} + \int_{\Omega} \chi \sigma_{k}(s)(1-\varphi_{k}(s)) \dx \\
& \quad + \int_{0}^{s} \frac{m_{0}}{2} \norm{\nabla \mu_{k}}_{\bm{L}^{2}}^{2} + \frac{n_{0} D^{2}}{2} \norm{\nabla \sigma_{k}}_{\bm{L}^{2}}^{2} + \frac{1}{K} \norm{\bm{v}_{k}}_{\bm{L}^{2}}^{2} + \frac{Db}{2} \norm{\sigma_{k}}_{L^{2}(\pd \Omega)}^{2} \dt \\
& \quad \leq \int_{0}^{s} C(1+b)(1+\chi^{2}) \left (1 + \norm{\Gamma_{\bm{v}}}_{L^{2}}^{4} \right ) \left (  \norm{\Psi(\varphi_{k})}_{L^{1}} + \norm{\nabla \varphi_{k}}_{\bm{L}^{2}}^{2} + \norm{\sigma_{k}}_{L^{2}}^{2} \right )\dt \\
& \quad + C(1+b)(1 + \chi^{2})T + \frac{C}{K}\norm{\Gamma_{\bm{v}}}_{L^{2}(0,T;L^{2})}^{2} + Cb(1+\chi^{2}) \norm{\sigma_{\infty}}_{L^{2}(0,T;L^{2}(\pd \Omega))}^{2}\\
& \quad + C \left ( \norm{\Psi(\varphi_{0})}_{L^{1}} + \norm{\varphi_{0}}_{H^{1}}^{2} + \norm{\sigma_{0}}_{L^{2}}^{2} \right ),
\end{aligned}
\end{equation}
for some positive constant $C$ independent of $\Gamma_{\bm{v}}$, $K$, $\chi$ and $b$.  Here we used $\sigma_{0} \in L^{2}$ and $\varphi_{0} \in H^{1}$, which implies by  the growth condition \eqref{assump:Psiquadratic} that $\Psi(\varphi_{0}) \in L^{1}$.  Next, by H\"{o}lder's inequality and Young's inequality we have \begin{equation}\label{chemotaxisterm:energy}
\begin{aligned}
& \abs{\int_{\Omega} \chi \sigma_{k}(x,s) (1-\varphi_{k}(x,s)) \dx} \leq  \frac{2D}{8} \norm{\sigma_{k}(s)}_{L^{2}}^{2} + \frac{2 \chi^{2} \abs{\Omega}}{D} + \frac{2 \chi^{2}}{D} \norm{\varphi_{k}(s)}_{L^{2}}^{2} \\
& \quad \leq \frac{D}{4} \norm{\sigma_{k}(s)}_{L^{2}}^{2} + \frac{2 \chi^{2}}{D R_{1}} \norm{\Psi(\varphi_{k}(s))}_{L^{1}}  + \frac{2 \chi^{2} \abs{\Omega}}{D} (1 + R_{2}).
\end{aligned}
\end{equation}
Substituting \eqref{chemotaxisterm:energy} into \eqref{nonnegative:Thetavarphi:apriori:integrated} then yields
\begin{equation}\label{PreGronwall:nonnegativeThetavarphi:energyineq}
\begin{aligned}
\min & \left ( A - \frac{2 \chi^{2}}{D R_{1}}, \frac{B}{2}, \frac{D}{4} \right ) \left ( \norm{\Psi(\varphi_{k}(s))}_{L^{1}(\Omega)} + \norm{\nabla \varphi_{k}(s)}_{\bm{L}^{2}(\Omega)}^{2} +  \norm{\sigma_{k}(s)}_{L^{2}(\Omega)}^{2} \right ) \\
& + \int_{0}^{s} \norm{\nabla \mu_{k}}_{\bm{L}^{2}}^{2} + \norm{\nabla \sigma_{k}}_{\bm{L}^{2}}^{2} + \frac{1}{K} \norm{\bm{v}_{k}}_{\bm{L}^{2}}^{2} + \frac{b}{2} \norm{\sigma_{k}}_{L^{2}(\pd \Omega)}^{2} \dt \\
& \leq  \int_{0}^{s} C(1+b)(1+\chi^{2}) \left (1 + \norm{\Gamma_{\bm{v}}}_{L^{2}}^{4} \right ) \left ( \norm{\Psi(\varphi_{k})}_{L^{1}} + \norm{\nabla \varphi_{k}}_{\bm{L}^{2}}^{2} + \norm{\sigma_{k}}_{L^{2}}^{2} \right ) \dt \\
& + C(1+b)(1 + \chi^{2})(1+T) + \frac{C}{K}\norm{\Gamma_{\bm{v}}}_{L^{2}(0,T;L^{2})}^{2},
\end{aligned}
\end{equation}
for some positive constant $C$ independent of $\Gamma_{\bm{v}}$, $K$, $b$ and $\chi$.  Setting
\begin{equation}\label{energyineq:constants}
\begin{aligned}
\alpha & :=  C(1+b)(1 + \chi^{2})(1+T) + \frac{C}{K} \norm{\Gamma_{\bm{v}}}_{L^{2}(0,T;L^{2})}^{2} , \\
\beta & := C(1+b)(1+\chi^{2}) \left ( 1 + \norm{\Gamma_{\bm{v}}}_{L^{2}}^{4} \right ) \in L^{1}(0,T),
\end{aligned}
\end{equation}
and noting that
\begin{align*}
\alpha \left ( 1 + \int_{0}^{s} \beta(t) \exp \left ( \int_{0}^{t} \beta(r) \dr \right ) \dt \right ) \leq \alpha \left ( 1 + \norm{\beta}_{L^{1}(0,T)} \exp \left ( \norm{\beta}_{L^{1}(0,T)} \right ) \right ) < \infty,
\end{align*}
we find after applying the Gronwall inequality \eqref{Gronwall} to \eqref{PreGronwall:nonnegativeThetavarphi:energyineq} leads to
\begin{equation}\label{nonnegativeThetavarphi:energyineq}
\begin{aligned}
& \sup_{s \in (0,T]} \left ( \norm{\Psi(\varphi_{k}(s))}_{L^{1}} + \norm{\nabla \varphi_{k}(s)}_{\bm{L}^{2}}^{2} +  \norm{\sigma_{k}(s)}_{L^{2}}^{2} \right ) \\
& \quad + \int_{0}^{T} \norm{\nabla \mu_{k}}_{\bm{L}^{2}}^{2} + \norm{\nabla \sigma_{k}}_{\bm{L}^{2}}^{2} + \frac{1}{K} \norm{\bm{v}_{k}}_{\bm{L}^{2}}^{2} + \frac{b}{2} \norm{\sigma_{k}}_{L^{2}(\pd \Omega)}^{2} \dt \leq \mathcal{E},
\end{aligned}
\end{equation}
where we recall that $\mathcal{E}$ denotes a constant that is uniformly bounded for $b, \chi \in (0,1]$ and is also uniformly bounded for $K \in (0,1]$ when $\Gamma_{\bm{v}} = 0$.

\subsubsection{Energy inequality for positive $\Theta_{\varphi}$}
Suppose $\Theta_{\varphi}$ satisfies \eqref{assump:TheatvarphiPositive} and $\Psi$ is a potential satisfying the growth condition \eqref{assump:PsiSuperquadratic}.  Similar to the previous case, we see that the specific form for $\Gamma_{\varphi}$ leads to
\begin{align*}
\Gamma_{\varphi,k} \mu_{k} = \Lambda_{\varphi}(\varphi_{k}, \sigma_{k}) \mu_{k} - \Theta_{\varphi}(\varphi_{k}, \sigma_{k}) \abs{\mu_{k}}^{2} .
\end{align*}
We move the term $\Theta_{\varphi}(\varphi_{k}, \sigma_{k}) \abs{\mu_{k}}^{2}$ to the left-hand side of \eqref{apriorienergy:identity} and estimate the remainder as in \eqref{GammavarphiSNsigma:remainder}.  Using Young's inequality differently and also \eqref{varphiL2norm}, we have
\begin{equation}\label{positiveThetavarphi:GammavarphiSNsigma:reaminder}
\begin{aligned}
& \abs{\int_{\Omega} \Lambda_{\varphi,k} \mu_{k} - \mathcal{S}_{k}(D \sigma_{k} + \chi(1 - \varphi_{k})) \dx} \\
& \quad \leq \frac{R_{5}}{2} \norm{\mu_{k}}_{L^{2}}^{2} + C(1+ \chi + D +\chi^{2}) \norm{\varphi_{k}}_{L^{2}}^{2} + C(1+ \chi + D)^{2} \norm{\sigma_{k}}_{L^{2}}^{2} + C(1 + \chi + \chi^{2}) \\
& \quad \leq \frac{R_{5}}{2} \norm{\mu_{k}}_{L^{2}}^{2} + C(1 + \chi^{2}) \norm{\sigma_{k}}_{L^{2}}^{2} + C(1 + \chi^{2}) \norm{\Psi(\varphi_{k})}_{L^{1}} + C(1 + \chi^{2}),
\end{aligned}
\end{equation}
for some positive constant $C$ depending only on $\abs{\Omega}$, $R_{5}$, $R_{1}$, $R_{2}$, $A$, $D$, and $C_{p}$.  Using \eqref{boundryintegralRHS}, \eqref{sourceterm:Gammabmv:prescribed:est}, \eqref{positiveThetavarphi:GammavarphiSNsigma:reaminder} and the lower bound $\Theta_{\varphi} \geq R_{5}$, instead of \eqref{nonnegative:Thetavarphi:apriori:est:combined} we obtain from \eqref{apriorienergy:identity}
\begin{equation}
\begin{aligned}
& \frac{\dd}{\dt} \int_{\Omega}  A \Psi(\varphi_{k}) + \frac{B}{2} \abs{\nabla \varphi_{k}}^{2} + \frac{D}{2} \abs{\sigma_{k}}^{2} + \chi \sigma_{k}(1-\varphi_{k}) \dx \\
& \quad + \frac{R_{5}}{2} \norm{\mu_{k}}_{L^{2}}^{2} +   \frac{m_{0}}{2} \norm{\nabla \mu_{k}}_{\bm{L}^{2}(\Omega)}^{2} + \frac{n_{0} D^{2}}{2} \norm{\nabla \sigma_{k}}_{\bm{L}^{2}}^{2}  + \frac{1}{K} \norm{\bm{v}_{k}}_{\bm{L}^{2}}^{2}  + \frac{Db}{2} \norm{\sigma_{k}}_{L^{2}(\pd \Omega)}^{2} \\
& \quad \leq C(1 + b)(1 + \chi^{2}) \norm{\Psi(\varphi_{k})}_{L^{1}} + C \left ( \norm{\Gamma_{\bm{v}}}_{L^{2}}^{2} + b \right )(1 + \chi^{2}) \norm{\nabla \varphi_{k}}_{\bm{L}^{2}}^{2}\\
& \quad + C \left (1 + \chi^{2} + \norm{\Gamma_{\bm{v}}}_{L^{2}}^{4} \right ) \norm{\sigma_{k}}_{L^{2}}^{2} + C(1+b)(1 + \chi^{2}) + \frac{C}{K} \norm{\Gamma_{\bm{v}}}_{L^{2}}^{2} \\
& \quad + Cb(1+\chi^{2}) \norm{\sigma_{\infty}}_{L^{2}(\pd \Omega)}^{2},
\end{aligned}
\end{equation}
for some positive constant $C$ independent of $\Gamma_{\bm{v}}$, $K$, $b$ and $\chi$.  We point out the main difference between \eqref{nonnegative:Thetavarphi:apriori:est:combined} and the above is the appearance of the term $\frac{R_{5}}{2} \norm{\mu_{k}}_{L^{2}}^{2}$ on the left-hand side.  The positivity of $\Theta_{\varphi}$ allows us to absorb the $\norm{\mu_{k}}_{L^{2}}^{2}$ term on the right-hand side of \eqref{positiveThetavarphi:GammavarphiSNsigma:reaminder} and thus we do not need to use \eqref{muL2norm}, which was the main reason why $\Psi$ has to be a quadratic potential for a non-negative $\Theta_{\varphi}$.  Then, applying a similar argument as in Section \ref{sec:energyineq:nonnegativeThetavarphi}, we arrive at an analogous energy inequality to \eqref{nonnegativeThetavarphi:energyineq},
\begin{equation}\label{positiveThetavarphi:energyineq}
\begin{aligned}
 \sup_{s \in (0,T]} & \left ( \norm{\Psi(\varphi_{k}(s))}_{L^{1}} + \norm{\nabla \varphi_{k}(s)}_{\bm{L}^{2}}^{2} +  \norm{\sigma_{k}(s)}_{L^{2}}^{2} \right ) \\
& + \int_{0}^{T} \norm{\mu_{k}}_{H^{1}}^{2} + \norm{\nabla \sigma_{k}}_{\bm{L}^{2}}^{2} + \frac{1}{K} \norm{\bm{v}_{k}}_{\bm{L}^{2}}^{2} + \frac{b}{2} \norm{\sigma_{k}}_{L^{2}(\pd \Omega)}^{2} \dt \leq \mathcal{E}.
\end{aligned}
\end{equation}
Using \eqref{muL2norm} and \eqref{varphiL2norm} applied to \eqref{nonnegativeThetavarphi:energyineq}, and similarly using \eqref{varphiL2norm} applied to \eqref{positiveThetavarphi:energyineq} we obtain
\begin{equation}\label{unified:energyineq}
\begin{aligned}
 \sup_{s \in (0,T]} & \left ( \norm{\Psi(\varphi_{k}(s))}_{L^{1}} + \norm{\varphi_{k}(s)}_{H^{1}}^{2} +  \norm{\sigma_{k}(s)}_{L^{2}}^{2} \right ) \\
& + \int_{0}^{T} \norm{\mu_{k}}_{H^{1}}^{2} + \norm{\nabla \sigma_{k}}_{\bm{L}^{2}}^{2} + \frac{1}{K} \norm{\bm{v}_{k}}_{\bm{L}^{2}}^{2} + \frac{b}{2} \norm{\sigma_{k}}_{L^{2}(\pd \Omega)}^{2} \dt \leq \mathcal{E}.
\end{aligned}
\end{equation}
This a priori estimate implies that the Galerkin ansatz $\varphi_{k}$, $\mu_{k}$, $\sigma_{k}$ and $\bm{v}_{k}$ can be extended to the interval $[0,T]$.  To determine if $p_{k}$ can also be extended to the interval $[0,T]$ we require some higher order estimates for $\varphi_{k}$ in order to use \eqref{pressure:L8/5H1bdd}.

\subsection{Higher order estimates}
Let $\Pi_{k}$ denote the orthogonal projection onto the finite-dimensional subspace $W_{k}$.  From \eqref{Galerkin:mu} we may view $\varphi_{k}$ as the solution to the following elliptic equation
\begin{subequations}\label{ellipticregularity:varphi}
\begin{alignat}{2}
-B \Laplace u  + u & = \mu_{k} - A \Pi_{k} \left ( \Psi'(u) \right ) + \chi \sigma_{k} + u && \text{ in } \Omega, \label{ellipticregular} \\
\pdnu u & = 0 && \text{ on } \pd \Omega.
\end{alignat}
\end{subequations}
For the case where $\Psi$ satisfies \eqref{assump:Psiquadratic}, as $\{\varphi_{k}\}_{k \in \N}$ is bounded in $L^{\infty}(0,T;H^{1})$, we have that $\{ \Psi'(\varphi_{k})\}_{k \in \N}$ is also bounded in $L^{\infty}(0,T;H^{1})$.  Using the fact that our basis functions $\{w_{i}\}_{i \in \N}$ are the eigenfunctions of the inverse Neumann-Laplacian operator and is therefore orthogonal in $H^{1}$, and the Sobolev embedding $H^{1} \subset L^{r}$ for $r \in [1,6]$, there exists a positive constant $C$ independent of $\varphi_{k}$ such that
\begin{align}\label{Projection:Psi}
\norm{\Pi_{k}(\Psi'(\varphi_{k}))}_{X} \leq C\norm{\Psi'(\varphi_{k})}_{X} \quad \text{ for } X = H^{1} \text{ or } L^{r}, 1 \leq r \leq 6.
\end{align}
Then, this implies that $\{\Pi_{k} \left ( \Psi'(\varphi_{k} ) \right )\}_{k \in \N}$ is also bounded in $L^{\infty}(0,T;H^{1})$.  As the right-hand side of \eqref{ellipticregular} belongs to $H^{1}$ for a.e. $t \in (0,T)$, and the boundary $\pd \Omega$ is $C^{3}$, by elliptic regularity theory, we have
\begin{align}\label{Galerkin:L2H3}
\norm{\varphi_{k}}_{L^{2}(H^{3})} \leq C \left (1 + \norm{\varphi_{k}}_{L^{2}(H^{1})} + \norm{\mu_{k} + \chi \sigma_{k}}_{L^{2}(H^{1})}  \right ) \leq \mathcal{E},
\end{align}
for some positive constant $C$ depending only on $\Omega$ and $R_{4}$.  For the case where $\Psi$ satisfies \eqref{assump:PsiSuperquadratic}, we employ a bootstrap argument from \cite[\S 3.3]{article:GarckeLamDirichlet}.  The growth assumption \eqref{assump:PsiSuperquadratic} implies that
\begin{align}\label{L2H3:Psi:growth:superlinear}
\abs{\Psi'(y)} \leq C \left ( 1 + \abs{y}^{m} \right ), \quad \abs{\Psi''(y)} \leq C \left ( 1 + \abs{y}^{m-1} \right ) \text{ for } m \in [1,5).
\end{align}
For fixed $m \in [1,5)$, we define a sequence of positive numbers $\{l_{j}\}_{j \in \N}$ by
\begin{align*}
l_{1}m \leq 6, \quad l_{j+1} = \frac{6 l_{j}}{6 - (5-m)l_{j}}.
\end{align*}
It can be shown that $\{l_{j}\}_{j \in \N}$ is a strictly increasing sequence such that $l_{j} \to \infty$ as $j \to \infty$.
The Gagliardo--Nirenberg inequality \eqref{GagNirenIneq} then yields the following continuous embedding
\begin{align}\label{Regularity:Bootstrap}
L^{2}(0,T;W^{2,l_{j}}) \cap L^{\infty}(0,T;L^{6}) \subset L^{2m}(0,T;L^{m l_{j+1}}).
\end{align}
At the first step, the boundedness of $\{\varphi_{k}\}_{k \in \N}$ in $L^{\infty}(0,T;H^{1})$ yields
\begin{align*}
\norm{\Pi_{k}(\Psi'(\varphi_{k}))}_{L^{l_{1}}}^{2} \leq C \left ( 1 + \norm{\varphi_{k}}_{L^{6}}^{2m} \right ),
\end{align*}
which implies that $\{\Pi_{k}(\Psi'(\varphi_{k})) \}_{k \in \N}$ is bounded in $L^{2}(0,T;L^{l_{1}})$.  As the other terms on the right-hand side of \eqref{ellipticregularity:varphi} are bounded in $L^{2}(0,T;H^{1})$, elliptic regularity then yields that $\{\varphi_{k}\}_{k \in \N}$ is bounded in $L^{2}(0,T;W^{2,l_{1}})$, and thus in $L^{2m}(0,T;L^{ml_{2}})$ by \eqref{Regularity:Bootstrap}.

At the $j$-th step, we have $\{\varphi_{k}\}_{k \in \N}$ is bounded in $L^{2}(0,T,W^{2,l_{j}}) \cap L^{2m}(0,T;L^{m l_{j+1}})$.  Then, it holds that
\begin{align*}
\norm{\Pi_{k}(\Psi'(\varphi_{k}))}_{L^{2}(L^{l_{j+1}})}^{2} \leq C \left ( 1 + \norm{\varphi_{k}}_{L^{2m}(L^{m l_{j+1}})}^{2m} \right ),
\end{align*}
and so $\{\Pi_{k}(\Psi'(\varphi_{k})) \}_{k \in \N}$ is bounded in $L^{2}(0,T;L^{l_{j+1}})$.  Elliptic regularity then implies that $\{\varphi_{k}\}_{k \in \N}$ is bounded in $L^{2}(0,T;W^{2,l_{j+1}})$.

We terminate the bootstrapping procedure once $l_{j} \geq 6$ for some $j \in \N$.  This occurs after a finite number of steps as $\lim_{j \to \infty} l_{j} = \infty $.  Altogether, we obtain that  $\{\varphi_{k}\}_{k \in \N}$ is bounded in $L^{2}(0,T;W^{2,6})$.  From \eqref{L2H3:Psi:growth:superlinear} it holds that
\begin{align*}
\abs{\Psi''(\varphi_{k}) \nabla \varphi_{k}}^{2} \leq C \left ( 1 + \abs{\varphi_{k}}^{2m-2} \right ) \abs{\nabla \varphi_{k}}^{2} \text{ for } m \in [1,5),
\end{align*}
and by the following continuous embeddings obtain from the Gagliardo--Nirenberg inequality \eqref{GagNirenIneq},
\begin{align*}
L^{2}(0,T;W^{2,6}) \cap L^{\infty}(0,T;H^{1}) \subset L^{2m}(0,T;L^{2m}) \cap L^{2m-2}(0,T;L^{\infty}) \text{ for } m \in [1,5),
\end{align*}
we find that $\{\Pi_{k}(\Psi'(\varphi_{k}))\}_{k \in \N}$ is bounded in $L^{2}(0,T;H^{1})$.  Applying elliptic regularity once more leads to the boundedness of $\{\varphi_{k}\}_{k \in \N}$ in $L^{2}(0,T;H^{3})$.  Consequently, the hypotheses of Lemma \ref{lem:reg:pressurevelocity} are satisfied and we obtain that
\begin{align*}
\norm{p_{k}}_{L^{\frac{8}{5}}(H^{1})} \leq \mathcal{E},
\end{align*}
which implies that the Galerkin ansatz $p_{k}$ can be extended to the interval $[0,T]$.

\subsection{Estimates for the convection terms and the time derivatives}\label{sec:convectionandtimederivative}
By the Gagliardo--Nirenberg inequality \eqref{GagNirenIneq} with $j = 0$, $p = \infty$, $m = 3$, $r = 2$, $q = 2$ and $d = 3$, we have
\begin{align*}
\norm{\varphi_{k}}_{L^{\infty}} \leq C \norm{\varphi_{k}}_{H^{3}}^{\frac{1}{4}} \norm{\varphi_{k}}_{L^{6}}^{\frac{3}{4}} \leq C \norm{\varphi_{k}}_{H^{3}}^{\frac{1}{4}} \norm{\varphi_{k}}_{H^{1}}^{\frac{3}{4}}.
\end{align*}
For any $\zeta \in L^{\frac{8}{3}}(0,T;H^{1})$ with coefficients $\{\zeta_{kj}\}_{1 \leq j \leq k} \subset \R^{k}$ such that $\Pi_{k} \zeta = \sum_{j=1}^{k} \zeta_{kj} w_{j}$, we can estimate
\begin{equation}\label{convectionvarphiv}
\begin{aligned}
 \abs{\int_{0}^{T} \int_{\Omega} \varphi_{k} \bm{v}_{k} \cdot \nabla \Pi_{k} \zeta \dx \dt} & \leq \int_{0}^{T} \norm{\bm{v}_{k}}_{\bm{L}^{2}} \norm{\varphi_{k}}_{L^{\infty}} \norm{\nabla \Pi_{k} \zeta}_{\bm{L}^{2}} \dt \\
& \leq  C \norm{\varphi_{k}}_{L^{\infty}(H^{1})}^{\frac{3}{4}} \norm{\bm{v}_{k}}_{L^{2}(\bm{L}^{2})} \norm{\varphi_{k}}_{L^{2}(H^{3})}^{\frac{1}{4}} \norm{\zeta}_{L^{\frac{8}{3}}(H^{1})}.
\end{aligned}
\end{equation}
Using \eqref{unified:energyineq} and \eqref{Galerkin:L2H3}, we find that
\begin{align}\label{divvarphikvk}
\norm{\div (\varphi_{k} \bm{v}_{k})}_{L^{\frac{8}{5}}((H^{1})^{*})} \leq K^{\frac{1}{2}} \mathcal{E}.
\end{align}
Next, multiplying \eqref{Galerkin:varphi} by $\zeta_{kj}$, summing from $j = 1$ to $k$ and then integrating in time from $0$ to $T$ leads to
\begin{align*}
\abs{\int_{0}^{T} \int_{\Omega} \pd_{t}\varphi_{k} \zeta \dx \dt}  & \leq \int_{0}^{T} m_{1} \norm{\nabla \mu_{k}}_{\bm{L}^{2}} \norm{\nabla \Pi_{k} \zeta}_{\bm{L}^{2}} \dt \\
& + \int_{0}^{T} \norm{\Gamma_{\varphi,k}}_{L^{2}} \norm{\Pi_{k} \zeta}_{L^{2}} + \norm{\div(\varphi_{k} \bm{v}_{k})}_{(H^{1})^{*}} \norm{\Pi_{k} \zeta}_{H^{1}} \dt.
\end{align*}
By \eqref{assump:Sourcetermspecificform},  \eqref{assump:ThetaLambda} and \eqref{unified:energyineq}, we find that
\begin{align*}
\norm{\Gamma_{\varphi,k}}_{L^{2}(L^{2})} \leq C(R_{0}, \abs{\Omega}, T) \left (1 + \norm{\varphi_{k}}_{L^{2}(L^{2})} + \norm{\mu_{k}}_{L^{2}(L^{2})} + \norm{\sigma_{k}}_{L^{2}(L^{2})} \right ) \leq \mathcal{E},
\end{align*}
and so, by H\"{o}lder's inequality, we find that
\begin{align*}
 \abs{\int_{0}^{T} \int_{\Omega} \pd_{t}\varphi_{k} \zeta \dx \dt} & \leq \left ( \mathcal{E} T^{\frac{1}{8}} + \norm{\div (\varphi_{k}\bm{v}_{k})}_{L^{\frac{8}{5}}((H^{1})^{*})} \right ) \norm{\zeta}_{L^{\frac{8}{3}}(H^{1})} .
\end{align*}
Taking the supremum over $\zeta \in L^{\frac{8}{3}}(0,T;H^{1})$ and using \eqref{unified:energyineq} and \eqref{divvarphikvk} yields that
\begin{align}\label{Galerkin:pdtvarphi}
\norm{\pd_{t}\varphi_{k}}_{L^{\frac{8}{5}}((H^{1})^{*})} \leq \mathcal{E} \left (1 + K^{\frac{1}{2}} \right ),
\end{align}
Similarly, by H\"{o}lder's inequality and the following Gagliardo--Nirenberg inequality \eqref{GagNirenIneq} with $j = 0$, $r = 2$, $m = 1$, $p = \frac{10}{3}$, $q = 2$ and $d = 3$,
\begin{align*}
\norm{f}_{L^{\frac{10}{3}}} \leq C \norm{f}_{H^{1}}^{\frac{3}{5}} \norm{f}_{L^{2}}^{\frac{2}{5}},
\end{align*}
which in turn implies that $\{\sigma_{k}\}_{k \in \N}$ is bounded uniformly in $L^{\frac{10}{3}}(Q)$.  Then, we find that for any $\zeta \in L^{5}(0,T;W^{1,5})$,
\begin{equation}\label{convectionsigmav}
\begin{aligned}
\abs{\int_{0}^{T} \int_{\Omega} \sigma_{k} \bm{v}_{k} \cdot \nabla \Pi_{k} \zeta \dx \dt} & \leq \int_{0}^{T} \norm{\sigma_{k}}_{L^{\frac{10}{3}}} \norm{\bm{v}_{k}}_{\bm{L}^{2}} \norm{\nabla \zeta}_{\bm{L}^{5}} \dt \\
& \leq \norm{\sigma_{k}}_{L^{\frac{10}{3}}(Q)} \norm{\bm{v}_{k}}_{L^{2}(\bm{L}^{2})} \norm{\nabla \zeta}_{L^{5}(\bm{L}^{5})},
\end{aligned}
\end{equation}
and
\begin{align}\label{divsigmakvk}
\norm{\div (\sigma_{k} \bm{v}_{k})}_{L^{\frac{5}{4}}((W^{1,5})^{*})} \leq K^{\frac{1}{2}} \mathcal{E}.
\end{align}
A similar calculation to \eqref{Galerkin:pdtvarphi} yields that
\begin{align}\label{Galerkin:pdtsigma}
\norm{\pd_{t}\sigma_{k}}_{L^{\frac{5}{4}}((W^{1,5})^{*})} \leq \mathcal{E} \left (1 + K^{\frac{1}{2}} \right ).
\end{align}

\begin{remark}
We may also use the Gagliardo--Nirenberg inequality to deduce that
\begin{align*}
\norm{f}_{L^{r}} \leq C \norm{f}_{H^{1}}^{\frac{3(r-2)}{2r}} \norm{f}_{L^{2}}^{\frac{6-r}{2r}} \text{ for any } r \in (2,6).
\end{align*}
Then, the computation \eqref{convectionsigmav} becomes
\begin{align*}
\abs{\int_{0}^{T} \int_{\Omega} \sigma_{k} \bm{v}_{k} \cdot \nabla \Pi_{k} \zeta \dx \dt} \leq C \norm{\bm{v_{k}}}_{L^{2}(\bm{L}^{2})} \norm{\sigma_{k}}_{L^{\infty}(L^{2})}^{\frac{6-r}{2r}} \norm{\sigma_{k}}_{L^{2}(H^{1})}^{\frac{3(r-2)}{2r}} \norm{\nabla \zeta}_{L^{\frac{4r}{6-r}}(\bm{L}^{\frac{2r}{r-2}})},
\end{align*}
which implies that $\{ \div (\sigma_{k} \bm{v}_{k}) \}_{k \in \N}$ and $\{ \pd_{t}\sigma_{k}\}_{k \in \N}$ are bounded uniformly in
\begin{align*}
L^{\frac{4r}{5r-6}}(0,T;(W^{1,\frac{2r}{r-2}})^{*}) \text{ for } r \in (2,6).
\end{align*}
Note that the temporal exponent decreases while the spatial exponent increases as $r$ increases, and they intersect at the point $r = \frac{10}{3}$.
\end{remark}

Here we point out that even with the improved regularity $\bm{v}_{k} \in L^{\frac{8}{7}}(0,T;\bm{H}^{1})$, we are unable to show $\div(\sigma_{k} \bm{v}_{k})$ is bounded in the dual space $(H^{1})^{*}$.  Indeed, let $q$, $r > 1$ be constants yet to be determined such that $\frac{1}{q} + \frac{1}{r} = \frac{1}{2}$.  Then, from H\"{o}lder's inequality we have
\begin{align*}
\abs{\int_{0}^{T} \int_{\Omega} \sigma_{k} \bm{v}_{k} \cdot \nabla \Pi_{k} \zeta \dx \dt} \leq \int_{0}^{T} \norm{\sigma_{k}}_{L^{q}} \norm{\bm{v}_{k}}_{\bm{L}^{r}} \norm{\nabla \zeta}_{\bm{L}^{2}} \dt .
\end{align*}
By the Gagliardo--Nirenberg inequality we have for $\alpha = \frac{3}{2} - \frac{3}{q} \leq 1$, $\beta = \frac{3}{2} - \frac{3}{r} \leq 1$,
\begin{align*}
& \abs{\int_{0}^{T} \int_{\Omega} \sigma_{k} \bm{v}_{k} \cdot \nabla \Pi_{k} \zeta \dx \dt} \leq C \int_{0}^{T} \norm{\sigma_{k}}_{L^{2}}^{1-\alpha} \norm{\sigma_{k}}_{H^{1}}^{\alpha} \norm{\bm{v}_{k}}_{\bm{H}^{1}}^{\beta} \norm{\bm{v}_{k}}_{\bm{L}^{2}}^{1-\beta} \norm{\nabla \zeta}_{\bm{L}^{2}} \dt \\
& \quad  \leq C \norm{\sigma_{k}}_{L^{\infty}(L^{2})}^{1-\alpha} \int_{0}^{T} \norm{\sigma_{k}}_{H^{1}}^{\alpha} \norm{\bm{v}_{k}}_{\bm{H}^{1}}^{\beta} \norm{\bm{v}_{k}}_{\bm{L}^{2}}^{1-\beta} \norm{\zeta}_{H^{1}} \dt \\
& \quad \leq C \norm{\sigma_{k}}_{L^{\infty}(L^{2})}^{1-\alpha} \norm{\sigma_{k}}_{L^{\alpha x_{1}}(H^{1})}^{\alpha} \norm{\bm{v}_{k}}_{L^{\beta x_{2}}(\bm{H}^{1})}^{\beta} \norm{\bm{v}_{k}}_{L^{(1-\beta)x_{3}}(\bm{L}^{2})}^{1-\beta} \norm{\zeta}_{L^{x_{4}}(H^{1})},
\end{align*}
where
\begin{align}\label{Holderineq:constraints}
\frac{1}{x_{1}} + \frac{1}{x_{2}} + \frac{1}{x_{3}} + \frac{1}{x_{4}} = 1, \quad \alpha x_{1} \leq 2, \quad \beta x_{2} \leq \frac{8}{7}, \quad (1-\beta) x_{3} \leq 2.
\end{align}
Note that $\alpha = \frac{3}{2} - \frac{3}{q} = \frac{3}{r}$, and then substituting into the constraints \eqref{Holderineq:constraints} we find that
\begin{align}
\frac{1}{x_{1}} + \frac{1}{x_{2}} + \frac{1}{x_{3}} \geq \frac{\alpha}{2} + \frac{7}{8} \beta + \frac{1-\beta}{2} = \frac{3}{2r} + \frac{21}{16} - \frac{21}{8r} + \frac{3}{2r} - \frac{1}{4} = \frac{17}{16} + \frac{3}{8}r > 1.
\end{align}
Hence, we cannot find $x_{1}$, $x_{2}$, $x_{3}$ and $x_{4}$ satisfying \eqref{Holderineq:constraints} and we are unable to deduce that $\div (\sigma_{k} \bm{v}_{k})$ lies in the dual space $(H^{1})^{*}$ even with the improved regularity $\bm{v}_{k} \in L^{\frac{8}{7}}(0,T; \bm{H}^{1})$.

\section{Passing to the limit}\label{sec:passlimit}
From \eqref{unified:energyineq}, \eqref{Galerkin:L2H3}, \eqref{pressure:L8/5H1bdd}, \eqref{pressure:L8/7H2bdd}, \eqref{velocityL8/7H1bdd}, \eqref{divvarphikvk}, \eqref{Galerkin:pdtvarphi}, \eqref{divsigmakvk}, \eqref{Galerkin:pdtsigma}, we find that
\begin{align*}
\{\varphi_{k}\}_{k \in \N} & \text{ bounded in } L^{\infty}(0,T;H^{1}) \cap L^{2}(0,T;H^{3}), \\
\{ \pd_{t}\varphi_{k}\}_{k \in \N}, \{ \div (\varphi_{k} \bm{v}_{k})\}_{k \in \N} & \text{ bounded in } L^{\frac{8}{5}}(0,T;(H^{1})^{*}), \\
\{\sigma_{k}\}_{k \in \N} &\text{ bounded in } L^{\infty}(0,T;L^{2}) \cap L^{2}(0,T;H^{1}) \cap L^{2}(\Sigma), \\
\{\pd_{t}\sigma_{k}\}_{k \in \N}, \{ \div (\sigma_{k} \bm{v}_{k})\}_{k \in \N} & \text{ bounded in } L^{\frac{5}{4}}(0,T;(W^{1,5})^{*}), \\
\{\mu_{k}\}_{k \in \N } & \text{ bounded in } L^{2}(0,T;H^{1}), \\
\{p_{k}\}_{k \in \N} & \text{ bounded in } L^{\frac{8}{5}}(0,T;H^{1}) \cap L^{\frac{8}{7}}(0,T;H^{2}), \\
\{ \bm{v}_{k}\}_{k \in \N} & \text{ bounded in } L^{2}(0,T; \bm{L}^{2}) \cap L^{\frac{8}{7}}(0,T;\bm{H}^{1}).
\end{align*}
By standard compactness results (Banach--Alaoglu theorem and reflexive weak compactness theorem), and \cite[\S 8, Cor. 4]{article:Simon86}, and the compact embeddings in dimension 3 (see \cite[Thm. 6.3]{book:AdamsFournier} and \cite[Thm. 11.2, p. 31]{book:Friedman})
\begin{align*}
H^{j+1}(\Omega) = W^{j+1,2}(\Omega) \subset \subset W^{j,q}(\Omega) \quad \forall j \geq 0, j \in \Z, \; 1 \leq q < 6,
\end{align*}
and the compact embedding $L^{2} \subset \subset (H^{1})^{*}$, we obtain, for a relabelled subsequence, the following weak/weak-* convergences:
\begin{alignat*}{3}
\displaybreak[3]
\varphi_{k} & \rightarrow \varphi && \quad \text{ weakly-}* && \quad \text{ in } L^{\infty}(0,T;H^{1}) \cap L^{2}(0,T;H^{3}) \cap W^{1,\frac{8}{5}}(0,T;(H^{1})^{*}), \\
\sigma_{k} & \rightarrow \sigma && \quad \text{ weakly-}* && \quad \text{ in } L^{2}(0,T;H^{1}) \cap L^{\infty}(0,T;L^{2}) \cap L^{2}(\Sigma), \\
\pd_{t}\sigma_{k} & \rightarrow \pd_{t}\sigma && \quad \text{ weakly } && \quad \text{ in } L^{\frac{5}{4}}(0,T;(W^{1,5})^{*}), \\
\mu_{k} & \rightarrow \mu && \quad \text{ weakly } && \quad \text{ in } L^{2}(0,T;H^{1}), \\
p_{k} & \rightarrow p && \quad \text{ weakly } && \quad \text{ in } L^{\frac{8}{5}}(0,T;H^{1}) \cap L^{\frac{8}{7}}(0,T;H^{2}), \\
\bm{v}_{k} & \rightarrow \bm{v} && \quad \text{ weakly } && \quad \text{ in } L^{2}(0,T;\bm{L}^{2}) \cap L^{\frac{8}{7}}(0,T;\bm{H}^{1}), \\
\div(\varphi_{k} \bm{v}_{k}) & \rightarrow \xi && \quad \text{ weakly } && \quad \text{ in } L^{\frac{8}{5}}(0,T;(H^{1})^{*}), \\
\div(\sigma_{k} \bm{v}_{k})& \rightarrow \theta && \quad \text{ weakly } && \quad \text{ in } L^{\frac{5}{4}}(0,T;(W^{1,5})^{*}),
\end{alignat*}
and the following strong convergences:
\begin{alignat*}{3}
\varphi_{k} & \rightarrow \varphi && \quad \text{ strongly } && \quad \text{ in } C^{0}([0,T];L^{r}) \cap L^{2}(0,T;W^{2,r}) \text{ and a.e. in } Q, \\
\sigma_{k} & \rightarrow \sigma && \quad \text{ strongly } && \quad \text{ in } C^{0}([0,T]; (H^{1})^{*}) \cap L^{2}(0,T;L^{r}) \text{ and a.e. in } Q,
\end{alignat*}
for any $r \in [1,6)$ and some functions $\xi \in L^{\frac{8}{5}}(0,T;(H^{1})^{*})$, $\theta \in L^{\frac{5}{4}}(0,T;(W^{1,5})^{*})$.

For the rest of this section, we fix $1 \leq j \leq k$ and $\delta \in C^{\infty}_{c}(0,T)$.  Then, we have $\delta(t) w_{j} \in C^{\infty}(0,T;H^{2})$.  By continuity of $m(\cdot)$, we see that $m(\varphi_{k}) \to m(\varphi)$ a.e. in $Q$.  Thanks to the boundedness of $m(\cdot)$, applying Lebesgue's dominated convergence theorem to $(m(\varphi_{k}) - m(\varphi))^{2} \abs{\delta \nabla w_{j}}^{2}$ yields
\begin{align*}
\norm{m(\varphi_{k}) \delta \nabla w_{j} - m(\varphi) \delta \nabla w_{j}}_{L^{2}(Q)} \to 0 \text{ as } k \to \infty.
\end{align*}
Together with the weak convergence $\mu_{k} \rightharpoonup \mu$ in $L^{2}(0,T;H^{1})$, we obtain by the product of weak-strong convergence
\begin{align*}
\int_{0}^{T} \int_{\Omega} m(\varphi_{k}) \delta \nabla w_{j} \cdot \nabla \mu_{k} \dx \dt \to \int_{0}^{T} \int_{\Omega} m(\varphi) \delta \nabla w_{j} \cdot \nabla \mu \dx \dt \text{ as } k \to \infty.
\end{align*}
Terms involving $n(\cdot)$ can be dealt with in a similar fashion.  For the source term $\Gamma_{\varphi,k} = \Lambda_{\varphi}(\varphi_{k}, \sigma_{k}) - \Theta_{\varphi}(\varphi_{k}, \sigma_{k})\mu_{k}$, by the continuity and boundedness of $\Theta_{\varphi}$, the a.e. convergence of $\varphi_{k} \to \varphi$ and $\sigma_{k} \to \sigma$ in $Q$, we may apply Lebesgue's dominated convergence theorem to deduce that
\begin{align*}
\int_{0}^{T} \int_{\Omega} \abs{\delta w_{j} (\Theta_{\varphi}(\varphi_{k}, \sigma_{k}) - \Theta_{\varphi}(\varphi, \sigma))}^{2} \dx \dt \to 0 \text{ as } k \to \infty,
\end{align*}
that is, we obtain the strong convergence $\delta w_{j} \Theta_{\varphi}(\varphi_{k}, \sigma_{k}) \to \delta w_{j} \Theta_{\varphi}(\varphi, \sigma)$ in $L^{2}(Q)$.  Hence, the weak convergence $\mu_{k} \rightharpoonup \mu$ in $L^{2}(0,T;H^{1})$ yields
\begin{align*}
\int_{0}^{T} \int_{\Omega} \delta w_{j} \Theta_{\varphi}(\varphi_{k}, \sigma_{k}) \mu_{k} \dx \dt \to \int_{0}^{T} \int_{\Omega} \delta w_{j} \Theta_{\varphi}(\varphi, \sigma) \mu \dx \dt \text{ as } k \to \infty.
\end{align*}
Meanwhile, by the triangle inequality $\abs{\abs{a} - \abs{b}} \leq \abs{a-b}$, and H\"{o}lder's inequality, we obtain
\begin{align*}
\int_{0}^{T} \int_{\Omega} \abs{( \abs{\varphi_{k}} - \abs{\varphi}) (\delta w_{j})} \dx \dt \leq \norm{\varphi_{k} - \varphi}_{L^{2}(0,T;L^{2})} \norm{\delta w_{j}}_{L^{2}(0,T;L^{2})} \to 0
\end{align*}
and
\begin{align*}
\int_{0}^{T} \int_{\Omega} \abs{( \abs{\sigma_{k}} - \abs{\sigma}) (\delta w_{j})} \dx \dt \leq \norm{\sigma_{k} - \sigma}_{L^{2}(0,T;L^{2})} \norm{\delta w_{j}}_{L^{2}(0,T;L^{2})} \to 0
\end{align*}
as $k \to \infty$.  In particular, we have
\begin{align*}
(1 + \abs{\varphi_{k}} + \abs{\sigma_{k}}) \abs{\delta w_{j}} \to (1+\abs{\varphi} + \abs{\sigma}) \abs{\delta w_{j}} \text{ strongly in } L^{1}(Q) \text{ as } k \to \infty.
\end{align*}
By the continuity of $\Lambda_{\varphi}$ we have
\begin{align*}
\Lambda_{\varphi}(\varphi_{k}, \sigma_{k}) \to \Lambda_{\varphi}(\varphi, \sigma) \text{ a.e. as } k \to \infty, \quad \abs{\Lambda_{\varphi}(\varphi_{k}, \sigma_{k}) \delta w_{j}} \leq R_{0}(1 + \abs{\varphi_{k}} + \abs{\sigma_{k}}) \abs{\delta w_{j}}.
\end{align*}
Then, the generalised Lebesgue dominated convergence theorem (see \cite[Thm. 1.9, p. 89]{book:Royden}, or \cite[Thm. 3.25, p. 60]{book:Alt}) yields
\begin{align*}
\Lambda_{\varphi}(\varphi_{k}, \sigma_{k}) \delta w_{j} \to \Lambda_{\varphi}(\varphi, \sigma) \delta w_{j} \text{ strongly in } L^{1}(Q) \text{ as } k \to \infty,
\end{align*}
which leads to
\begin{align}\label{Gammavarphi:converegence}
\int_{0}^{T} \int_{\Omega} \Gamma_{\varphi}(\varphi_{k}, \mu_{k}, \sigma_{k}) \delta w_{j} \dx \dt \to \int_{0}^{T} \int_{\Omega} \Gamma_{\varphi}(\varphi, \mu, \sigma) \delta w_{j} \dx \dt \text{ as } k \to \infty.
\end{align}
The same arguments can be applied for the source term $\mathcal{S}$ and for the derivative $\Psi'(\varphi)$ satisfying the linear growth condition \eqref{assump:Psiquadratic}.  For potentials satisfying the growth condition \eqref{assump:PsiSuperquadratic}, we refer to the argument in \cite[\S 3.1.2]{article:GarckeLamDirichlet}.

To identify the limits $\xi$ and $\theta$ of the convection terms $\div (\varphi_{k} \bm{v}_{k})$ and $\div (\sigma_{k} \bm{v}_{k})$, respectively, we argue as follows.  Since $\delta w_{j} \in C^{\infty}(0,T;H^{2}) \subset L^{\frac{8}{3}}(0,T;H^{1})$, by the weak convergence $\div (\varphi_{k} \bm{v}_{k}) \rightharpoonup \xi$ in $L^{\frac{8}{5}}(0,T;(H^{1})^{*})$, we have
\begin{align*}
\int_{0}^{T} \int_{\Omega} \div (\varphi_{k} \bm{v}_{k})  \delta w_{j} \dx \dt \to \int_{0}^{T} \inner{\xi}{w_{j}}_{H^{1}, (H^{1})^{*}} \delta \dt \text{ as } k \to \infty.
\end{align*}
Next, applying integrating by parts and by the boundary conditions $\bm{v}_{k} \cdot \bm{n} = 0$ on $\pd \Omega$ (see \eqref{Galerkinvkspace}), we see that
\begin{align}\label{Galerkin:limit:convectionterm}
\int_{0}^{T} \int_{\Omega} \div (\varphi_{k} \bm{v}_{k}) \delta w_{j} \dx \dt = - \int_{0}^{T} \int_{\Omega} \delta \varphi_{k} \bm{v}_{k} \cdot \nabla w_{j} \dx \dt.
\end{align}
Moreover, we claim that $\delta  \varphi_{k} \nabla w_{j}$ converges strongly to $\delta  \varphi \nabla w_{j}$ in $L^{2}(0,T;\bm{L}^{2})$.  Indeed, we compute
\begin{align*}
\int_{0}^{T} \int_{\Omega} \abs{\delta}^{2} \abs{\nabla w_{j}}^{2} \abs{\varphi_{k} - \varphi}^{2} \dx \dt & \leq \int_{0}^{T} \abs{\delta}^{2} \norm{\nabla w_{j}}_{L^{6}}^{2} \norm{\varphi_{k} - \varphi}_{L^{3}}^{2} \dt \\
& \leq \norm{w_{j}}_{H^{2}}^{2} \norm{\delta}_{L^{\infty}(0,T)}^{2} \norm{\varphi_{k} - \varphi}_{L^{2}(L^{3})}^{2} \to 0
\end{align*}
as $k \to \infty$ by the strong convergence $\varphi_{k} \to \varphi$ in $L^{2}(0,T;L^{r})$ for $r \in [1,6)$.  Together with the weak convergence $\bm{v}_{k} \rightharpoonup \bm{v}$ in $L^{2}(0,T;\bm{L}^{2})$, when passing to the limit $k \to \infty$ in \eqref{Galerkin:limit:convectionterm} we find that
\begin{align*}
\int_{0}^{T} \inner{\xi}{w_{j}}_{H^{1}, (H^{1})^{*}} \delta \dt = - \int_{0}^{T} \int_{\Omega} \delta \varphi \bm{v} \cdot \nabla w_{j}  \dx \dt.
\end{align*}
Applying integration by parts on the right-hand side shows that $\xi = \div (\varphi \bm{v})$ in the sense of distributions.

Now considering $\delta(t)w_{j}$ as an element in $L^{5}(0,T;W^{1,5})$, a similar argument can be used to show $\theta = \div(\sigma \bm{v})$ in the sense of distributions using the strong convergence $\sigma_{k} \to \sigma$ in $L^{2}(0,T;L^{r})$ for $r \in [1,6)$, the weak convergence $\bm{v}_{k} \rightharpoonup \bm{v}$ in $L^{2}(0,T;\bm{L}^{2})$, and the weak convergence $\div(\sigma_{k} \bm{v}_{k}) \rightharpoonup \phi$ in $L^{\frac{5}{4}}(0,T;(W^{1,5})^{*})$.

For the pressure and the velocity, we apply $-\Laplace_{N}$ to both sides of  \eqref{Galerkin:pressure} and test with $w_{j}$, then integrating by parts leads to
\begin{align*}
\int_{\Omega} \nabla p_{k} \cdot \nabla w_{j} \dx = \int_{\Omega} \frac{1}{K} \Gamma_{\bm{v}} w_{j} +(\mu_{k} + \chi \sigma_{k}) \nabla \varphi_{k} \cdot \nabla w_{j} \dx.
\end{align*}
Multiplying by $\delta(t)$, integrating in time and passing to the limit $k \to \infty$, keeping in mind the weak convergences $p_{k} \rightharpoonup p$ in $L^{\frac{8}{5}}(0,T;H^{1})$, $\mu_{k} \rightharpoonup \mu$ in $L^{2}(0,T;H^{1})$, $\sigma_{k} \rightharpoonup \sigma$ in $L^{2}(0,T;H^{1})$, and the strong convergence $\varphi_{k} \to \varphi$ in $L^{2}(0,T;W^{2,r})$ for $r \in [1,6)$ leads to
\begin{align}\label{limit:pressure}
\int_{0}^{T} \int_{\Omega} \delta(t) \nabla p \cdot \nabla w_{j} \dx \dt = \int_{0}^{T} \int_{\Omega} \delta(t) \left ( \frac{1}{K} \Gamma_{\bm{v}} w_{j} + (\mu + \chi \sigma) \nabla \varphi \cdot \nabla w_{j} \right )\dx \dt.
\end{align}
Here we used that $w_{j} \in H^{2}$, and
\begin{equation}\label{pressuretermstrongconvergence}
\begin{aligned}
& \int_{0}^{T} \int_{\Omega} \abs{\delta}^{2} \abs{\nabla \varphi_{k} - \nabla \varphi}^{2} \abs{\nabla w_{j}}^{2} \dx \dt \\
& \quad \leq  \norm{\delta}_{L^{\infty}(0,T)}^{2} \norm{w_{j}}_{W^{1,6}}^{2} \norm{\varphi_{k} - \varphi}_{L^{2}(W^{1,3})}^{2} \to 0 \text{ as } k \to \infty,
\end{aligned}
\end{equation}
to deduce that $\delta \nabla \varphi_{k} \cdot \nabla w_{j} \to \delta \nabla \varphi \cdot \nabla w_{j}$ in $L^{2}(0,T;L^{2})$.  Fix $1 \leq j_{1}, j_{2}, j_{3} \leq k$, and define $\bm{\zeta} = (w_{j_{1}}, w_{j_{2}}, w_{j_{3}})^{\top}$.  Then, we can consider $\delta(t) \bm{\zeta}$ as an element in $L^{\frac{8}{3}}(0,T;\bm{L}^{2}) \subset L^{2}(0,T;\bm{L}^{2})$.  Taking the scalar product of \eqref{Galerkin:velocity} with $\delta \bm{\zeta}$, integrating over $\Omega$ and in time from $0$ to $T$ leads to
\begin{align}\label{pressurevelocitylimit}
\int_{0}^{T} \int_{\Omega} \delta(t) (\bm{v}_{k} + K \nabla p_{k} ) \cdot \bm{\zeta} \dx \dt = \int_{0}^{T} \int_{\Omega} \delta(t) K (\mu_{k} + \chi \sigma_{k}) \nabla \varphi_{k} \cdot \bm{\zeta} \dx \dt.
\end{align}
By the weak convergences $\bm{v}_{k} \rightharpoonup \bm{v}$ in $L^{2}(0,T;\bm{L}^{2})$, $\mu_{k} \rightharpoonup \mu$ in $L^{2}(0,T;H^{1})$, $\sigma_{k} \rightharpoonup \sigma$ in $L^{2}(0,T;H^{1})$, $\nabla p_{k} \rightharpoonup \nabla p$ in $L^{\frac{8}{5}}(0,T;\bm{L}^{2})$, and the strong convergence $\delta \nabla \varphi_{k} \cdot \bm{\zeta} \to \delta \nabla \varphi \cdot \bm{\zeta}$ in $L^{2}(0,T; L^{2})$ (which is proved in a similar manner as \eqref{pressuretermstrongconvergence}), we find that passing to the limit in \eqref{pressurevelocitylimit} yields
\begin{align}\label{limit:velocity}
\int_{0}^{T} \int_{\Omega} \delta(t) (\bm{v} + K \nabla p) \cdot \bm{\zeta} \dx \dt = \int_{0}^{T} \int_{\Omega}  \delta (t) K(\mu + \chi \sigma) \nabla \varphi \cdot \bm{\zeta} \dx \dt.
\end{align}
Then, multiplying \eqref{Galerkin:system} with $\delta \in C^{\infty}_{c}(0,T)$, integrating with respect to time from $0$ to $T$, and passing to the limit $k \to \infty$, we obtain
\begin{equation*}
\begin{aligned}
\int_{0}^{T} \delta(t) \inner{\pd_{t}\varphi}{w_{j}}_{H^{1}, (H^{1})^{*}} \dt & = \int_{0}^{T} \int_{\Omega} \delta(t) \left ( -m(\varphi) \nabla \mu \cdot \nabla w_{j} + \Gamma_{\varphi} w_{j} + \varphi \bm{v} \cdot \nabla w_{j} \right ) \dx \dt, \\
\int_{0}^{T} \int_{\Omega} \delta(t) \mu w_{j} \dx \dt & = \int_{0}^{T} \int_{\Omega} \delta(t) \left ( A \Psi'(\varphi) w_{j} + B \nabla \varphi \cdot \nabla w_{j} - \chi \sigma w_{j} \right ) \dx  \dt,  \\
\int_{0}^{T} \delta(t) \inner{\pd_{t}\sigma}{w_{j}}_{W^{1,5}, (W^{1,5})^{*}} \dt & = \int_{0}^{T} \int_{\Omega} \delta(t) \left ( - n(\varphi) (D \nabla \sigma - \chi \nabla \varphi) \cdot \nabla w_{j} - \mathcal{S} w_{j} \right ) \dx \dt \\
\notag & + \int_{0}^{T} \delta(t) \left ( \int_{\Omega} \sigma \bm{v} \cdot \nabla w_{j} \dx + \int_{\Gamma} b (\sigma_{\infty} - \sigma) w_{j} \dH \right )\dt.
\end{aligned}
\end{equation*}
Since the above, \eqref{limit:pressure} and \eqref{limit:velocity} hold for all $\delta \in C^{\infty}_{c}(0,T)$, we infer that $\{\varphi, \mu, \sigma, p, \bm{v}\}$ satisfies \eqref{CHDN:weak} with $\zeta = \phi = w_{j}$ for a.e. $t \in (0,T)$ and for all $j \geq 1$.  As $\{ w_{j}\}_{j \in \N}$ is a basis for $H^{2}_{N}$, and $H^{2}_{N}$ is dense in both $H^{1}$ and $W^{1,5}$ (see Section \ref{sec:Galerkin}), we see that $\{ \varphi, \mu, \sigma, p, \bm{v}\}$ satisfy \eqref{CHDN:varphi}, \eqref{CHDN:mu}, \eqref{CHDN:pressure} for all $\zeta \in H^{1}$, \eqref{CHDN:sigma} for all $\phi \in W^{1,5}$, and \eqref{CHDN:velo} for all $\bm{\zeta} \in \bm{L}^{2}$.

\paragraph{Attainment of initial conditions.}
It remains to show that $\varphi$ and $\sigma$ attain their corresponding initial conditions.  Strong convergence of $\varphi_{k}$ to $\varphi$ in $C^{0}([0,T];L^{2})$, and the fact that $\varphi_{k}(0) \to \varphi_{0}$ in $L^{2}$ imply that $\varphi(0) = \varphi_{0}$.  Meanwhile, as the limit function $\sigma$ belongs to the function space $C^{0}([0,T]; (H^{1})^{*})$, we see that $\sigma(0):= \sigma(\cdot,0)$ makes sense as an element of $(H^{1})^{*}$.  Let $\zeta \in H^{1}$ be arbitrary, then by the strong convergence $\sigma_{k} \to \sigma$ in $C^{0}([0,T];(H^{1})^{*})$ we see that
\begin{align*}
\inner{\sigma_{k}(0)}{\zeta}_{H^{1},(H^{1})^{*}} \to \inner{\sigma(0)}{\zeta}_{H^{1},(H^{1})^{*}} \text{ as } k \to \infty.
\end{align*}
On the other hand, by \eqref{initialcond:bdd}, we have $\sigma_{k}(0) \to \sigma_{0}$ in $L^{2}$.  This yields
\begin{align*}
\inner{\sigma_{0}}{\zeta}_{H^{1},(H^{1})^{*}} = \lim_{k \to \infty} \inner{\sigma_{k}(0)}{\zeta}_{H^{1},(H^{1})^{*}} = \inner{\sigma(0)}{\zeta}_{H^{1},(H^{1})^{*}}.
\end{align*}

\paragraph{Energy inequality.}
For the energy inequality \eqref{weaksoln:energy:ineq} we employ the weak/weak-* lower semicontinuity of the norms and dual norms to \eqref{unified:energyineq}, \eqref{Galerkin:L2H3}, \eqref{pressure:L8/5H1bdd}, \eqref{pressure:L8/7H2bdd}, \eqref{velocityL8/7H1bdd}, \eqref{divvarphikvk}, \eqref{Galerkin:pdtvarphi}, \eqref{divsigmakvk}, and \eqref{Galerkin:pdtsigma}.

\section{Asymptotic limits}\label{sec:asymptoticlimit}

\subsection{Limit of vanishing permeability}
For $K, b \in (0,1]$ let $(\varphi^{K}, \mu^{K}, \sigma^{K}, \bm{v}^{K}, p^{K})$ denote a weak solution to \eqref{CHDN}-\eqref{CHDNbdy} with $\Gamma_{\bm{v}} = 0$, obtain from Theorem \ref{thm:exist}.  By \eqref{weaksoln:energy:ineq} we deduce that, for a relabelled subsequence as $b \to 0$ and $K \to 0$, the following weak/weak-* convergences:
\begin{alignat*}{3}
\varphi^{K} & \rightarrow \varphi && \quad \text{ weakly-}* && \quad \text{ in } L^{\infty}(0,T;H^{1}) \cap L^{2}(0,T;H^{3}) \cap W^{1,\frac{8}{5}}(0,T;(H^{1})^{*}), \\
\sigma^{K} & \rightarrow \sigma && \quad \text{ weakly-}* && \quad \text{ in } L^{2}(0,T;H^{1}) \cap L^{\infty}(0,T;L^{2}) \cap W^{1,\frac{5}{4}}(0,T;(W^{1,5})^{*}), \\
\mu^{K} & \rightarrow \mu && \quad \text{ weakly } && \quad \text{ in } L^{2}(0,T;H^{1}), \\
p^{K} & \rightarrow p && \quad \text{ weakly } && \quad \text{ in } L^{\frac{8}{5}}(0,T;H^{1}) \cap L^{\frac{8}{7}}(0,T;H^{2}),
\end{alignat*}
and the following strong convergences:
\begin{alignat*}{3}
\varphi^{K} & \rightarrow \varphi && \quad \text{ strongly } && \quad \text{ in } C^{0}([0,T];L^{r}) \cap L^{2}(0,T;W^{2,r}) \text{ and a.e. in } Q, \\
\sigma^{K} & \rightarrow \sigma && \quad \text{ strongly } && \quad \text{ in } C^{0}([0,T]; (H^{1})^{*}) \cap L^{2}(0,T;L^{r}) \text{ and a.e. in } Q, \\
\bm{v}^{K} & \rightarrow \bm{0} && \quad \text{ strongly } && \quad \text{ in } L^{2}(0,T;\bm{L}^{2}) \cap L^{\frac{8}{7}}(0,T;\bm{H}^{1}),  \\
\div(\varphi^{K} \bm{v}^{K}) & \rightarrow 0 && \quad \text{ strongly } && \quad \text{ in } L^{\frac{8}{5}}(0,T;(H^{1})^{*}), \\
\div(\sigma^{K} \bm{v}^{K})& \rightarrow 0 && \quad \text{ strongly } && \quad \text{ in } L^{\frac{5}{4}}(0,T;(W^{1,5})^{*}),
\end{alignat*}
for any $r \in [1,6)$.  The strong convergence of the velocity and the convection terms to zero follows from \eqref{weaksoln:energy:ineq}.  Upon multiplying \eqref{CHDN:weak} by $\delta \in C^{\infty}_{c}(0,T)$ and passing to the limit $b, K \to 0$, we obtain that the limit functions $(\varphi, \mu, \sigma, p)$ satisfy
\begin{subequations}
\begin{align}
\inner{\pd_{t}\varphi}{\zeta}_{H^{1}, (H^{1})^{*}} & = \int_{\Omega} -m(\varphi) \nabla \mu \cdot \nabla \zeta + \Gamma_{\varphi}(\varphi, \mu, \sigma) \zeta \dx, \label{Kzero:varphi}  \\
\int_{\Omega} \mu \zeta \dx & = \int_{\Omega} A \Psi'(\varphi) \zeta + B \nabla \varphi \cdot \nabla \zeta - \chi \sigma \zeta \dx,  \\
\inner{\pd_{t}\sigma}{\phi}_{W^{1,5}, (W^{1,5})^{*}} & = \int_{\Omega} -n(\varphi) (D \nabla \sigma - \chi \nabla \varphi) \cdot \nabla \phi - \mathcal{S}(\varphi, \mu, \sigma) \phi \dx \label{Kzero:sigma}  \\
\int_{\Omega}  \nabla p \cdot \nabla \zeta \dx &= \int_{\Omega} (\mu + \chi \sigma) \nabla \varphi \cdot \nabla \zeta \dx,
\end{align}
\end{subequations}
for all $\zeta \in H^{1}$ and $\phi \in W^{1,5}$ and a.e. $t \in (0,T)$.

Note that substituting any $\zeta \in L^{2}(0,T;H^{1})$ into \eqref{Kzero:varphi}, integrating in time from $0$ to $T$, using H\"{o}lder's inequality and the linear growth condition for $\Gamma_{\varphi}$ leads to the deduction that $\pd_{t} \varphi \in L^{2}(0,T;(H^{1})^{*})$.  To show that $\pd_{t} \sigma \in L^{2}(0,T;(H^{1})^{*})$ we argue as follows.  For any $\xi \in L^{2}(0,T;H^{1})$, we can define
\begin{align*}
F(\xi) := \int_{0}^{T} \int_{\Omega} -n(\varphi) (D \nabla \sigma - \chi \nabla \varphi) \cdot \nabla \xi - \mathcal{S}(\varphi, \mu, \sigma) \xi \dx \dt.
\end{align*}
By H\"{o}lder's inequality and the growth condition on $\mathcal{S}$, we see that $F \in L^{2}(0,T;(H^{1})^{*})$.  It is known that the set of functions that are finite linear combinations of $C^{1}_{c}(0,T) \cdot H^{2}_{N}(\Omega) := \{ \delta(t) \phi(x) : \delta \in C^{1}_{c}(0,T), \; \phi \in H^{2}_{N}(\Omega) \}$ is dense in $C^{1}_{c}(0,T;H^{1})$ (see for instance \cite[p. 384]{book:RenardyRogers}, and in fact this is what we use in Section \ref{sec:passlimit}).  Let $\zeta \in C^{1}_{c}(0,T;H^{1})$ and let $\{\zeta^{n}\}_{n \in \N}$ denote a sequence of functions of the above form such that $\zeta^{n} \to \zeta$ in $C^{1}_{c}(0,T;H^{1})$ as $n \to \infty$.  Then, substituting $\phi = \zeta^{n}$ in \eqref{Kzero:sigma}, integrating over $t$ from $0$ to $T$, and passing to the limit $n \to \infty$ yields
\begin{align*}
\lim_{n \to \infty} \int_{0}^{T} \inner{\pd_{t} \sigma}{\zeta^{n}}_{W^{1,5}, (W^{1,5})^{*}} \dt = \lim_{n \to \infty} F(\zeta^{n}) = F(\zeta).
\end{align*}
Moreover, by the definition of the weak time derivative, we have
\begin{align*}
\int_{0}^{T} \inner{\pd_{t} \sigma}{\zeta^{n}}_{W^{1,5}, (W^{1,5})^{*}} \dt = - \int_{0}^{T} \int_{\Omega} \sigma \pd_{t} \zeta^{n} \dx \dt \to -\int_{0}^{T} \int_{\Omega} \sigma \pd_{t} \zeta \dx \dt \text{ as } n \to \infty.
\end{align*}
Hence, we obtain
\begin{align*}
-\int_{0}^{T} \int_{\Omega} \sigma \pd_{t} \zeta \dx \dt = F(\zeta) = \int_{0}^{T} \int_{\Omega} - n(\varphi) (D \nabla \sigma - \chi \nabla \varphi) \cdot \nabla \zeta - \mathcal{S}(\varphi, \mu, \sigma) \zeta \dx \dt
\end{align*}
for all $\zeta \in C^{1}_{c}(0,T;H^{1})$.  This implies that the weak time derivative $\pd_{t} \sigma$ satisfies
\begin{align*}
\int_{0}^{T} \inner{\pd_{t} \sigma}{\zeta}_{H^{1}, (H^{1})^{*}} \dt = F(\zeta) \quad \forall \zeta \in C^{1}_{c}(0,T;H^{1}),
\end{align*}
and as $F$ belongs to $L^{2}(0,T;(H^{1})^{*})$, we find that $\pd_{t} \sigma$ also belongs to $L^{2}(0,T;(H^{1})^{*})$.  Furthermore, due to the improved regularity $\pd_{t} \sigma \in L^{2}(0,T;(H^{1})^{*})$, we use the continuous embedding
\begin{align*}
L^{2}(0,T;H^{1}) \cap H^{1}(0,T;(H^{1})^{*}) \subset C^{0}([0,T];L^{2})
\end{align*}
to deduce that $\sigma(0) = \sigma_{0}$.

\subsection{Limit of vanishing chemotaxis}
For $\chi, b \in (0,1]$, let $(\varphi^{\chi}, \mu^{\chi}, \sigma^{\chi}, \bm{v}^{\chi}, p^{\chi})$ denote a weak solution to (\ref{CHDN})-(\ref{CHDNbdy}) obtain from Theorem \ref{thm:exist}.  By \eqref{weaksoln:energy:ineq} we deduce that, for a relabelled subsequence as $b \to 0$ and $\chi \to 0$, the following weak/weak-* convergences:
\begin{alignat*}{3}
\varphi^{\chi} & \rightarrow \varphi && \, \text{ weakly-}* && \, \text{ in } L^{\infty}(0,T;H^{1}) \cap L^{2}(0,T;H^{3}) \cap W^{1,\frac{8}{5}}(0,T;(H^{1})^{*}), \\
\sigma^{\chi} & \rightarrow \sigma && \, \text{ weakly-}* && \, \text{ in } L^{2}(0,T;H^{1}) \cap L^{\infty}(0,T;L^{2}) \cap W^{1,\frac{5}{4}}(0,T;(W^{1,5})^{*}), \\
\mu^{\chi} & \rightarrow \mu && \, \text{ weakly } && \, \text{ in } L^{2}(0,T;H^{1}), \\
p^{\chi} & \rightarrow p && \, \text{ weakly } && \, \text{ in } L^{\frac{8}{5}}(0,T;H^{1}) \cap L^{\frac{8}{7}}(0,T;H^{2}), \\
\bm{v}^{\chi} & \rightarrow \bm{v} && \, \text{ weakly } && \, \text{ in } L^{2}(0,T;\bm{L}^{2}) \cap L^{\frac{8}{7}}(0,T;\bm{H}^{1}), \\
\div(\varphi^{\chi} \bm{v}^{\chi}) & \rightarrow \div( \varphi \bm{v}) && \, \text{ weakly } && \, \text{ in } L^{\frac{8}{5}}(0,T;(H^{1})^{*}), \\
\div(\sigma^{\chi} \bm{v}^{\chi})& \rightarrow \div (\sigma \bm{v}) && \, \text{ weakly } && \, \text{ in } L^{\frac{5}{4}}(0,T;(W^{1,5})^{*}),
\end{alignat*}
and the following strong convergences:
\begin{alignat*}{3}
\varphi^{\chi} & \rightarrow \varphi && \quad \text{ strongly } && \quad \text{ in } C^{0}([0,T];L^{r}) \cap L^{2}(0,T;W^{2,r}) \text{ and a.e. in } Q, \\
\sigma^{\chi} & \rightarrow \sigma && \quad \text{ strongly } && \quad \text{ in } C^{0}([0,T];(H^{1})^{*}) \cap L^{2}(0,T;L^{r}) \text{ and a.e. in } Q,
\end{alignat*}
for any $r \in [1,6)$.  For any $\delta \in C^{\infty}_{c}(0,T)$ and $\zeta \in H^{1}$, we have
\begin{align*}
\abs{\int_{0}^{T} \int_{\Omega} \delta \chi \sigma^{\chi} \zeta  \dx \dt} & \leq \chi \norm{\sigma^{\chi}}_{L^{2}(L^{2})} \norm{\zeta}_{L^{2}} \norm{\delta}_{L^{2}(0,T)} \to 0, \\
\abs{\int_{0}^{T} \int_{\Omega} \delta n(\varphi^{\chi}) \chi \nabla \varphi^{\chi} \cdot \nabla \zeta \dx \dt} & \leq n_{1} \chi \norm{\nabla \varphi^{\chi}}_{L^{2}(\bm{L}^{2})} \norm{\nabla \zeta}_{\bm{L}^{2}} \norm{\delta}_{L^{2}(0,T)} \to 0, \\
\abs{\int_{0}^{T} \int_{\Omega} \delta \chi \sigma^{\chi} \nabla \varphi^{\chi} \cdot \nabla \zeta \dx \dt} & \leq \chi \norm{\nabla \zeta}_{\bm{L}^{2}} \norm{\sigma^{\chi}}_{L^{2}(L^{4})} \norm{\nabla \varphi^{\chi}}_{L^{2}(\bm{L}^{4})} \norm{\delta}_{L^{\infty}(0,T)} \to 0,
\end{align*}
as $\chi \to 0$.  Thus, multiplying \eqref{CHDN:weak} with $\delta \in C^{\infty}_{c}(0,T)$, and then passing to the limit $b, \chi \to 0$, we see that $(\varphi, \mu, \sigma, \bm{v}, p)$ satisfies
\begin{subequations}
\begin{align}
\inner{\pd_{t}\varphi}{\zeta}_{H^{1}, (H^{1})^{*}} & = \int_{\Omega} -m(\varphi) \nabla \mu \cdot \nabla \zeta + \Gamma_{\varphi}(\varphi, \mu, \sigma) \zeta + \varphi \bm{v} \cdot \nabla \zeta \dx,  \\
\int_{\Omega} \mu \zeta \dx & = \int_{\Omega} A \Psi'(\varphi) \zeta + B \nabla \varphi \cdot \nabla \zeta \dx,  \\
\inner{\pd_{t}\sigma}{\phi}_{W^{1,5}, (W^{1,5})^{*}} & = \int_{\Omega} -n(\varphi) D \nabla \sigma \cdot \nabla \phi - \mathcal{S}(\varphi, \mu, \sigma) \phi + \sigma \bm{v} \cdot \nabla \phi \dx , \\
\int_{\Omega}  \nabla p \cdot \nabla \zeta \dx &= \int_{\Omega} \frac{1}{K} \Gamma_{\bm{v}} \zeta +  \mu \nabla \varphi \cdot \nabla \zeta \dx,  \\
\int_{\Omega} \bm{v} \cdot \bm{\zeta} \dx & = \int_{\Omega} - K (\nabla p - \mu \nabla \varphi) \cdot \bm{\zeta} \dx,
\end{align}
\end{subequations}
for all $\zeta \in H^{1}$, $\phi \in W^{1,5}$, $\bm{\zeta} \in \bm{L}^{2}$ and a.e. $t \in (0,T)$.

\section{Existence in two dimensions}\label{sec:2D}
We first derive an analogous result to Lemma \ref{lem:reg:pressurevelocity} for two dimensions.
\begin{lemma}\label{lem:2D:pressure:velo:reg}
Let $\Omega \subset \R^{2}$ be a bounded domain with $C^{3}$-boundary.  Let $\varphi \in L^{\infty}(0,T;H^{1}) \cap L^{2}(0,T;H^{2}_{N} \cap H^{3})$, $\sigma \in L^{2}(0,T;H^{1})$, $\mu \in L^{2}(0,T;H^{1})$, the source term $\Gamma_{\bm{v}} \in L^{2}(0,T;L^{2}_{0})$, and the function $p$ satisfying \eqref{pressure:sys}.  Then,
\begin{align*}
p \in L^{k}(0,T;H^{1}) \cap L^{q}(0,T;H^{2}), \quad \bm{v} \in L^{q}(0,T;\bm{H}^{1}),
\end{align*}
for any
\begin{align*}
1 \leq k < 2, \quad 1 \leq q < \frac{4}{3}.
\end{align*}

\end{lemma}
\begin{proof}
We estimate \eqref{Galerkin:pressure:L2grad} differently than in the proof of Lemma \ref{lem:reg:pressurevelocity}.  By H\"{o}lder's inequality for any $1 \leq s < \infty$ we have
\begin{align*}
\norm{(\mu + \chi \sigma) \nabla \varphi}_{\bm{L}^{2}} \leq \norm{\mu + \chi \sigma}_{L^{2s}} \norm{\nabla \varphi}_{\bm{L}^{\frac{2 s}{s-1}}}.
\end{align*}
Then, by the Gagliardo--Nirenberg inequality \eqref{GagNirenIneq} with $p = \frac{2s}{s-1}$, $j = 0$, $r = 2$, $m = 2$, $d = 2$, $q = 2$, and $\alpha = \frac{1}{2} - \frac{s-1}{2s} = \frac{1}{2s}$, we find that
\begin{align*}
\norm{\nabla \varphi}_{\bm{L}^{\frac{2s}{s-1}}} \leq C \norm{\nabla \varphi}_{\bm{H}^{2}}^{\frac{1}{2s}} \norm{\nabla \varphi}_{\bm{L}^{2}}^{1 - \frac{1}{2s}} \leq C \norm{\varphi}_{H^{3}}^{\frac{1}{2s}} \norm{\varphi}_{H^{1}}^{1 - \frac{1}{2s}}.
\end{align*}
Then, by H\"{o}lder's inequality and the Sobolev embedding $H^{1} \subset L^{r}$ for $1 \leq r < \infty$ in two dimensions, we have for $w, y \geq 1$,
\begin{align*}
& \int_{0}^{T} \norm{(\mu + \chi \sigma) \nabla \varphi}_{\bm{L}^{2}}^{w} \dt \leq C \int_{0}^{T} \norm{\mu + \chi \sigma}_{L^{2s}}^{w} \norm{\nabla \varphi}_{\bm{L}^{\frac{2s}{s-1}}}^{w} \dt \\
& \quad  \leq C \norm{\varphi}_{L^{\infty}(H^{1})}^{w \frac{2s-1}{2s}}  \norm{\mu + \chi \sigma}_{L^{wy}(H^{1})}^{w}  \norm{\varphi}_{L^{\frac{w}{2s} \frac{y}{y-1}}(H^{3})}^{\frac{w}{2s}}.
\end{align*}
As $\mu, \sigma$ belong to $L^{2}(0,T;H^{1})$ and $\varphi$ belongs to $L^{2}(0,T;H^{3}) \cap L^{\infty}(0,T;H^{1})$, we need
\begin{align*}
wy = 2, \quad \frac{w}{2s}\frac{y}{y-1} = 2 \Longrightarrow y = \frac{2s +1}{2s}, \quad w = \frac{4s}{1+2s}.
\end{align*}
Since $w = \frac{4s}{1+2s} < 2$ for all $s \in [1,\infty)$, and $\Gamma_{\bm{v}} \in L^{2}(0,T;L^{2}_{0})$, the computations in the proof of Lemma \ref{lem:reg:pressurevelocity} yields that
\begin{align*}
p \in L^{k}(0,T;H^{1}) \quad \text{ for } 1 \leq k < 2
\end{align*}
Next, we see that
\begin{align*}
\norm{\div ((\mu + \chi \sigma) \nabla \varphi)}_{L^{2}} & \leq \norm{(\mu + \chi \sigma) \Laplace \varphi}_{L^{2}} + \norm{\nabla (\mu + \chi \sigma) \cdot \nabla \varphi}_{L^{2}} \\
& \leq \norm{\mu + \chi \sigma}_{L^{2s}} \norm{\Laplace \varphi}_{L^{\frac{2s}{s-1}}} + \norm{\nabla (\mu + \chi \sigma)}_{\bm{L}^{2}} \norm{\nabla \varphi}_{\bm{L}^{\infty}} .
\end{align*}
By the Gagliardo--Nirenberg inequality \eqref{GagNirenIneq} with $p = \infty$, $j = 0$, $r = 2$, $m = 2$, $d = 2$, $q = 2$ and $\alpha = \frac{1}{2}$, we have
\begin{align}\label{2D:nablavarphi:Linfty}
\norm{\nabla \varphi}_{\bm{L}^{\infty}} \leq C\norm{\nabla \varphi}_{\bm{H}^{2}}^{\frac{1}{2}} \norm{\nabla \varphi}_{\bm{L}^{2}}^{\frac{1}{2}} \leq C \norm{\varphi}_{H^{3}}^{\frac{1}{2}} \norm{\varphi}_{H^{1}}^{\frac{1}{2}},
\end{align}
and with $p = \frac{2s}{s-1}$, $j = 1$, $r = 2$, $m = 2$, $d = 2$, $q = 2$ and $\alpha = \frac{s+1}{2s} \in (\frac{1}{2},1]$ for $s \in [1,\infty)$, we have
\begin{align}\label{2D:Laplacevarphi:L2s/s-1}
\norm{\Laplace \varphi}_{L^{\frac{2s}{s-1}}} \leq C \norm{\nabla \varphi}_{\bm{H}^{2}}^{\frac{s+1}{2s}} \norm{\nabla \varphi}_{\bm{L}^{2}}^{\frac{s-1}{2s}} \leq C \norm{\varphi}_{H^{3}}^{\frac{s+1}{2s}} \norm{\varphi}_{H^{1}}^{\frac{s-1}{2s}}.
\end{align}
Hence, for $w, y, z \geq 1$, we find that
\begin{align*}
\int_{0}^{T} \norm{\div ((\mu + \chi \sigma) \nabla \varphi)}_{L^{2}}^{w} \dt & \leq C \norm{\varphi}_{L^{\infty}(H^{1})}^{\frac{w}{2}} \norm{\mu + \chi \sigma}_{L^{wz}(H^{1})}^{w}  \norm{\varphi}_{L^{\frac{w}{2} \frac{z}{z-1}}(H^{3})}^{\frac{w}{2}} \\
& + C \norm{\varphi}_{L^{\infty}(H^{1})}^{w \frac{s-1}{2s}} \norm{\mu + \chi \sigma}_{L^{wy}(H^{1})}^{w} \norm{\varphi}_{L^{w \frac{s+1}{2s} \frac{y}{y-1}}(H^{3})}^{w \frac{s+1}{2s}}.
\end{align*}
Since $\frac{s+1}{2s} \leq 1$ for all $s \in [1,\infty)$, we require
\begin{align*}
wy = 2, \quad \frac{wy (s+1)}{2s(y-1)} = 2 \Longrightarrow y = \frac{3s+1}{2s}, \quad w = \frac{4s}{1+3s}.
\end{align*}
We choose $z = \frac{3s+1}{2s} \in (\frac{3}{2},2]$ so that
\begin{align*}
wz = 2, \quad \frac{w}{2} \frac{z}{z-1} = \frac{2s}{1 + 3s} \frac{3s+1}{s+1} = \frac{2s}{s+1} \in [1,2),
\end{align*}
and thus we obtain
\begin{align*}
\int_{0}^{T} \norm{\div ((\mu + \chi \sigma) \nabla \varphi)}_{L^{2}}^{\frac{4s}{1+3s}} \dt & \leq C \norm{\varphi}_{L^{\infty}(H^{1})}^{\frac{2s}{1 + 3s}} \norm{\mu + \chi \sigma}_{L^{2}(H^{1})}^{\frac{4s}{1+3s}}  \norm{\varphi}_{L^{\frac{2s}{s+1}}(H^{3})}^{\frac{2s}{1+3s}} \\
& + C \norm{\varphi}_{L^{\infty}(H^{1})}^{\frac{2s-2}{1 + 3s}} \norm{\mu + \chi \sigma}_{L^{2}(H^{1})}^{\frac{4s}{1+3s}} \norm{\varphi}_{L^{2}(H^{3})}^{\frac{2s+2}{1 +3s}}.
\end{align*}
From \eqref{Galerkin:pressure:H2} and using the fact that $\frac{4s}{1+3s} < \frac{4s}{1+2s}$ for all $s \in [1,\infty)$, we see that
\begin{align*}
p \in L^{q}(0,T;H^{2}) \quad \text{ for } 1 \leq q < \frac{4}{3}.
\end{align*}
Similarly, from \eqref{Galerkin:velo:H1}, \eqref{2D:nablavarphi:Linfty} and \eqref{2D:Laplacevarphi:L2s/s-1}, we obtain for fixed $1 \leq i, j \leq 2$, and any $s \in [1,\infty)$,
\begin{equation}\label{2D:velo:H1:estimate}
\begin{aligned}
\norm{D_{i} v_{j}}_{L^{2}} & = K \norm{D_{i} D_{j} p - (D_{i}(\mu + \chi \sigma) D_{j} \varphi - (\mu + \chi \sigma) D_{i} D_{j} \varphi}_{L^{2}} \\
& \leq K \left ( \norm{p}_{H^{2}} + \norm{\nabla (\mu + \chi \sigma)}_{\bm{L}^{2}} \norm{\nabla \varphi}_{\bm{L}^{\infty}} + \norm{\mu + \chi \sigma}_{L^{2s}} \norm{D^{2}\varphi}_{L^{\frac{2s}{s-1}}} \right ) \\
& \leq K \left ( \norm{p}_{H^{2}} + C \norm{\mu + \chi \sigma}_{H^{1}} \left ( \norm{\varphi}_{H^{3}}^{\frac{1}{2}} \norm{\varphi}_{H^{1}}^{\frac{1}{2}} + \norm{\varphi}_{H^{3}}^{\frac{s+1}{2s}} \norm{\varphi}_{H^{1}}^{\frac{s-1}{2s}} \right ) \right ).
\end{aligned}
\end{equation}
Then, a similar calculation shows that the right-hand side is bounded in $L^{\frac{4s}{1+3s}}(0,T)$, which in turn implies that
\begin{align*}
\bm{v} \in L^{q}(0,T;\bm{H}^{1}) \text{ for } 1 \leq q < \frac{4}{3}.
\end{align*}
\end{proof}
By the above new estimates we can show that $\div(\varphi \bm{v})$ and $\pd_{t} \varphi$ have improved temporal regularity, and that $\div (\sigma \bm{v})$ and $\pd_{t}\sigma$ belong to the dual space $(H^{1})^{*}$.
\begin{lemma}
For dimension $d = 2$, let $(\varphi_{k}, \mu_{k}, \sigma_{k}, p_{k}, \bm{v}_{k})$ denote the Galerkin ansatz from Section \ref{sec:Galerkin} satisfying \eqref{unified:energyineq}.  Then, it holds that for $\frac{4}{3} \leq w < 2$ and $1 < r < \frac{8}{7}$,
\begin{align*}
\norm{\div (\varphi_{k} \bm{v}_{k})}_{L^{w}((H^{1})^{*})} + \norm{\div (\sigma_{k} \bm{v}_{k})}_{L^{r}((H^{1})^{*})} & \leq K^{\frac{1}{2}} \mathcal{E}, \\
\norm{\pd_{t} \varphi_{k}}_{L^{w}((H^{1})^{*})}  + \norm{\pd_{t} \sigma_{k}}_{L^{r}((H^{1})^{*})} & \leq \mathcal{E} \left (1 + K^{\frac{1}{2}} \right ),
\end{align*}
where $\mathcal{E}$ denotes positive constants that are uniformly bounded for $b, \chi \in (0,1]$ and are also uniformly bounded for $K \in (0,1]$ when $\Gamma_{\bm{v}} = 0$.
\end{lemma}
\begin{proof}
The assertions for $\pd_{t}\varphi_{k}$ and $\pd_{t}\sigma_{k}$ will follow via similar arguments in Section \ref{sec:convectionandtimederivative} once we establish the assertion for the convection terms.  In dimension $d = 2$, we have the embedding $L^{2}(0,T;H^{1}) \cap L^{\infty}(0,T;L^{2}) \subset L^{4}(Q)$, and by the Gagliardo--Nirenberg inequality \eqref{GagNirenIneq} with $p = 4$, $j = 0$, $r = 2$, $d = 2$, $m = 1$, $q = 2$ and $\alpha = \frac{1}{2}$,
\begin{align*}
\norm{f}_{L^{4}} \leq C \norm{f}_{H^{1}}^{\frac{1}{2}} \norm{f}_{L^{2}}^{\frac{1}{2}}.
\end{align*}
Consider an arbitrary $\zeta \in L^{s}(0,T;H^{1})$ for some $s \geq 1$ yet to be determined.  Then, we compute that
\begin{align*}
\abs{\int_{0}^{T} \int_{\Omega} \sigma_{k} \bm{v}_{k} \cdot \nabla \Pi_{k} \zeta \dx \dt} & \leq \int_{0}^{T} \norm{\sigma_{k}}_{L^{4}} \norm{\bm{v}_{k}}_{\bm{L}^{4}} \norm{\nabla \zeta}_{\bm{L}^{2}} \dt \\
& \leq C \norm{\sigma_{k}}_{L^{4}(Q)} \left ( \int_{0}^{T} \norm{\bm{v}_{k}}_{\bm{H}^{1}}^{\frac{2}{3}} \norm{\bm{v}_{k}}_{\bm{L}^{2}}^{\frac{2}{3}} \norm{\zeta}_{H^{1}}^{\frac{4}{3}} \dt \right )^{\frac{3}{4}} \\
& \leq  C \norm{\sigma_{k}}_{L^{4}(Q)} \norm{\bm{v}_{k}}_{L^{\frac{2}{3} x_{1}}(\bm{H}^{1})}^{\frac{1}{2}} \norm{\bm{v}_{k}}_{L^{\frac{2}{3} x_{2}}(\bm{L}^{2})}^{\frac{1}{2}} \norm{\zeta}_{L^{\frac{4}{3} x_{3}}(H^{1})},
\end{align*}
where $x_{1}, x_{2}, x_{3} \geq 1$ satisfy
\begin{align*}
\frac{1}{x_{1}} + \frac{1}{x_{2}} + \frac{1}{x_{3}} = 1, \quad \frac{2}{3} x_{1} < \frac{4}{3}, \quad \frac{2}{3} x_{2} \leq 2 \Longrightarrow x_{3} > 6.
\end{align*}
Then, from \eqref{unified:energyineq} and \eqref{2D:velo:H1:estimate}, it holds that
\begin{align*}
\abs{\int_{0}^{T} \int_{\Omega} \sigma_{k} \bm{v}_{k} \cdot \nabla \Pi_{k} \zeta \dx \dt} \leq \mathcal{E} K^{\frac{1}{2}} \norm{\zeta}_{L^{s}(H^{1})} \quad \text{ for } s = \frac{4}{3}x_{3} > 8,
\end{align*}
that is, $\{\div (\sigma_{k} \bm{v}_{k})\}_{k \in \N}$ is uniformly bounded in the dual space of $L^{s}(0,T;H^{1})$ for $s > 8$.  Similarly, by the Gagliardo--Nirenberg inequality \eqref{GagNirenIneq} with $p = \infty$, $j = 0$, $r = 2$, $d = 2$, $m = 3$, $q \in [1,\infty)$ and $\alpha = \frac{1}{q+1}$,
\begin{align*}
\norm{\varphi_{k}}_{L^{\infty}} \leq C \norm{\varphi_{k}}_{H^{3}}^{\frac{1}{q+1}} \norm{\varphi_{k}}_{L^{q}}^{\frac{q}{q+1}} \leq C \norm{\varphi_{k}}_{H^{3}}^{\frac{1}{q+1}} \norm{\varphi_{k}}_{H^{1}}^{\frac{q}{q+1}}.
\end{align*}
Proceeding as in \eqref{convectionvarphiv}, we find that for an arbitrary $\zeta \in L^{s}(0,T;H^{1})$, where $s \geq 1$ is yet to be determined,
\begin{align*}
\abs{\int_{0}^{T} \int_{\Omega} \varphi_{k} \bm{v}_{k} \cdot \nabla \Pi_{k} \zeta \dx \dt} & \leq \int_{0}^{T} \norm{\bm{v}_{k}}_{\bm{L}^{2}} \norm{\varphi_{k}}_{L^{\infty}} \norm{\nabla \Pi_{k} \zeta}_{\bm{L}^{2}} \dt \\
& \leq C \norm{\varphi_{k}}_{L^{\infty}(H^{1})}^{\frac{q}{q+1}} \norm{\bm{v}_{k}}_{L^{2}(\bm{L}^{2})} \norm{\varphi_{k}}_{L^{2}(H^{3})}^{\frac{1}{q+1}} \norm{\zeta}_{L^{\frac{2(q+1)}{q}}(H^{1})} \\
& \leq \mathcal{E} K^{\frac{1}{2}} \norm{\zeta}_{L^{\frac{2(q+1)}{q}}(H^{1})},
\end{align*}
and so $\{\div(\varphi_{k} \bm{v}_{k})\}_{k \in \N}$ is uniformly bounded in the dual space of $L^{s}(0,T;H^{1})$ for $s = 2 + \frac{2}{q} \in (2,4]$.
\end{proof}
\begin{remark}\label{rem:2D:W14star}
We point out that in the absence of the regularity result $\bm{v}_{k} \in L^{q}(0,T;\bm{H}^{1})$ from Lemma \ref{lem:2D:pressure:velo:reg}, and if we only have $\bm{v}_{k} \in L^{2}(0,T;\bm{L}^{2})$, then we obtain
\begin{align*}
\abs{\int_{0}^{T} \int_{\Omega} \sigma_{k} \bm{v}_{k} \cdot \nabla \Pi_{k} \zeta \dx \dt} \leq \norm{\sigma_{k}}_{L^{4}(Q)} \norm{\bm{v}_{k}}_{L^{2}(\bm{L}^{2})} \norm{\nabla \zeta}_{L^{4}(\bm{L}^{4})},
\end{align*}
and this implies that both $\{\div(\sigma_{k} \bm{v}_{k})\}_{k \in \N}$ and $\{\pd_{t}\sigma_{k}\}_{k \in \N}$ are bounded uniformly only in $L^{\frac{4}{3}}(0,T;(W^{1,4})^{*})$.
\end{remark}

\section{Discussion}\label{sec:Discussion}

\paragraph{Reformulations of Darcy's law and the pressure.}
Associated to Darcy's law \eqref{CHDN:Darcy} is the term $\lambda_{\bm{v}} := p - \mu \varphi - \frac{D}{2} \abs{\sigma}^{2}$ which will contribute the source term $\Gamma_{\bm{v}} \lambda_{\bm{v}}$ in the energy identity \eqref{apriorienergy:identity}.  In \cite[Rmk. 2.1]{article:GarckeLamSitkaStyles} three other reformulations of Darcy's law \eqref{CHDN:Darcy} and the pressure are considered:
\begin{enumerate}[label=$(\mathrm{R \arabic*})$, ref = $(\mathrm{R \arabic*})$]
\item Let $q := p - A \Psi(\varphi) - \frac{B}{2} \abs{\nabla \varphi}^{2}$ so that
\begin{subequations}
\begin{align}
\lambda_{\bm{v}} & = q + A \Psi(\varphi) + \frac{B}{2} \abs{\nabla \varphi}^{2} - \frac{D}{2} \abs{\sigma}^{2} + \chi \sigma(1-\varphi) - \mu \varphi, \label{Darcy0:lambdav}\\
\bm{v} & = K (\nabla (-q - \tfrac{B}{2} \abs{\nabla \varphi}^{2}) - B \Laplace \varphi \nabla \varphi) = - K (\nabla q + B \div (\nabla \varphi \otimes \nabla \varphi)). \label{Darcy:type0}
\end{align}
\end{subequations}
\item Let $\hat{p} := p + \frac{D}{2} \abs{\sigma}^{2} + \chi \sigma(1-\varphi)$ so that
\begin{subequations}
\begin{align}
\lambda_{\bm{v}} & = \hat{p} - \mu \varphi - D\abs{\sigma}^{2} - \chi \sigma (1-\varphi), \label{Darcy1:lambdav}\\
\bm{v} & = -K (\nabla \hat{p} - \mu \nabla \varphi - (D \sigma + \chi (1-\varphi)) \nabla \sigma).\label{Darcy:type1}
\end{align}
\end{subequations}
\item Let $\tilde{p} := p - \frac{D}{2} \abs{\sigma}^{2} - \mu \varphi$ so that
\begin{subequations}
\begin{align}
\lambda_{\bm{v}} & = \tilde{p}, \label{Darcy2:lambdav} \\
\bm{v} & =  -K (\nabla \tilde{p} + \varphi \nabla \mu + \sigma \nabla (D\sigma + \chi (1-\varphi))).\label{Darcy:type2}
\end{align}
\end{subequations}
\end{enumerate}
From the viewpoint of estimating the source term $\Gamma_{\bm{v}} \lambda_{\bm{v}}$, we see that \eqref{Darcy2:lambdav} has the advantage of being the simplest.  Meanwhile, for \eqref{Darcy1:lambdav} the analysis for $\Gamma_{\bm{v}} \lambda_{\bm{v}}$ is similar to that performed in Section \ref{sec:SourcetermDarcy}, but for \eqref{Darcy0:lambdav} the main difficulty will be to estimate the terms $(A\Psi(\varphi) + \frac{B}{2} \abs{\nabla \varphi}^{2}) \Gamma_{\bm{v}}$ and $(-\frac{D}{2} \abs{\sigma}^{2} + \chi \sigma \varphi) \Gamma_{\bm{v}}$, which at first glance would require the assumption that $\Gamma_{\bm{v}} \in L^{\infty}(Q)$, and obtaining an $L^{2}$-estimate for the pressure $q$ from the Darcy law \eqref{Darcy:type0} would be difficult due to the term $\div (\nabla \varphi \otimes \nabla \varphi)$.
\paragraph{Other boundary conditions for the pressure and velocity.}
In \cite[\S 2.4.4]{article:GarckeLamSitkaStyles} the authors have discussed possible boundary conditions for the velocity and for the pressure.  As discussed in Section \ref{sec:mainresults} following Assumption \ref{assump:main}, we require the source term $\Gamma_{\bm{v}}$ to have zero mean due to the no-flux boundary condition $\bm{v} \cdot \bm{n} = 0$ on $\pd \Omega$.  The general energy identity (with homogeneous Neumann boundary conditions for $\varphi$ and $\mu$) from \cite[Equ. (2.27)]{article:GarckeLamSitkaStyles} reads as
\begin{align*}
& \frac{\dd}{\dt} \int_{\Omega}  A \Psi(\varphi) + \frac{B}{2} \abs{\nabla \varphi}^{2} + \frac{D}{2} \abs{\sigma}^{2} + \chi \sigma (1-\varphi)  \dx \\
& \quad + \int_{\Omega} m(\varphi) \abs{\nabla \mu}^{2} + n(\varphi) \abs{\nabla (D \sigma + \chi (1-\varphi))}^{2} + \frac{1}{K} \abs{\bm{v}}^{2} \dx \\
& \quad =  \int_{\Omega} \Gamma_{\varphi} \mu - \mathcal{S} (D \sigma + \chi (1-\varphi)) + \Gamma_{\bm{v}} \lambda_{\bm{v}} \dx \\
& \quad + \int_{\pd \Omega} (D \sigma + \chi (1-\varphi)) n(\varphi) (D \pdnu  \sigma ) - (\bm{v} \cdot \bm{n}) \left ( \frac{D}{2} \abs{\sigma}^{2} + \chi \sigma (1-\varphi) + p \right ) \dH,
\end{align*}
and we see the appearance of an extra boundary source term involving the normal component of the velocity and the pressure.  Here it would be advantageous to use the rescaled pressure $\hat{p}$ and the Darcy law \eqref{Darcy:type1}, as the extra boundary source term will become
\begin{align*}
\int_{\pd \Omega} - (\bm{v} \cdot \bm{n}) \hat{p} \dH,
\end{align*}
which motivates the consideration of a Robin-type boundary condition for $\hat{p}$:
\begin{align*}
g & = a \hat{p} - \bm{v} \cdot \bm{n} = a \hat{p} + K \pdnu \hat{p} - K (D \sigma + \chi (1-\varphi)) \pdnu \sigma   \text{ on } \pd \Omega,
\end{align*}
for some given datum $g$ and positive constant $a$.  On one hand, this would allow us to consider source terms $\Gamma_{\bm{v}}$ that need not have zero mean, but on the other hand, the analysis of the Darcy system becomes more complicated.  In particular, the weak formulation of the pressure system now reads as
\begin{align*}
\int_{\Omega} K \nabla \hat{p} \cdot \nabla \zeta \dx + \int_{\pd \Omega} a \hat{p} \zeta \dH & = \int_{\Omega} \Gamma_{\bm{v}} \zeta + K  \left ( \mu \nabla \varphi + \left ( D \sigma + \chi (1-\varphi) \right ) \nabla \sigma \right )\cdot \nabla \zeta \dx \\
& + \int_{\pd \Omega} g \zeta \dH,
\end{align*}
and we observe that the term $D\sigma \nabla \sigma$ on the right-hand side belongs to $\bm{L}^{1}$ as $\sigma$ has at most $H^{1}$-spatial regularity from the energy identity.  Thus, it is not clear if the pressure system can be solved with the regularities stated in Lemma \ref{lem:pressure:est}.  A deeper study into the theory of linear elliptic equations with right-hand sides of the form $\div \bm{f}$ where $\bm{f} \in \bm{L}^{1}$ is required.

\bibliographystyle{plain}
\bibliography{CHNWellposed}

\begin{thebibliography}{10}

\bibitem{book:AdamsFournier}
R.A. Adams and J.J.F. Fournier.
\newblock {\em {Sobolev spaces}}, volume 140 of {\em Pure and applied
  mathematics}.
\newblock Elsevier/Academic Press, Amsterdam, second edition, 2003.

\bibitem{book:Alt}
H.W. Alt.
\newblock {\em {Linear Functional Analysis. An Application-Oriented
  Introduction. Translated from the German edition by Robert N\"{u}rnberg}}.
\newblock Universitext. Springer Berlin London, 2016.

\bibitem{article:BosiaContiGrasselli14}
S.~Bosia, M.~Conti, and M.~Grasselli.
\newblock {On the Cahn--Hilliard--Brinkman system}.
\newblock {\em Commun. Math. Sci.}, 13(6):1541--1567, 2015.

\bibitem{book:Coddington}
E.A. Coddington and N.~Levinson.
\newblock {\em {Theory of Ordinary Differential Equations}}.
\newblock International series in pure and applied mathematics. Tata
  McGraw-Hill, New York, 1955.

\bibitem{article:ColliGilardiHilhorst15}
P.~Colli, G.~Gilardi, and D.~Hilhorst.
\newblock {On a Cahn--Hilliard type phase field model related to tumor growth}.
\newblock {\em Discrete Contin. Dyn. Syst.}, 35(6):2423--2442, 2015.

\bibitem{article:ColliGilardiRoccaSprekelsAA}
P.~Colli, G.~Gilardi, E.~Rocca, and J.~Sprekels.
\newblock {Asymptotic analyses and error estimates for a Cahn–-Hilliard type
  phase field system modelling tumor growth}.
\newblock To appear in Discrete Contin. Dyn. Syst. Ser. S. Preprint
  arXiv:1503.00927, 2015.

\bibitem{article:ColliGilardiRoccaSprekelsVV}
P.~Colli, G.~Gilardi, E.~Rocca, and J.~Sprekels.
\newblock {Vanishing viscosities and error estimate for a Cahn–-Hilliard type
  phase field system related to tumor growth}.
\newblock {\em Nonlinear Anal. Real World Appl.}, 26:93--108, 2015.

\bibitem{book:Cristini}
V.~Cristini and J.~Lowengrub.
\newblock {\em {Multiscaled modeling of cancer. An Integrated Experiemental and
  Mathematical Modeling Approach}}.
\newblock Cambridge University Press, 2010.

\bibitem{Dai}
M.~Dai, E.~Feireisl, E.~Rocca, G.~Schimperna, and M.~Schonbek.
\newblock {Analysis of a diffuse interface model for multispecies tumor
  growth}.
\newblock Preprint arXiv:1507.07683, 2015.

\bibitem{article:DellaportaGrasselli16}
F.~Della~Porta and M.~Grasselli.
\newblock {On the nonlocal Cahn--Hilliard--Brinkman and
  Cahn--Hilliard--Hele--Shaw systems}.
\newblock {\em Commun. Pure Appl. Anal.}, 15:299--317, 2016.

\bibitem{book:DiBenedetto}
E.~DiBenedetto.
\newblock {\em {Degenerate Parabolic Equations}}.
\newblock Universitext. Springer--Verlag New York, 1993.

\bibitem{book:Evans}
L.C. Evans.
\newblock {\em {Partial Differential Equations}}.
\newblock Graduate Studies in Mathematics, Volume 19. AMS, Providence, Rhode
  Island, 2002.

\bibitem{book:Fasano}
A.~Fasano, A.~Bertuzzi, and A.~Gandolfi.
\newblock {\em {Mathematical modeling of tumour growth and treatment. Complex
  Systems in Biomedicine}}.
\newblock Springer Milan, 2006.

\bibitem{article:FengWise12}
X.~Feng and S.~Wise.
\newblock {Analysis of a Darcy--Cahn--Hilliard diffuse interface model for the
  Hele--Shaw flow and its fully discrete finite element approximation}.
\newblock {\em SIAM J. Numer. Anal.}, 50(3):1320--1343, 2012.

\bibitem{book:Friedman}
A.~Friedman.
\newblock {\em {Partial Differential Equations}}.
\newblock Holt, Rinehart and Winston, New York, 1969.

\bibitem{article:FrigeriGrasselliRocca15}
S.~Frigeri, M.~Grasselli, and E.~Rocca.
\newblock On a diffuse interface model of tumor growth.
\newblock {\em European J. Appl. Math.}, 26:215--243, 2015.

\bibitem{article:GarckeLamNeumann}
H.~Garcke and K.F. Lam.
\newblock {Well-posedness of a Cahn--Hilliard system modelling tumour growth
  with chemotaxis and active transport}.
\newblock To appear in European J. Appl.
  Math.,\url{http://dx.doi.org/10.1017/S0956792516000292}. Preprint
  arXiv:1511.06143, 2015.

\bibitem{article:GarckeLamDirichlet}
H.~Garcke and K.F. Lam.
\newblock {Analysis of a Cahn--Hilliard system with non zero Dirichlet
  conditions modelling tumour growth with chemotaxis}.
\newblock Preprint arXiv:1604.00287, 2016.

\bibitem{OCTumour}
H.~Garcke, K.F. Lam, and E.~Rocca.
\newblock {Optimal control of treatment time in a diffuse interface model for
  tumor growth}.
\newblock Preprint arXiv:1608.00488, 2016.

\bibitem{article:GarckeLamSitkaStyles}
H.~Garcke, K.F. Lam, E.~Sitka, and V.~Styles.
\newblock {A Cahn--Hilliard--Darcy model for tumour growth with chemotaxis and
  active transport}.
\newblock {\em Math. Models Methods Appl. Sci.}, 26(6):1095--1148, 2016.

\bibitem{article:HawkinsZeeOden12}
A.~Hawkins-Daarud, K.G. van~der Zee, and J.T. Oden.
\newblock Numerical simulation of a thermodynamically consistent four-species
  tumor growth model.
\newblock {\em Int. J. Numer. Methods Biomed. Eng.}, 28:3--24, 2012.

\bibitem{preprint:JiangWuZheng14}
J.~Jiang, H.~Wu, and S.~Zheng.
\newblock Well-posedness and long-time behavior of a non-autonomous
  {C}ahn--{H}illiard--{D}arcy system with mass source modeling tumor growth.
\newblock {\em J. Differential Equ.}, 259(7):3032--3077, 2015.

\bibitem{article:LeeLowengrubGoodman01}
H.~Lee, J.~Lowengrub, and J.~Goodman.
\newblock {Modeling pinchoff and reconnection in a Hele--Shaw cell. I. The
  models and their calibration}.
\newblock {\em Phys. Fluids}, 14(2):492--513, 2002.

\bibitem{article:LeeLowengrubGoodman01:Part2}
H.~Lee, J.~Lowengrub, and J.~Goodman.
\newblock {Modeling pinchoff and reconnection in a Hele--Shaw cell. II.
  Analysis and simulation in the nonlinear regime}.
\newblock {\em Phys. Fluids}, 14(2):514--545, 2002.

\bibitem{article:LowengrubTitiZhao13}
J.S. Lowengrub, E.~Titi, and K.~Zhao.
\newblock Analysis of a mixture model of tumor growth.
\newblock {\em European J. Appl. Math.}, 24:691--734, 2013.

\bibitem{book:RenardyRogers}
M.~Renardy and R.C. Rogers.
\newblock {\em {An Introduction to Partial Differential Equations}}.
\newblock Texts in Applied Mathematics. Springer-Verlag New York, 2nd edition,
  2004.

\bibitem{book:Royden}
H.L. Royden and P.~Fitzpatrick.
\newblock {\em Real Analysis}.
\newblock Featured Titles for Real Analysis Series. Pearson Prentice Hall,
  Boston, 4th edition, 2010.

\bibitem{article:Simon86}
J.~Simon.
\newblock {Compact sets in space $L^{p}(0,T;B)$}.
\newblock {\em Ann. Mat. Pura Appl.}, 146(1):65--96, 1986.

\bibitem{article:WangWu}
X.~Wang and H.~Wu.
\newblock {Long-time behavior for the Hele--Shaw--Cahn--Hilliard system}.
\newblock {\em Asymptot. Anal.}, 78(4):217--245, 2012.

\bibitem{article:WangZhang}
X.~Wang and Z.~Zhang.
\newblock {Well-posedness of the Hele--Shaw--Cahn--Hilliard system}.
\newblock {\em Ann. Inst. H. Poincar\'{e} Anal. Non Lin\'{e}aire},
  30(3):367--–384, 2013.

\end{thebibliography}

\end{document}